\documentclass[11pt]{amsart}
\usepackage{esint}
\usepackage{tcolorbox}
\newtcolorbox{setting}[2][]{colback=white,colframe=c3!25!black,fonttitle=\bfseries, title=#2,#1}
\usepackage{bm}
\usepackage[margin=1.0in]{geometry}
\usepackage{tikz}

\usepackage{amsmath,amsthm,amssymb}
\numberwithin{equation}{section}
\newtheorem{prop}{Proposition}
\newtheorem{lemma}[prop]{Lemma}

\newtheorem{thm}[prop]{Theorem}
\newtheorem{cor}[prop]{Corollary}

\numberwithin{prop}{section}
\usepackage{color}
\definecolor{c1}{rgb}{0.0,0.3,0.4}
\definecolor{c2}{rgb}{0.0,0.0,0.5}
\definecolor{c3}{rgb}{0.1,0.3,0.5}
\definecolor{grn}{rgb}{0,0.4,0}
\definecolor{dgrn}{rgb}{0.0,0.3,0.0}
\definecolor{dpur}{rgb}{0.5,0.0,0.5}
\usepackage[colorlinks=true, linkcolor=c1, urlcolor=c1, citecolor=c1]{hyperref}
\newcommand{\del}{\partial}

\newcommand{\brs}[1]{\left| #1 \right|}
\newcommand{\brk}[1]{\left[ #1 \right]}
\newcommand{\prs}[1]{\left( #1 \right)}
\newcommand{\sqg}[1]{\left\{ #1 \right\}}
\newcommand{\ip}[1]{\left\langle #1 \right\rangle}
\newcommand{\gG}{\Gamma}
\renewcommand{\gg}{\gamma}
\newcommand{\fN}{\blacktriangledown}

\newcommand{\gd}{\delta}
\newcommand{\gs}{\sigma}

\newcommand{\gt}{\theta}
\newcommand{\gw}{\omega}
\newcommand{\gz}{\zeta}
\newcommand{\dVe}{d V_{\mathring{g},\eta}}
\newcommand{\ga}{\alpha}
\newcommand{\gb}{\beta}
\renewcommand{\ge}{\epsilon}
\newcommand{\N}{\triangledown}
\newcommand{\tN}{\widetilde{\triangledown}}
\newcommand{\gN}{\textcolor{gray}{\blacktriangledown}}

\newcommand{\hsp}{\hspace{0.5cm}}
\newcommand{\lap}{\vartriangle}
\newcommand{\la}{\lambda}
\newcommand{\vp}{\varphi}
\newcommand{\vs}{\varsigma}

\newcommand{\gY}{\Upsilon}

\newcommand{\w}{\wedge}

\newcommand{\ten}{\otimes}

\newcommand{\bR}{\mathbb{R}}

\DeclareMathOperator{\Rm}{Rm}

\DeclareMathOperator{\tr}{tr}
\DeclareMathOperator{\Ker}{Ker}

\DeclareMathOperator{\Id}{Id}

\DeclareMathOperator{\Vol}{Vol}

\DeclareMathOperator{\SO}{SO}

\DeclareMathOperator{\SU}{SU}
\DeclareMathOperator{\Aut}{Aut}
\DeclareMathOperator{\Grad}{Grad}
\DeclareMathOperator{\Ad}{Ad}

\DeclareMathOperator{\refc}{ref}

\newtheorem{innercustomthm}{Theorem}
\newenvironment{customthm}[1]
  {\renewcommand\theinnercustomthm{#1}\innercustomthm}
  {\endinnercustomthm}
\newtheorem{innercustomcor}{Corollary}
\newenvironment{customcor}[1]
  {\renewcommand\theinnercustomcor{#1}\innercustomcor}
  {\endinnercustomcor}
\theoremstyle{definition}
\newtheorem{defn}[prop]{Definition}

\newtheorem{rmk}[prop]{Remark}

 \title[Limits of Yang-Mills $\alpha$-connections]{Limits of Yang-Mills {\boldmath$\alpha$}-connections}
\author{Casey Lynn Kelleher}
\email{\href{mailto:clkelleh@uci.edu}{clkelleh@uci.edu}}
\address{Rowland Hall\\
         University of California\\
         Irvine, CA 92617}
\begin{document}
\begin{abstract}
In the spirit of recent work of Lamm, Malchiodi and Micallef in the setting of harmonic maps \cite{LMM}, we identify Yang-Mills connections obtained by approximations with respect to the Yang-Mills $\ga$-energy. More specifically, we show that for the $\SU(2)$ Hopf fibration over $\mathbb{S}^4$, for sufficiently small $\ga$ values the $\SO(4)$ invariant ADHM instanton is the unique $\alpha$-critical point which has Yang-Mills $\alpha$-energy lower than a specific threshold.
\end{abstract}
\maketitle
%

%
%
%
\section{Introduction}
Let $(E ,h) \to (M,g)$ be a smooth vector bundle over a closed Riemannian manifold. For a connection $\N \in W^{1,2}(\mathcal{A}_E(M))$ the \emph{Yang-Mills energy of $\N$} is
\begin{equation}
\mathcal{YM}(\N) := \tfrac{1}{2} \int_M \brs{F_{\N}}_{g,h}^2 \, dV_g,
\end{equation}
where $F_{\N}$ denotes the curvature tensor of $\N$.

Assume $M$ is $4$-dimensional. For every $\ga > 1$ and $\N \in W^{1,2\ga}\prs{\mathcal{A}_{E}\prs{M}}$, Hong, Tian and Yin \cite{HTY} introduced a perturbation of the Yang-Mills energy, the \emph{Yang-Mills $\ga$-energy}\footnote{In \cite{HTY} the energy is $\tfrac{1}{2} \int_{M} \prs{ 1+ \brs{F_{\N}}_{g,h}^2 }^{\ga} dV_g$. This alternate version retains the main intrinsic properties.},
\begin{align*}
\mathcal{YM}_{\ga}(\N) := \tfrac{1}{2} \int_M \prs{3+ \brs{F_{\N}}^2_{g,h}}^{\ga} \, dV_{g}.
\end{align*}
Unlike the Yang-Mills energy, this quantity is not conformally invariant for $\dim M = 4$. Critical points of this energy are smooth up to gauge due to work of Isobe \cite{Isobe} (stated explicitly \S 4 of \cite{HS}) and satisfy
\begin{equation}\label{eq:varYMga}
\prs{\prs{D^*_{\N} F_{\N}}_{j \gb}^{\mu} +  \tfrac{2\prs{\ga - 1}}{\prs{2 + \brs{F_{\N}}_g^2}} \prs{\N_i F_{k \ell \gz}^{\eta} } F_{k \ell \eta}^{\gz} F_{ij \gb}^{\mu} } = 0.
\end{equation}
where $\star$ denotes the Hodge star operator. Setting $\ga = 1$ one obtains the equation satisfied precisely by a \emph{Yang-Mills connection}. In \cite{HTY} the corresponding negative gradient flow of the $\ga$-energy was studied. Hong and Schabrun continued exploring the energy in \cite{HS}, verifying the Palais Smale condition and acquiring energy identities for the $\ga$-energy and flow (cf. \cite{HS} Theorem 1).

In the recent work of Lamm, Malchiodi and Micallef in \cite{LMM}, the authors studied degree $1$ maps over $\mathbb{S}^2$ and demonstrated that for sufficiently small $\ga$, the only $\ga$-harmonic maps with suitably low $\ga$-energy are rotations. In this spirit we demonstrate that the $\ga$-limiting process identifies a special element of the space of Yang-Mills connections in the following setting.

\begin{setting}{}{}
\noindent  
Henceforth $E$ denotes the adjoint bundle associated to the Hopf fibration $\SU(2) \hookrightarrow \mathbb{S}^7 \to \mathbb{S}^4$. Here the second Chern class (charge) of $E$ is $1$ thus $E$ does not admit self dual connections.
\end{setting}
The first fundamental observation is a universal lower bound on the Yang-Mills $\ga$-energy.
\begin{prop}\label{prop:W12AE}
For all $\N \in W^{1,2\ga}\prs{\mathcal{A}_{E}\prs{\mathbb{S}^4}}$,
\begin{equation}\label{eq:YMgaLB}
\mathcal{YM}_{\ga} \prs{\N} \geq 6^{\ga} \tfrac{4}{3} \pi^2.
\end{equation}
\begin{proof}
Recalling the formula for the second Chern class of $E$ (denoted $C_2 \brk{E \to \mathbb{S}^4}$) combined with the decomposition $F_{\N} = F^+_{\N} + F^-_{\N}$ (where $F^{\mp}_{\N}$ denotes the (anti)self dual pieces of $F_{\N}$) yields
\begin{align}
\begin{split}\label{eq:LMM1.7}
\tfrac{16}{3}\pi^2 &= \Vol(\mathbb{S}^4)  + \tfrac{8}{3} \pi^2 C_2[E \to \mathbb{S}^4] \\
&= \int_{\mathbb{S}^4}\, dV_{\mathring{g}}  + \tfrac{1}{3} \int_{\mathbb{S}^4}\,  \text{tr}_{h} (F_{\N} \w F_{\N}) \\
&=  \int_{\mathbb{S}^4} \prs{1+ \tfrac{1}{3} \prs{\brs{F_{\N}^-}^2_{\mathring{g}} - \brs{F^+_{\N}}^2_{\mathring{g}} }} \, dV_{\mathring{g}} \\
&\leq \int_{\mathbb{S}^4} \prs{1+ \tfrac{1}{3} \brs{F_{\N}}_{\mathring{g}}^2} \, dV_{\mathring{g}}  \qquad \qquad \qquad \qquad \qquad \qquad (\star_1)\\
& = \brs{\brs{1+ \tfrac{1}{3}\brs{F_{\N}}_{\mathring{g}}^2}}_{L^1} \\
& \leq \brs{\brs{ 1+ \tfrac{1}{3} \brs{F_{\N}}_{\mathring{g}}^2 }}_{L^{\ga}} \brs{\brs{ 1 }}_{L^{\ga^*}} \qquad \qquad \qquad \qquad (\star_2) \\
&= \tfrac{2^{1/\ga}}{3} \prs{ \tfrac{1}{2} \int_{\mathbb{S}^4} \prs{3 + \brs{F_{\N}}_{\mathring{g}}^2}^{\ga} dV_{\mathring{g}} }^{1/\ga} \prs{ \tfrac{2^3 \pi^2}{3} }^{(\ga -1) / \ga }  \\
& = \tfrac{2^{1/\ga}}{3}  \prs{ \mathcal{YM}_{\ga}(\N) }^{1/\ga} \prs{ \tfrac{2^3 \pi^2}{3} }^{(\ga -1) / \ga },
\end{split} 
\end{align}
where $\ga^*$ appearing in $\prs{\star_2}$ denotes the dual of $\ga$ in the sense of H\"{o}lder's inequality. Raising both sides of the resultant inequality of \eqref{eq:LMM1.7} to the $\ga$ power and rearranging yields the result.
\end{proof}
\end{prop}
We introduce a special class of connections, the ADHM instantons. We withhold discussion of the quaternionic coordinate system over which these are defined (cf. \cite{Naber}, \S 6.3 pp.353-361). While these are defined over $\mathbb{H}^1$, we may regard each as a connection on $\mathbb{S}^4$ via Uhlenbeck's Removable Singularities Theorem of \cite{Uhlenbeck1}.
\begin{defn}[ADHM instanton]\label{defn:ADHM}
Let $\xi \in \mathbb{H}^1$ and $\la \in \bR$. Then the \emph{$(\xi,\la)$-ADHM instanton} is the connection $\N^{\xi,\la} := \del + \gG^{\xi,\la}$ where  $\gG^{\xi,\la}$ is the $\Im \mathbb{H}^1$ valued $1$-form (with corresponding curvature)
\begin{align*}
\gG^{\xi,\la}(\zeta) &:= \Im \brk{\tfrac{\bar{\zeta} - \bar{\xi}}{\brs{\zeta-\xi}^2 + \la^2} \, d\zeta}, \quad
F^{\xi,\la}_{\N}(\zeta) = \tfrac{\la^2}{\prs{\brs{\zeta-\xi}^2 + \la^2}^2} \, d\bar{\zeta} \w d\zeta.
\end{align*}
For $\la \equiv 1$ define the \emph{basic connection} $\tN := \N^{\xi,1}$ (invariant under choice of $\xi$). Recall the following result of Atiyah, Drinfeld, Hitchin and Manin.
\end{defn}
\begin{customthm}{of \cite{ADHM}}\label{thm:LMM1.1} Every Yang-Mills energy minimizer is an ADHM instanton.
\end{customthm}
\noindent In contrast to the Yang-Mills energy, the Yang-Mills $\ga$-energy differentiates $\tN$ from other ADHM instantons. This fact provides guiding intuition for our main result, Theorem \ref{thm:LMM1.2}.
\begin{prop}
$\tN$ is the only Yang-Mills $\ga$-energy minimizer.
\begin{proof}
We compute the $\ga$-energy of $\tN$ by recalling the formulae introduced above
\begin{align*}
\mathcal{YM}_{\ga}\prs{ \tN} 
&= \tfrac{1}{2}\int_{\mathbb{S}^4} \prs{3 + \brs{F_{\tN}}_{\mathring{g}}^2 }^{\ga} dV_{\mathring{g}}\\
&= \tfrac{1}{2}\int_{\mathbb{H}^1} \prs{3 + \tfrac{1}{16} \prs{ \brs{\zeta}^2 + 1}^{4} \brs{F_{\tN}}^2 }^{\ga} 16 \prs{\brs{\zeta}^2 + 1}^{-4} dV \\
&=\tfrac{1}{2} \int_{\mathbb{H}^1} \prs{ 3 + \tfrac{1}{16} \prs{ \brs{\zeta}^2 + 1}^{4} \tfrac{48}{\prs{\brs{\zeta}^2 + 1}^4 } }^{\ga} 16 \prs{\brs{\zeta}^2 + 1}^{-4} dV \\
&= 6^{\ga} 2^3  \int_{\mathbb{S}^4} \tfrac{1}{16} \, dV_{\mathring{g}} \\
&=6^{\ga} \tfrac{4 }{3}\pi^2.
\end{align*}
Comparing against \eqref{eq:YMgaLB}, above, we see that only in the case of antiself dual connections ($\star_1$) with pointwise curvature norm precisely $3$ ($\star_2$) the inequalities are in equalities, that is, when $\N \equiv \tN$.
\end{proof}
\end{prop}
\subsection{Main results}\label{ss:mainresultoutline}
Our task is to prove the following.
\begin{customthm}{A}\label{thm:LMM1.2}
Let $E \rightarrow \prs{\mathbb{S}^4,\mathring{g}}$ be the adjoint bundle associated to the $\SU(2)$ Hopf fibration. There exists $\ge > 0$ and $\ga_0 > 1$ such that for any $\ga \leq \ga_0$ the only critical point $\N^{\ga}$ of the Yang-Mills $\ga$-energy which satisfies $\mathcal{YM}_{\ga}\prs{\N^{\ga}} \leq 6^{\ga} \tfrac{4 }{3}\pi^2 + \ge$ is the basic ADHM connection $\tN$.
\end{customthm}
\begin{rmk} This gives a beautiful analogy with the results of \cite{LMM}. Through its rotational invariance, $\tN$ corresponds to all rotations in the harmonic maps setting.
\end{rmk}
\noindent A natural corollary to Theorem \ref{thm:LMM1.2} follows directly as a consequence of long time existence and convergence of Yang-Mills $\ga$-flow presented in Theorem 1.1 of \cite{HTY}.
\begin{customcor}{A}\label{eight}\label{cor:LMM1.2flow}
Let $\sqg{\N_t^{\ga}}$ be a family of solutions to Yang-Mills $\ga$-flow satisfying $\mathcal{YM}_{\ga} \prs{\N^{\ga}_0} \leq 6^{\ga} \tfrac{4}{3} \pi^2 + \ge$ and $\ga \leq \ga_0$ as in the assumptions of Theorem \ref{thm:LMM1.2}. Then $\sqg{\N_t^{\ga}}$ converges smoothly in $\ga$ and $t$ to $\tN$.
\end{customcor}
\subsubsection{Concepts of proof}
We begin this work establishing background and discussing relevant actions of gauge transformations and conformal automorphisms in \S \ref{s:gaugeconfaut}. In \S \ref{s:closeSO(5,1)} we demonstrate that connections with sufficiently small Yang-Mills ${\ga}$-energy are close in the $W^{1,2}$ sense (up to gauge transform and conformal pull back) to $\tN$. Then we strengthen the result in \S \ref{s:closenessW2p} for such connections which are also low $\ga$-energy $\ga$-critical points by instead showing such closeness in $W^{2,p}$ for certain $p$ values necessary for later Sobolev embeddings. Next, in \S \ref{s:labd} we show that connections which, when pulled back by conformal automorphisms are close to the basic connection, are in fact contained within a compact set which depends on their Yang-Mills $\ga$-energy. In \S \ref{s:optlacloseW2p} we improve this notion of $W^{2,p}$ closeness and conclude the proof of Theorem \ref{thm:LMM1.2}.

The guiding forces for intuition come from the close ties between harmonic map and Yang-Mills theory (many of the relationships are outlined in \cite{Bourguignon, Bourguignon1}). The discrepancies between the two theories are mainly the nonlinearity of the curvature quantities and their relationship with corresponding connections, the fact that the space of connections is affine, and the delicate balance of certain gauge transformations' interaction with conformal automorphisms in this specific setting.

Unexpectedly, due to the dimensionality of our setting, we faced significant complication in a step of the proof in comparison to that of \cite{LMM}. A key step in their argument is attaining a certain $L^{\infty}$ control of a comparison between a general $\ga$-critical map and the identity map. This $L^{\infty}$ bound comes from controlling two derivatives of such a quantity in $L^2$. Unfortunately, in our setting due to dimensionality more derivatives were required to obtain the analogous type of control. Na\"{i}vely, one would attempt to estimate higher derivatives, but we were unsuccessful in a direct attempt due to the difficulty of estimating terms coming from action of the conformal automorphisms. Thus, we had to introduce an intricate strategy requiring a hole-filling technique to prove Morrey norm estimates to obtain such $L^{\infty}$ control.

As is ubiquitous throughout Yang-Mills theory, the action of the gauge group must be considered. In our setting is was natural to construct a global `comparative' gauge transformation relating to the basic connection following an argument of Tao and Tian \cite{TT} which in turn was based on the classic argument of Uhlenbeck \cite{Uhlenbeck}. Though the techniques are well known, this Coulomb-type gauge on a global scale has not yet appeared in the literature. Our gauge transformation was powerful in acquiring key estimates and naturally interacted with the action of conformal automorphisms.

%
%
%
\subsubsection{Reflections}
Naturally, one asks how this type of $\ga$-limiting process behaves in other settings. More precisely, given a sequence of $\sqg{\N^{\ga}}$ be a sequence of $\ga$-harmonic connections over a charge $\ell$ bundle $E \to \prs{M,g}$, which critical points of the Yang-Mills energy could $\sqg{\N^{\ga}}$ possibly converge to? Some cases to consider are $\SU(2)$ type bundles over the following spaces, which admit ADHM constructions (as stated in \cite{DK} pp. 127--129): $\overline{\mathbb{C P}}^2$ (charge $1$), $\mathbb{CP}^2$ (charge $2$), and $\mathbb{S}^2 \times \mathbb{S}^2$ (charge $2$). It would be interesting to know which canonical connections the Yang-Mills $\ga$-energy identifies in these various settings, especially cases which do \emph{not} have as obvious structural symmetry as the setting above. Based on our initial work as well as that of \cite{LMM}, we conjecture that in general settings the limits coming from $\ga$-approximations lie in a compact subset of the moduli space.


In \cite{LMM}, the authors construct a family of critical points of the harmonic map $\ga$-energy which are degree $1$, are high energy and are not rotations. In the Yang-Mills setting, an analogous construction is not obvious- is there an operation on connections over $\mathbb{S}^4$ that corresponds to `wrapping' maps around $\mathbb{S}^2$? 
\subsection*{Acknowledgements}
The author deeply thanks Richard Schoen for encouraging her to investigate this subject and for the many insightful and motivating conversations. She expresses sincere appreciation to Jeffrey Streets for his unwavering support. The author also gratefully thanks Princeton University, where much of the initial work was conducted, for their warm hospitality. The author deeply appreciates the insight from Mark Stern, Gang Tian, and Karen Uhlenbeck, which helped significantly with her perspective on the subject and key components of the argument. She is grateful for her encouraging conversations with Tobias Lamm about his related work.

This material is based upon work supported by a National Science Foundation Graduate Research Fellowship under Grant No. DGE-1321846. The research was supported by NSF Grant No. DMS-1440140 while the author was in residence at the Mathematical Sciences Research Institute in Berkeley, California, during the spring semester in 2016. The author was supported by a University of California President's Year Dissertation Fellowship while completing this work.
\subsection{Background concepts}
\subsubsection{Connection Sobolev spaces}\label{ss:Sobolevspaces}
%
%
%
\label{defn:SobAdE} Fix $\N_{\refc} \in C^{\infty}\prs{\mathcal{A}_E \prs{\mathbb{S}^4}}$.
 The space $W^{l,p}(\Lambda^i(\Ad E))$ is the
completion of the space of smooth sections of $\Lambda^i(\Ad E)$ with respect to
the norm
\begin{equation*}
\brs{\brs{  A }}_{W^{l,p}} := \prs{\sum_{k=0}^l
\brs{\brs{ \N^{(k)}_{\refc}  A }}_{L^p}^p }^{1/p} <
\infty, \quad A \in \Lambda^i(\Ad E).
\end{equation*}
This space is preserved (up to equivalent norms) regardless of choice of $\N_{\refc}$ though for this setting we take $\N_{\refc} \equiv \tN$.
\subsubsection{Quaternionic coordinates}\label{ss:quaternion}
Via Uhlenbeck's Removable Singularity theorem in \cite{Uhlenbeck1}, a connection $\N \in \mathcal{A}_{E} \prs{\mathbb{S}^4,\mathring{g}}$ may be identified through stereographic projection uniquely with a connection over the compactified complex $2$-plane (and thus over quaternionic space, $\mathbb{H}^1$). It is an elementary to show the Yang-Mills $\ga$-energy translates to
\begin{equation*}
\mathcal{YM}_{\ga} \prs{\N} = \int_{\mathbb{H}^1} \prs{ 3 + \tfrac{1}{16} \prs{ \brs{\zeta}^2 + 1}^{4} \brs{F_{\N}}^2 }^{\ga} 16 \prs{\brs{\zeta}^2 + 1}^{-4} \, dV.
\end{equation*}
\subsubsection{ADHM instantons}\label{s:ADHM}
The ADHM connections are instanton solutions to the Yang-Mills equations discovered by Belavin, Polyakov, Schwartz and Tyupkin in \cite{BPST}, and later constructed using linear algebraic methods via Atiyah, Drinfeld, Hitchin and Manin in \cite{ADHM}. Referring to Definition \ref{defn:ADHM}, one may derive the following formulas for the ADHM connections' curvatures (cf. \cite{Naber} pp.356-357).
\begin{prop}\label{prop:ADHMnorms} For $\xi \in \mathbb{H}^1$ and $\la \in \mathbb{R}$,
\begin{align*}
\brs{F_{\N}^{\xi,\la}(\zeta)}^2
=  \tfrac{48 \la^4}{\prs{ \brs{\zeta-\xi}^2 + \la^2 }^4}, \quad \brs{\brs{F_{\N}^{\xi,\la}}}_{L^2}^2 = 8 \pi^2.
\end{align*}
\end{prop}
\begin{rmk}
It is immediate that $\brs{F_{\tN}}_{\mathring{g}}$ is constant with respect to the spherical metric, since
\begin{equation}\label{eq:F2three}
\brs{F_{\tN} }^2_{\mathring{g}} = \tfrac{1}{16} \prs{\brs{\zeta}^2 + 1}^4 \brs{F_{\tN}}^2=3.
\end{equation}
\end{rmk}
\section{Action of gauges transformations and conformal automorphisms}\label{s:gaugeconfaut}
\subsection{Gauge group and enlargement}\label{ss:tfmEdensity}
Many of the main complications and interesting properties of Yang-Mills theory stem from the interactions of the gauge group with the connections. For four dimensional manifolds, the gauge group can be naturally extended to a larger class of objects which will be key for our upcoming analysis. We discuss the interactions of connections, conformal automorphisms and gauge transformations as well as define key quantities which identify their action on the $\ga$-energy.

\subsubsection{Gauge group}\label{ss:tfmEdensity}

Regardless of setting, the gauge group of Yang-Mills theory is both the primary difficulty to overcome and the tool to overcome it.

\begin{defn}\label{defn:gaugetfm} A \emph{gauge transformation} is a section of $\vs$ of $\mathcal{G}_E := \Aut E$. The action of a gauge transformation $\varsigma$ on a connection $\N$ is denoted by $\varsigma \brk{ \N}$ and given by
\begin{align*}
\varsigma
&: \mathcal{A}_{E}\prs{\mathbb{S}^4} \rightarrow \mathcal{A}_E\prs{\mathbb{S}^4}\\
& : \N \mapsto \varsigma\brk{\N} := \varsigma^{-1} \circ \N \circ \varsigma,
\end{align*}
Furthermore, for any $\gw \in \Lambda^{p}(\Ad E)$ for $p \in \mathbb{N} \cup \{ 0 \}$ we have the action of the gauge transformation given in coordinates by $\varsigma \brk{\gw_{i \mu}^{\gb}} := (\varsigma^{-1})^{\gb}_{\gz} \gw^{i \gz}_{\tau} \varsigma^{\tau}_{\mu}$.
\end{defn}
\begin{lemma}\label{lem:coordsN} Let $\vs \in S(\mathcal{G}_E)$ and $\N \in \mathcal{A}_E\prs{\mathbb{S}^4}$, and $\nu \in S(E)$, the coordinate expression of $\varsigma \brk{\N}$ is
\begin{equation*}
\left( \vs \brk{\N} \nu \right)_{i}^{\gb}  = \del_i \nu^{\gb} + ({\Gamma}_{\varsigma \brk{\N}})_{i \gt}^{\gb} \nu^{\gt},
\end{equation*}
where
\begin{equation}\label{eq:commutator}
(\gG_{\vs[\N]})_{i \theta}^{\gb} := (\varsigma^{-1})_{\gd}^{\gb}(\del_i \varsigma_{\gt}^{\gd}) + (\varsigma^{-1})_{\gd}^{\gb} \gG_{i \gamma}^{\gd} (\varsigma_{\gt}^{\gamma}).
\end{equation}
\end{lemma}
\begin{cor}\label{cor:Fs} Suppose $\N \in C^{\infty}\prs{\mathcal{A}_E\prs{\mathbb{S}^4}}$ and $\varsigma \in S \prs{\mathcal{G}_E }$. Then $(F_{\varsigma \brk{\N}})_{ij \mu}^{\gb} = (\vs^{-1})^{\gb}_{\gd} (F_{\N})^{\gd}_{ij \gt} \vs_{\mu}^{\gt}$.
\end{cor}

\subsubsection{Conformal automorphisms enlarging the group}\label{ss:tfmEdensity}
Given a connection, we are interested in how various pointwise and $L^2$-norms as well as the $\ga$-energy depends on actions by conformal automorphisms of the four sphere. The special indefinite orthogonal group $\SO(5,1)$ is isomorphic to the space of conformal automorphisms of the four sphere (a detailed summary of this isomorphism is outlined in pages pp.49-52 in \cite{Slovak}). 
%
Via stereographic projection every $\vp \in \SO(5,1)$ may be expressed in the form
\begin{equation*}
\varphi \prs{\zeta} = \xi_2 + \tfrac{\la \cdot \varrho \prs{\zeta-\xi_1}}{\brs{\zeta-\xi_1}^{\ge}}, \quad \xi_1,\xi_2 \in \mathbb{H}^1,
\end{equation*}
where $\varrho$ is a rotation, $\ge \in \sqg{0,2}$, and $\la > 0$ is called the \emph{dilation factor}, characterizing the magnitude of $\vp$.

The gauge group $\mathcal{G}_E$ is the group of automorphisms in each fiber, and thus projects onto the base manifold as the identity map on $M$. Since for four dimensional base manifolds, the Yang-Mills energy depends only on the conformal class of the base metric, we can enlarge the gauge group by considering automorphisms of $E$ which project onto $M$ as conformal automorphisms (cf. \cite{BL} pp.198). This is naturally called the \emph{enlarged gauge group}, where for $\vp : \mathbb{S}^4 \to \mathbb{S}^4$ a conformal automorphism,
%
\begin{align*}
\vp &: \mathcal{A}_E \prs{\mathbb{S}^4} \to \mathcal{A}_{E} \prs{\mathbb{S}^4} \\
&: \prs{ \vp^* \N}_V \prs{ \vp^* \mu} = \vp^* \prs{\N_{\vp_*(V)} \mu} \text{ for }V \in T \mathbb{S}^4,
 \end{align*}
that is, given any $\vp \in \SO(5,1)$ and $\N \in \mathcal{A}_{E}\prs{\mathbb{S}^4}$ with local expression $\N = \del + \gG$,
\begin{align}
\begin{split}\label{eq:pullbackform}
\prs{\vp^*\N} &:= \del + \vp^* \gG, \text{ where }(\vp^* \gG)_{i \mu}^{\gb} := (\del_i \vp^j) \gG_{j \mu}^{\gb}.
\end{split}
\end{align}
The curvature of $\vp^*\N$ is
\begin{align}
\begin{split}\label{eq:FvpN}
F_{{\vp}^*\N}
&= ( \del_j \vp^k) \prs{\del_i \gG_{k \mu}^{\gb}} - (\del_i \vp^k) \prs{\del_j\gG_{k \mu}^{\gb}} +  (\del_i \vp^k) (\del_j \vp^l)\gG_{k \gz}^{\gb}  \gG_{l \mu}^{\gz} - (\del_j \vp^k)(\del_i \vp^l)  \gG_{k \gz}^{\gb}   \gG_{l \mu}^{\gz}.
\end{split}
\end{align}
%
%
%
%
%
\subsubsection{A Coulomb-type projection}
We will be considering a specific type of Coulomb gauge, that is `centered' with respect to $\tN$. This particular transformation is key to our argument and naturally interacts with conformal automorphisms.
\begin{defn} A gauge transformation $\vs \in S \prs{\mathcal{G}_{E}}$ puts $\N \in \mathcal{A}_E \prs{\mathbb{S}^4}$ in \emph{(global) $\tN$-Coulomb gauge} if
\begin{equation}\label{eq:coulgauge}
D_{\tN}^* \brk{ \vs \brk{\N} - \tN }
 \equiv 0.
\end{equation}
Call $\vs$ satisfying \eqref{eq:coulgauge} the \emph{$\tN$-Coulomb gauge transformation (associated to $\N$)}. Define the \emph{$\tN$-Coulomb gauge projection} by
\begin{align*}
\widetilde{\Pi} &: \mathcal{A}_{E} \prs{\mathbb{S}^4} \to \mathcal{A}_{E} \prs{\mathbb{S}^4} \\
&: \N \mapsto \vs \brk{\N}.
\end{align*}
\end{defn}
\noindent The action of these $\tN$-Coulomb gauges commutes with that of conformal automorphisms.
\begin{prop}
For any $\vp \in \SO(5,1)$, one has $\widetilde{\Pi} \brk{ \vp^* \N} \equiv  \vp^* \widetilde{\Pi} \brk{ \N}$.
\begin{proof}This result can be confirmed by considering dilations and rotations individually since every element of $\SO(5,1)$ is the composition of these two types. The first follows immediately from natural commutation of dilations with gauge transformations (cf. \eqref{eq:commutator} and \eqref{eq:pullbackform}), and the second from the rotational invariance of $\tN$, which we describe now. By the definition of our $\tN$-Coulomb projection, and application of appropriate pullback gauge transformations, given a rotation $\varrho: \mathbb{S}^4 \to \mathbb{S}^4$,
\begin{align*}
0 &= D^*_{\tN} \brk{ \varrho^* \N - \tN } = \varrho_* D^*_{\prs{\varrho^{-1}}^*\tN} \brk{ \N- \prs{\varrho^{-1}}^*\tN } = \varrho_* D^*_{\tN} \brk{ \N- \tN }.
\end{align*}
Thus, the $\tN$-Coulomb gauge for $\varrho^* \N$ coincides with the $\tN$-Coulomb gauge for $\N$.
\end{proof}
\end{prop}
\subsubsection{Behavior of energies}
We explore the behavior of energy quantities undergoing conformal dilation and propose a modified energy which distinguishes the dilations' effects. When considering the dilation $\zeta \mapsto \la \zeta$ we abuse notation by referring to the map as $\la$. Using \eqref{eq:FvpN},
\begin{align*}
\brs{ F_{\la^*\N}(\zeta)}_{\mathring{g}}^2 &=\la^4 \tfrac{1}{16} \brs{F_{\N}(\zeta)}^2\prs{ \brs{\zeta}^2 + 1 }^4, \\
\brs{F_{\N} (\la \zeta)}^2_{\mathring{g}} &= \tfrac{1}{16} \prs{ \brs{\la \zeta}^2 + 1}^{4} \brs{F_{\N}(\la \zeta)}^2.
\end{align*}
Combining these yields
\begin{align*}
\brs{ F_{\la^*\N}(\zeta)}_{\mathring{g}}^2 = 
\la ^4  \prs{\tfrac{\brs{\zeta}^2 + 1 }{1 + \brs{\la \zeta}^2} }^4 \brs{ F_{\N}( \la \zeta)}_{\mathring{g}}^2.
\end{align*}
Then setting
\begin{equation}\label{eq:chidefn}
\chi_{\la}(\zeta) := \tfrac{1}{\la ^4}  \prs{ \tfrac{1 + \brs{\la \zeta}^2}{\brs{\zeta}^2 + 1 } }^4 ,
\end{equation}
it follows that for $\la >0$,
\begin{align*}
\mathcal{YM}_{\ga}(\N)
&= \tfrac{1}{2} \int_{\mathbb{S}^4} \prs{3 + \brs{F_{\N}(\zeta)}_{\mathring{g}}^2 }^{\ga} dV_{\mathring{g}}(\zeta) \\
&= \tfrac{1}{2} \int_{\mathbb{S}^4} \prs{3 + \brs{F_{\N}(\la \zeta)}_{\mathring{g}}^2 }^{\ga} dV_{\mathring{g}}(\la \zeta) \\
&= \tfrac{3^{\ga}}{2} \int_{\mathbb{H}^1} \prs{1 + \tfrac{1}{3}  \chi_{\la}(\zeta) \brs{F_{\la^*\N}\prs{\zeta}}_{\mathring{g}}^2}^{\ga} \tfrac{16 \la^4}{\prs{1+ \brs{\la \zeta}^2}^4} \, dV \\
&=\tfrac{3^{\ga}}{2} \int_{\mathbb{H}^1} \prs{1 + \tfrac{1}{3} \chi_{\la}(\zeta) \brs{F_{\la^*\N}\prs{\zeta}}_{\mathring{g}}^2}^{\ga} \tfrac{16}{ \prs{1+ \brs{\zeta}^2}^4 \chi_{\la}(\zeta)} \, dV.
\end{align*}
Set
\begin{equation}\label{eq:LMM2.7}
\mathcal{YM}_{\ga,\la}(\N) := \tfrac{1}{2} \int_{\mathbb{S}^4} \prs{3 + \chi_{\la} \brs{F_{\N}}_{\mathring{g}}^2}^{\ga} \tfrac{1}{\chi_{\la}} \, dV_{\mathring{g}},
\end{equation}
resulting in a natural relationship and symmetry, where for $\vp \in \SO(5,1)$ with $\brs{\vp} = \la$,
\begin{equation}\label{eq:LMM2.6}
\mathcal{YM}_{\ga}(\N) = \mathcal{YM}_{\ga, \la} \prs{\vp^*\N} = \mathcal{YM}_{\ga, \la^{-1}} \prs{\prs{\vp^{-1}}^*\N}.
\end{equation}
Thus $\N$ is a critical point of $\mathcal{YM}_{\ga}$ if and only if $\la^*\N$ is a critical point of $\mathcal{YM}_{\ga,\la}$. By utilizing this symmetry of $\mathcal{YM}_{\ga,\la}$ with respect to $\la$ about $1$, it suffices to consider $\la \geq 1$.
\begin{rmk}
Henceforth we refer to $\fN$ when discussing critical points of $\mathcal{YM}_{\ga,\la}$ to help notify the reader that these connections are $\la$-dilations of $\ga$-critical points (rather than $\ga$-critical points).
\end{rmk}
\begin{prop}\label{prop:LMM2.1}
With $\chi_{\la}$ is as in \eqref{eq:chidefn}, the gradient equation of $\mathcal{YM}_{\ga,\la}$ is
\begin{equation*}
\prs{\Grad_{\fN} \mathcal{YM}_{\ga,\la}} = D^*_{\fN} F_{\fN} + \Theta_1( \fN) + \Theta_2(\fN),
\end{equation*}
where
\begin{equation}
\begin{cases}
(\Theta_1(\fN))_{j \mu}^{\gb} &:= \frac{ 2  \chi_{\la}\prs{\ga -1}}{\prs{3 + \chi_{\la} \brs{F_{\fN}}^2_{\mathring{g}}}}  \prs{\fN_i F_{pq \gz}^{\gd}} (F_{\fN})_{pq \gd}^{\gz} (F_{\fN})_{ij \mu}^{\gb},  \\
(\Theta_2(\fN))_{j \mu}^{\gb} &:= - \frac{  \chi_{\la}\prs{\ga -1}}{\prs{3 + \chi_{\la} \brs{F_{\fN}}^2_{\mathring{g}}}} \prs{ \fN_i \log \chi_{\la}} \brs{F_{\fN}}_{\mathring{g}}^2 (F_{\fN})_{ij \mu}^{\gb}.
\end{cases} 
\end{equation}
\begin{proof} Let $\fN_t$ be a one parameter family of smooth connections, and consider
\begin{align*}
\tfrac{\del}{\del t} \brk{\mathcal{YM}_{\ga,\la}(\fN_t)}
&= \tfrac{1}{2} \int_{\mathbb{S}^4} \tfrac{\del}{\del t} \brk{\prs{3 + \chi_{\la} \brs{F_{\fN_t}}_{\mathring{g}}^2}^{\ga}} \tfrac{1}{\chi_{\la}} dV_{\mathring{g}} \\
&= \tfrac{1}{2} \int_{\mathbb{S}^4} \ga \prs{3 + \chi_{\la} \brs{F_{\fN_t}}_{\mathring{g}}^2}^{\ga - 1} \tfrac{\del}{\del t} \brk{\brs{F_{\fN_t}}_{\mathring{g}}^2} dV_{\mathring{g}} \\
&= \int_{\mathbb{S}^4} \ga \prs{3 + \chi_{\la} \brs{F_{\fN_t}}_{\mathring{g}}^2}^{\ga - 1} \ip{D_{\fN_t} \brk{\tfrac{\del \fN_t}{\del t}} ,F_{\fN_t}}_{\mathring{g}} dV_{\mathring{g}}\\
&=  2\int_{\mathbb{S}^4} \ga \prs{3 + \chi_{\la} \brs{F_{\fN_t}}_{\mathring{g}}^2}^{\ga - 1} \ip{\fN_t \brk{\tfrac{\del \fN_t}{\del t}} ,F_{\fN_t}}_{\mathring{g}} dV_{\mathring{g}}\\
& \hsp + 2 \int_{\mathbb{S}^4} \ga (\ga - 1) \prs{3 +  \chi_{\la} \brs{F_{\fN_t}}_{\mathring{g}}^2}^{\ga - 2} \fN_i \brk{ \chi_{\la} \brs{F_{\fN_t}}_{\mathring{g}}^2} \tr_h \brk{ (F_{\fN_t})_{ij}  (\tfrac{\del \fN_t}{\del t})_j }dV_{\mathring{g}}.
\end{align*}
Consequently
\begin{align*}
&\brk{ \Grad \mathcal{YM}_{\ga,\la} \prs{\fN} }_{j \gb}^{\mu} \\
&= \prs{3 + \chi_{\la} \brs{F_{\fN}}_{\mathring{g}}^2}^{\ga - 1} \brk{\prs{D_{\fN}^* F_{\fN}}_{j \gb}^{\mu} +  \tfrac{  \chi_{\la}\prs{\ga -1}}{\prs{3 + \chi_{\la} \brs{F_{\fN}}^2_{\mathring{g}}}} \prs{ 2 \tr_h \brk{\prs{\fN_i (F_{\fN})_{pq }} (F_{\fN})_{pq} }- \brs{F_{\fN}}_{\mathring{g}}^2 \prs{\fN_i \log \chi_{\la}}} (F_{\fN})_{ij \gb}^{\mu}},
\end{align*}
yielding the result.
\end{proof}
\end{prop}
\section{General closeness to $\SO(5,1)$ pullbacks}\label{s:closeSO(5,1)}
\noindent Initially we will prove a general result on connections with sufficiently controlled $\mathcal{YM}_{\ga}$ energy (not necessarily $\ga$-critical) which asserts existence of and characterizes the magnitude of a conformal automorphism required to `pull' the connection sufficiently close to $\tN$.
\begin{prop}\label{prop:LMM3.1}
There exists some $\delta_0  > 0$ so that for any $\delta \in (0,\delta_0)$, there exists some $\ge > 0$ so that if $\ga \in \brk{1,2}$ and $\mathcal{YM}_{\ga}(\N) \leq 6^{\ga} \tfrac{4 }{3}\pi^2+ \ge$, then there exists $\vp \in \SO(5,1)$ such that
\begin{equation}\label{eq:LMM3.1}
\brs{\brs{ {\widetilde{\Pi} \brk{\vp^* \N}} - {\tN} }}_{W^{1,2}} + \brs{\brs{ F_{\widetilde{\Pi} \brk{\vp^* \N}} - F_{\tN} }}_{L^2} \leq \delta,
\end{equation}
furthermore there is some fixed constant $C > 0$ such that if $\la \geq 1$ is the dilation factor of $\vp$, then
\begin{equation}\label{eq:LMM3.2}
\prs{\ga - 1} \prs{\log \la} \min \sqg{ \log \la , 1 } \leq C \delta.
\end{equation}
\begin{rmk}
The proof of this result relies on the following
\begin{itemize}
\item Lemma \ref{lem:LMM3.2}. The smallness in terms of curvature difference in \eqref{eq:LMM3.1} is shown for $\delta$ is restricted to a single value rather than a given range.
\item Theorem \ref{thm:curvcontrol}. A Poincar\'{e} inequality (Proposition \ref{prop:poincare}) is used to establish a property of $\tN$-Coulomb gauges with an eye toward concluding the complete smallness of \eqref{eq:LMM3.1}.
\item Lemma \ref{lem:LMM3.3}. The Yang-Mills $\ga$-energy of an arbitrary connection is compared to that of the basic connection $\tN$.
\item Lemma \ref{lem:LMM3.4}. The gap between $\mathcal{YM}_{\ga,\la}$ energy and $\mathcal{YM}_{\ga}$ is characterized on $\tN$ as a function of $\ga$ and $\la$.
\end{itemize}
The proof of Proposition \ref{prop:LMM3.1} tying together these results will be concluded at end of this section.
\end{rmk}
\end{prop}
This first lemma proves a result similar to Proposition \ref{prop:LMM3.1} above, though here we will fix one value of $\delta$ rather than a range of values. It requires a concentration compactness result for $\ga$-connections (Theorem \ref{thm:compactstate}) which follows quickly from results of \cite{HTY}. We state this result the appendix and supply a sketch of the proof (cf. Theorem \ref{thm:compactstate}).
\begin{lemma}\label{lem:LMM3.2} Given $\delta > 0$ there exists $\ge > 0$ sufficiently small with the following property: for all $\ga \geq 1$, if $\N \in W^{1,2\ga}(\mathcal{A}_{E}(\mathbb{S}^4))$ and $\mathcal{YM}_{\ga} \leq 6^{\ga} \tfrac{4 }{3}\pi^2+  \ge$, then there exists $\vp \in \SO(5,1)$ so that
\begin{equation}\label{eq:LMM3.3}
\brs{\brs{ F_{\widetilde{\Pi} \brk{\vp^* \N}} - F_{\tN} }}_{L^2} \leq \delta.
\end{equation}
\begin{proof}
Suppose $\mathcal{YM}_{\ga}(\N) \leq 6^{\ga} \tfrac{4 }{3}\pi^2 + \ge$. As a consequence of \eqref{eq:LMM1.7} we have that
\begin{align*}
\mathcal{YM}_1(\N)
&= \tfrac{1}{2} \int_{\mathbb{S}^4}\prs{3 + \brs{F_{\N}}_{\mathring{g}}^2} dV_{\mathring{g}}\\
& \leq \prs{ \mathcal{YM}_{\ga}(\N)  \prs{\tfrac{4 \pi^2}{3} }^{(\ga -1)}}^{1/\ga} \\
& \leq \prs{ \prs{ 1+ \tfrac{\ge}{6^{\ga} \frac{4 }{3}\pi^2}} 6^{\ga} \tfrac{4 }{3}\pi^2 \prs{\tfrac{4 \pi^2}{3} }^{(\ga -1)}}^{1/\ga} \\
& = \prs{ \prs{ 1+ \tfrac{\ge}{6^{\ga} \frac{4 }{3}\pi^2}} 6^{\ga} \prs{\tfrac{4 \pi^2}{3} }^{\ga}}^{1/\ga} \\
&\leq 8 \pi^2 + \ge.
\end{align*}
Suppose to the contrary that the statement were not true. This would imply that the Coulomb gauge may not actually exist (as in Theorem \ref{thm:curvcontrol}) and thus no gauge transformation could get the curvature `close' to that of $\tN$. More precisely, if the contrary statement holds, then there is a sequence $\ge_n \searrow 0$, and a sequence $\sqg{\N^n} \subset W^{1,2}(\mathcal{A}_{E}(\mathbb{S}^4))$ with $\mathcal{YM}_1(\N^n) \leq  6^{\ga} \tfrac{4 }{3}\pi^2+ \ge_n$, some $\delta > 0$ so that
\begin{equation}\label{eq:lowerbd}
\brs{\brs{F_{\sigma_n\brk{\vp^*_n\N^n}} - F_{\tN}}}_{L^2} > \delta \text{ for all } \sqg{\vp_n} \in \SO(5,1), \sqg{\gs_n} \in \mathcal{G}_E.
\end{equation}
For each $\sqg{\N^n}$, let $\xi \in \mathbb{S}^4$ be a point such that $\brs{F_{\N^n} \prs{\xi}} =  \sup_{\zeta \in \mathbb{S}^4} \brs{F_{\N^n}}$. For each $n$ there exists a conformal automorphism $\vp_n$ which consists of a translation of $\xi$ to the north pole $\infty \in \mathbb{S}^4$, combined with a dilation so that $\brs{F_{\vp_n^* \N^n} \prs{\infty}} = 3$. Via the Removable Singularities Theorem of \cite{Uhlenbeck1}, we can argue via a standard gauge patching argument that there exists a sequence of gauge transformations $\gs_n$ such that the sequences $\sqg{\gs_n \brk{ \vp_n^* \N^n } }$ converges to some connection $\bm{\N}$, and furthermore due to the assumed energy bounds of these quantities we know that $\bm{\N}$ is antiself dual (to see this in more detail, see our argument in Theorem \ref{thm:compactstate}). This implies, in particular that  $\bm{\N} = \psi^* \tN$. So now we have
\begin{equation*}
\prs{\psi^{-1}}^* \sigma_n \brk{ \vp^*_n \N^n} \to \tN.
\end{equation*}
We note in particular, that $\SO(5,1)$, interpreted as a subset of the enlarged gauge group $\widetilde{\mathcal{G}}_E$ acts naturally on $\mathcal{G}_E$ by conjugation, and so
\begin{equation*}
\prs{\psi^{-1}}^* \sigma_n \brk{ \vp^* \N^n}  = \prs{\psi^{-1} \gs_n \psi} \brk{ \prs{\psi^{-1} \vp}^* \N^n }.
\end{equation*}
Indeed, this is of the form of a gauge transformation applied to a pullback. Therefore we have that there exists $\sqg{\nu_n} \subset \mathcal{G}_E$ and $\sqg{\phi_n} \subset \SO(5,1)$ so that
\begin{align*}
\brs{\brs{F_{\nu_n \brk{ \phi_n^* \brk{\N^n} }} - F_{\tN}}}_{L^2} \to 0.
\end{align*}
In turn, as a result of Theorem \ref{thm:curvcontrol} we have that
\begin{equation*}
\brs{\brs{ F_{\widetilde{\Pi} \brk{ \phi_n^* \N^n }} - F_{\tN} }}_{L^2}  \to 0,
\end{equation*}
which contradictions \eqref{eq:lowerbd}, and so the result follows.
\end{proof}
\end{lemma}
Next we will show the implications of this smallness of $L^2$ in terms of curvature difference directly on the norm of connection difference itself. This is a proof in the spirit of Theorem 1.3 of \cite{Uhlenbeck1}, where the fundamental difference is that it is given on a \emph{global scale} and bounds the difference of connections in terms of the difference of curvatures. As of now, there is no such proof present in the literature. This will be highly necessary in the following section. To do so we first establish a fundamental Poincar\'{e} inequality. The proof, which is stated in the appendix (\S \ref{ss:Poincare}), relies on the following result (Lemma \ref{eq:commutatorbd}), which gives crucial control over commutator type terms which is used regularly in our arguments. The result was inspired by Lemma 2.30 of \cite{BL}, and the proof is included in the appendix (cf. \S \ref{ss:Poincare}). Combining these estimates with the geometry of the setup at play is a crucial technique used multiple times through our remaining arguments within this paper.
\begin{lemma}\label{eq:commutatorbd} Let $A \in \Lambda^1 \prs{\Ad E}$ and $B \in \prs{\Lambda^1 \prs{\Ad E}}^{\ten 2}$. Then
\begin{equation}\label{eq:Fcomm}
\ip{\widetilde{F}_{ij}, \brk{A_i, A_j}}_{\mathring{g}} \leq \brs{A}^2_{\mathring{g}}, \text{ and }
\ip{\widetilde{F}_{ij}, \brk{B_{ki}, B_{kj}}}_{\mathring{g}} \leq 4 \brs{B}^2_{\mathring{g}}.
\end{equation}
\end{lemma}
\noindent One of the many consequences of this result is the following
%
%
\begin{prop}[Global $\tN$-Poincar\'{e} inequalities]\label{prop:poincare} Given $A$ in $\Lambda^1 \prs{\Ad E}$ there exists $C_{P} > 0$ such that 
\begin{equation}\label{eq:Poincare}
\brs{ \brs{ A}}_{L^{2}}^2 \leq C_P \brs{\brs{ \tN A }}_{L^2} \text{ and } \brs{ \brs{ \tN A}}_{L^{2}}^2 \leq C_P \brs{\brs{ \tN^{(2)} A }}_{L^2}.
\end{equation}
\end{prop}
\noindent A proof of a localized version of these Poincar\'{e} inequalities is given in the appendix, with a simple gluing argument one can construct this global argument over $\mathbb{S}^4$ (a rough constant will do for our purposes). With this we characterize a norm-controlling behavior of our $\tN$-Coulomb gauge.
\begin{thm}\label{thm:curvcontrol} If $\delta \in \prs{0,1}$ is sufficiently small, then every connection $\N$ such that there exists a $\vs \in \mathcal{G}_E$ so that
\begin{equation}\label{eq:curvsmall}
\brs{\brs{F_{\vs \brk{\N}} -F_{\tN}}}_{L^2} \leq \delta
\end{equation}
in fact admits a $\tN$-Coulomb projection $\widetilde{\Pi}\brk{\N}$ which obeys the bounds
\begin{equation}\label{eq:bootstrap}
\brs{\brs{ \widetilde{\Pi}\brk{\N} - \tN  }}_{L^{2}} \leq C \brs{\brs{ F_{\widetilde{\Pi}\brk{\N}} -  F_{\tN}}}_{L^{2}}, \quad \brs{\brs{F_{\widetilde{\Pi} \brk{\N}} -F_{\tN}}}_{L^2} \leq \delta.
\end{equation} 
\end{thm}
Through the next two lemmata we will put a bound on the dilation factor $\la$ of $\vp$ in Lemma \ref{prop:LMM3.1}. To obtain it we will use the fact that due to the closeness characterized in \eqref{eq:LMM3.3} we expect that $\mathcal{YM}_{\ga,\la}(\vp^* \N)$ should be close to $\mathcal{YM}_{\ga,\la}(\tN)$. We will compute explicitly how $\mathcal{YM}_{\ga,\la} (\tN)$ grows with $\la$ (that is, how the $\ga$-energy of pullbacks of the basic connection grow).
\begin{lemma}\label{lem:LMM3.3}
If $\la \geq 1$ and $\ga \in \brk{1,2}$
\begin{equation}\label{eq:LMM3.5}
\mathcal{YM}_{\ga,\la}(\fN) - \mathcal{YM}_{\ga,\la}(\tN) \geq -\ga \prs{\tfrac{3^{\ga-1}}{2}} \prs{1 + \la^4}^{\ga - 1} \brs{ \brs{  \brs{F_{\fN}}_{\mathring{g}}^2 - 3  }}_{L^1}.
\end{equation}
\begin{proof}
As a result of the mean value theorem, there exists some positive function $f : \mathbb{S}^4 \to [0,\infty)$ whose value at $\xi \in \mathbb{S}^4$ lies between $\brs{ F_{\N} }_{\mathring{g}}^2$ and $3 \equiv \brs{F_{\tN}}_{\mathring{g}}^2$ and satisfies
\begin{equation}\label{eq:LMM3.6}
\mathcal{YM}_{\ga,\la}(\fN) - \mathcal{YM}_{\ga,\la}(\tN) = \tfrac{\ga}{2} \int_{\mathbb{S}^4} \prs{3+ \chi_{\la} f}^{\ga - 1} \prs{ \brs{F_{\fN}}_{\mathring{g}}^2 -3 } \, dV_{\mathring{g}}.
\end{equation}
Now, set
\begin{align*}
A_+ &:= \sqg{ \zeta \in \mathbb{S}^4 : \brs{F_{\fN}\prs{\zeta}}_{\mathring{g}}^2 \geq 3 } \quad A_- := \sqg{\zeta \in \mathbb{S}^4 : \brs{F_{\fN}\prs{\zeta}}_{\mathring{g}}^2 < 3 }.
\end{align*}
Then $f \geq 3$ on $A_+$, and $f \leq 3$ on $A_-$. Considering integration on these individual sets we have
\begin{equation*}
\int_{A_{\pm}} \prs{3 + \chi_{\la} f}^{\ga - 1} \prs{ \brs{F_{\fN}}_{\mathring{g}}^2 - 3 } \, dV_{\mathring{g}} \geq 3^{\ga - 1} \int_{A_{\pm}} \prs{1+\chi_{\la}}^{\ga -1 } \prs{ \brs{F_{\fN}}_{\mathring{g}}^2 -3} \, dV_{\mathring{g}}.
\end{equation*}
Consequently it follows that over the entire region
\begin{equation}\label{eq:LMM3.7}
\int_{\mathbb{S}^4} \prs{ 3 + \chi_{\la} f}^{\ga - 1} \prs{ \brs{F_{\fN}}^2_{\mathring{g}} - 3 } \, dV_{\mathring{g}} \geq 3^{\ga - 1} \int_{\mathbb{S}^4} \prs{1+ \chi_{\la}}^{\ga -1} \prs{ \brs{F_{\fN}}^2_{\mathring{g}} - 3} \, dV_{\mathring{g}}.
\end{equation}
Now since $\sup_{\mathbb{S}^4} \chi_{\la} = \la^4$,
\begin{equation}\label{eq:LMM3.8}
\brs{ \int_{\mathbb{S}^4} \prs{1+\chi_{\la}}^{\ga - 1} \prs{\ga -1} \prs{\brs{F_{\fN}}^2_{\mathring{g}} - 3} \, dV_{\mathring{g}}  } \leq \prs{1+\la^4}^{\ga - 1}\brs{\brs{ \brs{F_{\fN}}^2_{\mathring{g}} - 3}}_{L^1}.
\end{equation}
Consequently combining the estimates, \eqref{eq:LMM3.6}, \eqref{eq:LMM3.7}, \eqref{eq:LMM3.8} we obtain the result \eqref{eq:LMM3.5}.
\end{proof}
\end{lemma}
\begin{lemma}\label{lem:LMM3.4} One has
\begin{equation}\label{eq:LMM3.4a}
\mathcal{YM}_{\ga,\la}\prs{ \tN} = \mathcal{YM}_{\ga} \prs{ (\la^{-1})^* \tN} = \mathcal{YM}_{\ga} \prs{\la^* \tN}.
\end{equation}
By setting
\begin{equation*}
\mho \prs{\ga,\la} := \mathcal{YM}_{\ga} \prs{\la^* \tN}- 6^{\ga} \tfrac{4 }{3}\pi^2,
\end{equation*}
there exists a fixed constant $C > 0$ such that for $\ga \in (1,2]$,
\begin{equation}\label{eq:LMM3.11}
\mho(\ga, \la) \geq
\begin{cases}
C \la^{4 \ga-4} & \text{ if } \prs{\ga -1} \log \la \geq 5 \\
C \prs{\ga - 1} \log \la & \text{ if } \prs{\ga - 1} \leq \prs{\ga -1} \log \la \leq 5\\
C \prs{\ga - 1} \prs{\log \la}^2 & \text{ if } 0 \leq \log \la \leq 1.
\end{cases}
\end{equation}
Furthermore $\mathcal{YM}_{\ga}(\la^* \N)$ is increasing in $\la$ and for $0 \leq (\ga - 1) \log \la \leq 2$,
\begin{equation}\label{eq:LMM3.12}
\tfrac{\del}{\del \log \la} \brk{\mathcal{YM}_{\ga}(\la^*\tN)} = \tfrac{\del}{\del \log \la}\brk{\mathcal{YM}_{\ga,\la}(\tN)} \geq C \prs{\ga - 1} \tfrac{\brs{\log \la}_{\mathring{g}}}{1+ \brs{\log \la}_{\mathring{g}}}.
\end{equation}
\begin{proof} First observe that
\begin{equation}\label{eq:YMgaintegraleq}
\mathcal{YM}_{\ga,\la}(\tN) := \tfrac{3^{\ga}}{2} \int_{\mathbb{S}^4} \prs{1 + \chi_{\la} }^{\ga} \tfrac{1}{\chi_{\la}} dV_{\mathring{g}}.
\end{equation}
We will first write an alternate formula for $\mathcal{YM}_{\ga}(\la^* \tN)$ using a change of variables. First we convert \eqref{eq:YMgaintegraleq} to a radial integral; let $r := \brs{\zeta}$ and then, akin to the computations in \S \ref{ss:tfmEdensity},
\begin{align*}
\brs{ F_{\la^* \tN} }^2_{\mathring{g}} = \la^4 \tfrac{1}{16} \brs{F_{\tN} (\la \zeta)}^2 \prs{ \brs{\zeta}^2 + 1 }^4 = \tfrac{3 \la^4 \prs{\brs{\zeta}^2 + 1}^4}{\prs{1+\brs{\la \zeta}^2}^4} = \tfrac{3}{\chi_{\la} (\zeta)},
\end{align*}
and thus
\begin{equation}\label{eq:YMgalar}
\mathcal{YM}_{\ga}\prs{\la^* \tN} = \tfrac{3^{\ga}}{2} \tfrac{8\pi^2}{3} \int_0^{\infty}\prs{1+ \tfrac{ \la^4 \prs{ r^2 + 1}^4}{\prs{1+\la^2 r^2}^4} }^{\ga} \tfrac{r^3}{(1+r^2)^4} \, dr.
\end{equation}
We will next change variables, gathering up ``$(\chi_\la)^{1/4}$'' type quantities:
\begin{equation*}
\begin{cases}
w &: = \la \frac{ 1 + r^2}{1+ \la^2 r^2} \\
d w &: = 2r \la \prs{\frac{1 - \la^2}{(1+ \la^2 r^2)^2} } dr.
\end{cases}
\end{equation*}
Observing that $r^2 = \tfrac{\prs{\la - w}}{\la (\la w - 1)}$,
%
we update \eqref{eq:YMgalar} and obtain
\begin{equation}\label{eq:intwver}
\tfrac{2\pi^2 3^{\ga -1}  }{\prs{\la - \la^{-1}}^3} \int_{1/\la}^{\la} (1+w^4)^{\ga} \tfrac{( \la - w)(w - \la^{-1})}{w^4} \, dw.
\end{equation}
It is now clear that the symmetry between $\la$ and $\tfrac{1}{\la}$ is preserved since
\begin{equation*}
\mathcal{YM}_{\ga}\prs{(\la^{-1})^* \tN } = \mathcal{YM}_{\ga}\prs{\la^* \tN }.
\end{equation*}
Combining this with \eqref{eq:LMM2.6} yields \eqref{eq:LMM3.4a}. We apply the change of variables $w = e^{t}$ and $\la = e^{\tau}$ to \eqref{eq:intwver},
\begin{align*}
 \tfrac{2\pi^2 3^{\ga -1} }{\prs{e^{\tau} - e^{-\tau}}^3} &\int_{-\tau}^{\tau} (1+e^{4t})^{\ga} \tfrac{ e^t ( e^{\tau} - e^t)(e^t - e^{-\tau})}{e^{4t}} \, dt  \\
&=  \tfrac{\pi^2 3^{\ga -1} }{ 4 \prs{\sinh \tau}^3} \int_{-\tau}^{\tau} (e^{-2t}+e^{2t})^{\ga} \tfrac{ e^{2\ga t} (e^{\tau}- e^t)(e^t - e^{-\tau})}{e^{3t}} \, dt\\
& =  \tfrac{\pi^2 2^{\ga} 3^{\ga -1} }{ 4 \prs{\sinh \tau}^3} \int_{-\tau}^{\tau} (\cosh 2t)^{\ga} e^{(\ga -1)2t} \tfrac{ (e^{\tau}- e^t)(e^t - e^{-\tau})}{e^{t}} \, dt\\
& =  \tfrac{\pi^2 2^{\ga} 3^{\ga -1} }{ 4 \prs{\sinh \tau}^3} \int_{-\tau}^{\tau} (\cosh 2t)^{\ga} e^{(\ga -1)2t}  \underbrace{\tfrac{ (e^{\tau}- e^t)(e^t - e^{-\tau})}{e^{t}} } \, dt.
\end{align*}
We compute out the underbraced term
\begin{align*}
\tfrac{\prs{e^{\tau} - e^t} \prs{e^t - e^{-\tau}}}{e^t}
&= e^{-t}\prs{ e^{t + \tau} -1 - e^{2t} + e^{t-\tau}}
= 2\prs{ \cosh \tau - \cosh t}.
\end{align*}
We observe the following symmetry identity
\begin{align*}
\int_{0}^{\tau} \prs{ \cosh 2t }^{\ga} e^{(\ga - 1) 2t} \prs{ \cosh \tau - \cosh t} \, dt
&= \int_{-\tau}^{0} \prs{ \cosh 2(-t) }^{\ga} e^{ (\ga - 1) 2(-t)} \prs{ \cosh \tau - \cosh (-t)} \, dt.
\end{align*}
Applying this symmetry yields
\begin{align*}
\begin{split}
\mathcal{YM}_{\ga}(\la^* \tN)
&=  \tfrac{\pi^2 2^{\ga} 3^{\ga -1} }{ \prs{\sinh \tau}^3} \int_0^{\tau} \prs{ \cosh 2t }^{\ga} \cosh \brk{ (\ga - 1) 2t } \prs{ \cosh \tau - \cosh t} \, dt.
\end{split}
\end{align*}
The idea of the next portion of the proof is to provide lower bounds depending on $\la$ and $\alpha$ for the gap function $\mho$ measuring the difference of the ${\ga}$-energy of $\la^* \tN$ from the $\ga$-energy of $\tN$. We set
\begin{align}
\begin{split}\label{eq:LMM3.15}
\mathcal{YM}_{\ga}\prs{\prs{ e^{\tau}}^* \tN} &:=  6^{\ga} \tfrac{4 }{3}\pi^2 G(\sigma).
\end{split}
\end{align}
We will solve for $G$, and apply the change of variables
\begin{align*}
\begin{cases}
\gb &:= (\ga - 1), \\
s &:= \gb t, \\
\gs &: = \gb \tau = (\ga - 1) \log \la.
\end{cases}
\end{align*}
Now, applying the change of variables (denoted `c.o.v' below),
\begin{align}
\begin{split}\label{eq:Gdefn}
G(\sigma)&=  \tfrac{1 }{4 \prs{\sinh \tau}^3} \int_0^{ \tau} \prs{ \cosh 2t }^{\ga} \cosh \brk{ 2 (\ga - 1) t } \prs{ \cosh \tau - \cosh t} \, dt \\
&\overset{c.o.v}{\mapsto}  \tfrac{1 }{4 \gb \prs{\sinh \tfrac{\gs}{\gb}}^3} \int_0^{\gs} \prs{ \cosh \tfrac{2s}{\gb}}^{\gb+1} \prs{\cosh 2s}  \prs{ \cosh \tfrac{\sigma}{\gb} - \cosh \tfrac{s}{\gb}}  \, ds.
\end{split}
\end{align}
With this new formulation we address the various cases of \eqref{eq:LMM3.11}.\\ \\
\noindent \fbox{$\prs{\ga -1} \log \la \geq 5 $} Noting that
\begin{align*}
\tfrac{\del}{\del s} \brk{ \prs{\sinh \tfrac{s}{\gb} }^3 }
&= \tfrac{3}{\gb} \prs{\sinh \tfrac{s}{\gb}}^2  \cosh \tfrac{s}{\gb},
\end{align*}
we rewrite \eqref{eq:Gdefn} (and decrease the region of integration) as follows
\begin{align*}
G(\gs) &\geq \tfrac{1}{12 \prs{\sinh \frac{\gs}{\gb}}^3} \int_{\gs - 2}^{\gs-1} \tfrac{\del}{\del s}\brk{\prs{\sinh \tfrac{s}{\gb}}^3} \underbrace{ \prs{\cosh \tfrac{s}{\gb}}^{-1} \prs{\sinh \tfrac{s}{\gb}}^{-2} \prs{\cosh \tfrac{2s}{ \gb}}^{\gb+1} \cosh 2s \prs{ \cosh \tfrac{\gs}{\gb} - \cosh \tfrac{s}{\gb}} }ds.
\end{align*}
We estimate the underbraced part, which we call $Q(s)$ on the interval $\brk{\gs - 1, \gs}$,
\begin{align*}
Q(s)
&= \cosh 2s \prs{\sinh \tfrac{s}{\gb}}^{-2} \prs{\cosh \tfrac{2s}{ \gb}}^{\gb + 1} \prs{ \tfrac{\cosh \frac{\gs}{\gb}}{\cosh \frac{s}{\gb}} - 1} \\
&= \brk{\cosh 2 s}_{T_1} \brk{ \cosh \tfrac{2s}{ \gb} \prs{\sinh \tfrac{s}{\gb}}^{-2}}_{T_2} \brk{ \prs{\cosh \tfrac{2s}{ \gb} }^{\gb}}_{T_3}  \brk{\prs{ \tfrac{\cosh \frac{\gs}{\gb}}{\cosh \frac{s}{\gb}} - 1 }}_{T_4}.
\end{align*}
For the $T_1$, we estimate that
\begin{align*}
T_1 &\geq \left[ \tfrac{e^{2s}}{2} \right|_{\gs-2}^{\gs-1}   \geq \tfrac{e^{2(\gs - 2)}}{2} = \tfrac{e^{2\gs}}{2e^4}.
\end{align*}
For the $T_2$, using the hyperbolic sine additive angle identity,
\begin{align*}
T_2 &\geq \prs{\sinh \tfrac{2s}{\gb}} \prs{\sinh \tfrac{s}{\gb}}^{-2} = 2.
\end{align*}
For $T_3$ we have that
\begin{align*}
T_3
&\geq \prs{ \tfrac{e^{\frac{2s}{\gb}}}{ 2} }^{\gb} = \left[ \tfrac{e^{2s }}{2^{\gb}} \right|_{\gs-2}^{\gs-1}= \tfrac{e^{2(\gs-2) + 2\frac{(\gs-2)}{\gb}}}{2^{\gb}} \geq \tfrac{e^{2(\gs-2)} }{2^{\gb}} = \tfrac{e^{2\gs} }{2^{\gb} e^4}.
\end{align*}
For $T_4$, we estimate
\begin{align*}
\tfrac{\cosh \frac{\gs}{\gb}}{\cosh \frac{s}{\gb} } &\geq \tfrac{e^{\frac{\gs}{\gb}} + e^{-\frac{\gs}{\gb}}}{e^{\frac{\gs-1}{\gb}} + e^{-\frac{\gs-1}{\gb}}} \geq  \tfrac{e^{\frac{\gs}{\gb}} + e^{-\frac{\gs}{\gb}}}{2 e^{\frac{\gs-1}{\gb}}}  \geq \tfrac{1}{2} \prs{e^{\frac{1}{\gb}} + e^{\frac{- 2\gb + 1}{\gb}} } \geq \tfrac{1}{2} e^{\frac{1}{\gb}}  \geq \tfrac{e}{2},
\end{align*}
and so we conclude that
\begin{align*}
\left. Q(s) \right|_{\gs - 2}^{\gs - 1} =\left. \Pi_{i=1}^4 T_i \right|_{\gs - 2}^{\gs - 1} 
 = \tfrac{e^{4\gs}}{2^{\ga} e^8} \prs{ e - 2}.
\end{align*}
Now, we compute, noting that for $x > y > 0$, we have $x^3-y^3 \geq \prs{x-y}^3$, giving
\begin{align*}
\tfrac{1}{\prs{\sinh \tfrac{\gs}{\gb}}^3} \int_{\gs-2}^{\gs-1} \tfrac{\del}{\del s} \brk{ \prs{\sinh \tfrac{s}{\gb}}^3 } \, ds 
&=\tfrac{1}{\prs{\sinh \tfrac{\gs}{\gb}}^3}  \prs{\prs{\sinh \tfrac{(\gs - 1)}{\gb}}^3 - \prs{\sinh \tfrac{(\gs - 2)}{\gb}}^3} \\
& \geq \tfrac{1}{\prs{\sinh \frac{\gs}{\gb}}^3}  \prs{\sinh \tfrac{(\gs - 1)}{\gb} - \sinh \tfrac{(\gs - 2)}{\gb}}^3 \\
&= \prs{ \tfrac{e^{\frac{\gs - 1}{\gb}} - e^{-\frac{\gs - 1}{\gb}} - e^{\frac{\gs - 2}{\gb}} + e^{-\frac{\gs - 2}{\gb}}}{e^{\frac{\gs}{\gb}}-e^{- \frac{\gs}{\gb}}}  }^3\\
&=  \prs{ e^{\frac{- 1}{\gb}}  - e^{\frac{- 2}{\gb}}   }^3 \\
&\geq \tfrac{\prs{e-1}^3}{e^6}.
\end{align*}
Note that the last line follows from the monotonicity of $f(x) = \frac{x-1}{x^2}$.
Now we combine everything together,
\begin{align*}
G(\gs) \geq e^{4\gs} \prs{\tfrac{\prs{ e - 2}\prs{e-1}^3}{2^{\gb+3} 3 e^{14}} } \geq e^{4\gs} \prs{\tfrac{\prs{ e - 2}\prs{e-1}^3}{2^{4} 3 e^{14}} }.
\end{align*}
Taking $\gs \geq 5$, we conclude that
\begin{equation*}
G(\gs) - 1 \geq e^{4\gs} \prs{\tfrac{\prs{ e - 2}\prs{e-1}^3}{2^{4} 3 e^{14}} - \tfrac{1}{e^{4 \gs}}} \geq 0,
\end{equation*}
which concludes the first estimate.\\

Now we make some necessary preparations for the remaining two cases.
We start with
\begin{align*}
G(\gs) := \tfrac{1 }{4 \gb \prs{\sinh \frac{\gs}{\gb}}^3} \int_0^{ \gs} \prs{ \cosh \tfrac{2s}{\gb}}^{\gb+1} \prs{\cosh 2s}  \prs{ \cosh \tfrac{\sigma}{\gb} - \cosh \tfrac{s}{\gb}} \, ds.
\end{align*}
We differentiate, and obtain
%
\begin{align*}
G' \prs{\gs}  = \tfrac{1}{4 \gb^2 \prs{\sinh \frac{\gs}{\gb}}^4} \int_0^{\gs}\prs{\cosh \tfrac{2s}{\gb}}^{\gb + 1} \prs{\cosh 2s} \prs{3 \prs{\cosh \tfrac{\gs}{\gb} \cosh \tfrac{s}{\gb} -\cosh^2 \tfrac{\gs}{\gb}} + \sinh^2 \tfrac{\gs}{\gb}  } \, ds.
\end{align*}
Now we set $g(s,\gs) = \prs{ \cosh \tfrac{2 s}{\gb} }^{\gb} \prs{\cosh 2s}$ and observe that
\begin{align*}
\tfrac{\del g}{\del s}(s,\gs)
&= 2 \prs{\cosh \tfrac{2s}{\gb}}^{\gb-1} \prs{\sinh \tfrac{2s\ga}{\gb}}.
\end{align*}
Next, set $h(s)$ to be the function
%
\begin{align*}
h(s) = - \gb \sinh \tfrac{s}{\gb} \prs{\cosh \tfrac{\gs}{\gb} - \cosh \tfrac{s}{\gb}} \prs{2 \cosh \tfrac{\gs}{\gb} \cosh \tfrac{s}{\gb} - 1},
\end{align*}
which is negative and satisfies $h(0) = h(\gs) = 0$. Differentiating this, we obtain the familiar quantity
\begin{align*} 
\tfrac{\del h}{\del s}(s) :=  \prs{\cosh \tfrac{2s}{\gb}} \prs{3 \prs{\cosh \tfrac{\gs}{\gb} \cosh \tfrac{s}{\gb} - \cosh^2 \tfrac{\gs}{\gb} } + \sinh^2 \tfrac{\gs}{\gb}}.
\end{align*}
From this we can observe the following identity and apply an appropriate integration by parts,
\begin{align}
\begin{split}\label{eq:G'form}
G'(\gs) &= \tfrac{1}{4 \gb^2 \prs{\sinh \tfrac{\gs}{\gb}}^4} \int_0^{\gs} g(s,\gs) \prs{\tfrac{\del h(s)}{\del s}} \, ds \\
& = - \tfrac{1}{4 \gb^2 \prs{\sinh \tfrac{\gs}{\gb}}^4} \int_0^{\gs} \prs{\tfrac{\del g(s,\gs)}{\del s}} h(s)  \, ds\\
&=  \tfrac{1}{2 \gb \prs{\sinh \tfrac{\gs}{\gb}}^4} \int_0^{\gs} \prs{\cosh \tfrac{2s}{\gb}}^{\gb-1} \prs{\sinh \tfrac{2s\ga}{\gb}} \sinh \tfrac{s}{\gb} \prs{\cosh \tfrac{\gs}{\gb} - \cosh \tfrac{s}{\gb}} \prs{2 \cosh \tfrac{\gs}{\gb} \cosh \tfrac{s}{\gb} - 1} \, ds.
\end{split}
\end{align}
We estimate $G'$ from below in the two different cases.\\ \\%
\noindent \fbox{$\prs{\ga - 1} \leq \prs{\ga -1} \log \la \leq 5$} This is equivalent to considering $0<\gb \leq \gs \leq 5$. For this, we show that $G'$ is bounded below by a positive constant which is independent of $\gb$. To do so we apply the following:
\begin{equation*}
\tfrac{\cosh \frac{\gs}{\gb}}{\sinh \frac{\gs}{\gb}} > 1, \quad \tfrac{\sinh \frac{\ga s}{\gb}}{\cosh \frac{s}{\gb}} \geq \tanh \tfrac{\ga s}{\gb}.
\end{equation*}
Then for $\theta \in (0,1)$ and $\gb \leq \gs$, utilizing further the fact that $\tanh \ga \theta \geq \tanh \theta$ and $\cosh \tfrac{s}{\gb} \geq 1$,
\begin{align*}
G'(\gs) &=  \tfrac{1}{2 \gb \prs{\sinh \tfrac{\gs}{\gb}}^4} \int_0^{\gs} \prs{\cosh \tfrac{2s}{\gb}}^{\gb}\prs{ \tfrac{\sinh \frac{2s\ga}{\gb}}{\cosh \frac{2s}{\gb}}}  \prs{\gb \sinh \tfrac{s}{\gb}} \prs{\cosh \tfrac{\gs}{\gb} - \cosh \tfrac{s}{\gb}} \prs{2 \cosh \tfrac{\gs}{\gb} \cosh \tfrac{s}{\gb} - 1} \, ds \\
&\geq \tfrac{1}{2 \gb \prs{\sinh \tfrac{\gs}{\gb}}^4} \int_{\gb \theta}^{\gs} \prs{\cosh \tfrac{2s}{\gb}}^{\gb} \prs{\tanh \tfrac{2\ga s}{\gb}} \prs{\gb \sinh \tfrac{s}{\gb}} \prs{\cosh \tfrac{\gs}{\gb} - \cosh \tfrac{s}{\gb}} \prs{\cosh \tfrac{\gs}{\gb} \cosh \tfrac{s}{\gb}} \, ds \\
&\geq \tfrac{ \tanh 2\ga \theta }{2 \prs{\sinh \tfrac{\gs}{\gb}}^3} \int_{\gb \theta}^{\gs} \prs{\tfrac{1}{\gb} \sinh \tfrac{s}{\gb}} \prs{\cosh \tfrac{\gs}{\gb} - \cosh \tfrac{s}{\gb}}\cosh \tfrac{s}{\gb} \, ds \\
&\geq \tfrac{ \tanh 2\ga \theta }{2 \prs{\sinh \tfrac{\gs}{\gb}}^3}\prs{ \cosh \tfrac{\gs}{\gb} \int_{\gb \theta}^{\gs} \tfrac{1}{\gb} \sinh \tfrac{s}{\gb} \cosh \tfrac{s}{\gb}\, ds -  \int_{\gb \theta}^{\gs} \tfrac{1}{\gb} \sinh \tfrac{s}{\gb} \cosh^2 \tfrac{s}{\gb} \, ds}\\
&= \tfrac{\tanh 2 \ga \theta}{2 \prs{\sinh \frac{\gs}{\gb}}^3} \prs{\tfrac{1}{6} \cosh^3 \tfrac{\gs}{\gb} - \tfrac{1}{2} \cosh \tfrac{\gs}{\gb} \cosh^2 \theta + \tfrac{1}{3} \cosh^3 \theta} \\
&\geq \tfrac{\tanh 2 \theta}{2 \prs{\sinh \frac{\gs}{\gb}}^3} \prs{\tfrac{1}{6} \cosh \tfrac{\gs}{\gb} \prs{\cosh^2 \tfrac{\gs}{\gb}  - 3 \cosh^2 \theta}}.
\end{align*}
Choose $\theta$ so that $\cosh \theta \leq \tfrac{1}{\sqrt{6}} \cosh 1$, then we can conclude that there is some constant $C>0$ independent of anything such that if $\ga > 1$ and $\la \geq e$, that is, $\tau \geq 1$ and $0 < \gb \leq \gs$ then
\begin{equation*}
G'(\gs) \geq \tfrac{\tanh 2 \theta}{24 \prs{\sinh \tfrac{\gs}{\gb}}}^3 \prs{\cosh \frac{\gs}{\gb}}^3 \geq C > 0.
\end{equation*}
Therefore, we have that for $0 < \gb \leq \gs$ we have
\begin{equation}\label{eq:MVTthing}
G(\gs) \geq G \prs{\gb} + C \prs{\gs - \gb}.
\end{equation}
\noindent \fbox{$0 < \prs{\ga-1} \log\la \leq \prs{\ga-1} \leq 1$} 
Next let's consider the lower bound for $G'$ in this setting, which is equivalent to $0 < \gs \leq \gb \leq 1$. Beginning again from the inequality \eqref{eq:G'form}, we apply the identities
\begin{align*}
\prs{\cosh \tfrac{s}{\gb}}^{\gb} \leq 1, \quad  \tfrac{\sin \frac{2s \ga}{\gb}}{\cosh \frac{2s}{\gb}} \geq \tanh \tfrac{2s}{\gb}, \quad \prs{2 \cosh \tfrac{\gs}{\gb} \cosh \tfrac{s}{\gb} -1} \geq \cosh \tfrac{\gs}{\gb} \cosh \tfrac{s}{\gb} .
\end{align*}
Applying them in, we have
\begin{align}
\begin{split}\label{eq:Gprime1}
G'(\gs) 
&\geq  \tfrac{1}{2 \gb \prs{\sinh \tfrac{\gs}{\gb}}^4} \int_0^{\gs} \prs{\tanh \tfrac{2s\ga}{\gb}} \sinh \tfrac{s}{\gb} \prs{\cosh \tfrac{\gs}{\gb} - \cosh \tfrac{s}{\gb}} \cosh \tfrac{\gs}{\gb} \cosh \tfrac{s}{\gb} \, ds.
\end{split}
\end{align}
Then we apply three more identities:
\begin{align*}
\sinh x \geq x, \quad \cosh x \geq 1 + \tfrac{x^2}{2}, \quad \tanh x \geq \tfrac{x}{(\cosh 2)^2} \text{ for } x \in \brk{0,2}, \quad \sinh x \leq x \cosh x,
\end{align*}
and update \eqref{eq:Gprime1} to obtain
\begin{align}
\begin{split}\label{eq:Gprime2}
G'(\gs) &\geq \tfrac{\gb^4}{2 \gb \gs^4 \prs{\cosh \tfrac{\gs}{\gb}}^4 } \int_0^{\gs} \tfrac{1}{(\cosh 2)^2} \tfrac{2 s}{\gb} \tfrac{s}{\gb} \prs{1+\tfrac{\gs^2}{2\gb^2}} \prs{\cosh \tfrac{\gs}{\gb} - \cosh \tfrac{s}{\gb}} \, ds \\
&\geq \tfrac{\gb}{2 \gs^4 \prs{\cosh 1}^4(\cosh 2)^2 } \int_0^{\gs} \tfrac{1}{(\cosh 2)^2} s^2 \prs{1+\tfrac{\gs^2}{2\gb^2}} \prs{\cosh \tfrac{\gs}{\gb} - \cosh \tfrac{s}{\gb}} \, ds.
\end{split}
\end{align}
Lastly, we implement an identity from standard ODE theory that
\begin{align*}
\prs{\cosh x - \cosh y} \geq \tfrac{1}{2} \prs{x-y}^2.
\end{align*}
%
Applying this to \eqref{eq:Gprime2} we can finally conclude that
\begin{align}
\begin{split}\label{eq:LMM3.19}
G'(\gs) &\geq \tfrac{1}{2 \gb\gs^4 \prs{\cosh 1}^4(\cosh 2)^2 }  \prs{1+\tfrac{\gs^2}{2\gb^2}}  \int_0^{\gs}  s^2( \gs - s)^2 \, ds \\
&\geq \tfrac{\gs}{60 \prs{\cosh 1}^4(\cosh 2)^2 \gb}  \prs{1+\tfrac{\gs^2}{2\gb^2}} \\
&\geq \tfrac{\gs}{60 \prs{\cosh 1}^4(\cosh 2)^2 \gb},
\end{split}
\end{align}
so for $0 \leq \gs \leq \gb$,
\begin{equation}\label{eq:LMM3.20}
G(\gs) - G(0) \geq \tfrac{\gs^2}{60 \prs{\cosh 1}^4(\cosh 2)^2 \gb} \geq \tfrac{\prs{\ga-1}\prs{\log \la}^2}{60 \prs{\cosh 1}^4(\cosh 2)^2}.
\end{equation}
Now we can establish the last two estimates. In the event that $\prs{\ga - 1} \leq \prs{\ga - 1} \log \la \leq 5$, then by rescaling \eqref{eq:MVTthing} and applying the identities of $\gb$ and $\gs$ we have
\begin{align*}
\mho \prs{\ga,\la} \geq 6^{\ga} \tfrac{4}{3} \pi^2 \prs{\prs{G(\ga-1)-1} + C \prs{\ga -1} \prs{\log \la -1}} \geq C \prs{\ga - 1}\log\la.
\end{align*}
If $\log \la \leq 1$ (the latter case) we again obtain from rescaling from \eqref{eq:LMM3.20} that
\begin{align*}
\mho \prs{\ga,\la} \geq  6^{\ga} \tfrac{2^5}{15} \pi^2 \prs{\ga - 1}\prs{\log \la}^2.
\end{align*}
Now we prove the last two claims. We know that $\mathcal{YM}_{\ga,\la}$ is monotonically increasing with respect to $\la$ since $G'$ is positive (cf. \eqref{eq:G'form}). To show the final expression, \eqref{eq:LMM3.12}, we note that from \eqref{eq:LMM3.15} that
\begin{align*}
\tfrac{\del}{\del \log \la}\brk{ \mathcal{YM}_{\ga,\la} \prs{\tN}}  \geq C \prs{\ga-1} \geq C \prs{\ga -1 } \tfrac{\brs{\log \la}}{1+ \brs{\log \la}}.
\end{align*}
For $0<\log \la \leq 1$ we use \eqref{eq:LMM3.19} to conclude
\begin{align*}
\tfrac{\del}{\del \log \la}  \brk{ \mathcal{YM}_{\ga,\la} \prs{\tN}} \geq C \prs{\ga-1} \log \la \geq C \prs{\ga - 1} \tfrac{\brs{\log \la}}{1+ \brs{\log \la}}.
\end{align*}
This concludes the argument.
\end{proof}
\end{lemma}
\begin{proof}[Proof of Proposition \ref{prop:LMM3.1}.] Since we have Lemma \ref{lem:LMM3.2} and thus appropriate control over connection and curvature differences, it remains to prove \eqref{eq:LMM3.2}, the resultant control over $\la$ (though dependent on $\ga$). To do so we will apply Lemma \ref{lem:LMM3.3} to $\vp^* \N$, where $\vp$ a conformal automorphism which enforces curvature difference smallness in the sense of \eqref{eq:LMM3.3}, and $\la \geq 1$ its corresponding dilation factor. Obtaining the inequality
\begin{align*}
\brs{\brs{ \brs{F_{\vp^*\N}}_{\mathring{g}}^2 - 3 }}_{L^1}
&\leq \brs{\brs{  F_{\vp^* \N} - F_{\tN} }}_{L^2} \brs{\brs{ F_{\vp^* \N} + F_{\tN} }}_{L^2} \leq \delta \sqrt{(\tfrac{16}{3} \pi^{2} + \ge)(\tfrac{16}{3} \pi^{2})}  \leq \delta (\tfrac{16}{3}) \pi^2,
\end{align*}
we apply this to \eqref{eq:LMM3.5} to obtain
\begin{equation}\label{eq:LMM3.21}
6^{\ga} \tfrac{4}{3}  \pi^{2} + \ge \geq \mathcal{YM}_{\ga}(\N) = \mathcal{YM}_{\ga,\la}(\vp^* \N) \geq \mathcal{YM}_{\ga,\la} \prs{\tN} - \ga 6^{\ga} \tfrac{4}{3}  \pi^{2} \la^{4( \ga -1)} \delta.
\end{equation}
Additionally, from Lemma \ref{lem:LMM3.4},
\begin{equation}
\mathcal{YM}_{\ga,\la}\prs{\tN} = \mathcal{YM}_{\ga}(\la^* \tN) =  \tfrac{16}{3} \pi^{2} + \mho(\ga,\la),
\end{equation}
and note that $\ge$ in Lemma \ref{lem:LMM3.2} is always bounded above by $\delta$. Then we rewrite \eqref{eq:LMM3.21} as,
\begin{equation}\label{eq:LMM3.22}
\delta \prs{1 + C' \la^{4 \ga-4} } \geq \mho(\ga,\la) \text{ for some }C' \in \bR.
\end{equation}
Now we need to consider the regions mentioned in Lemma \ref{lem:LMM3.4}. In the case that $\prs{\ga - 1} \log \la \geq 5$, that is (multiplying both sides by $5$ and exponentiating) $\la^{5(\ga - 1)} \geq e^{10}$ (i.e. $\la^{4 \prs{\ga - 1}} \geq e^8$). Then by \eqref{eq:LMM3.11} we have deduced the lower bound $\mho(\ga,\la) \geq  C\la^{4(\ga - 1)}$, which implies that \eqref{eq:LMM3.22} is false when $0 \leq \delta < \delta_0 := \min \sqg{\frac{C}{2C'}, \frac{C}{2} e^{8} }$. Therefore $\la^{4(\ga - 1)} < e^{8}$ and thus, combining the latter two cases of \eqref{eq:LMM3.11} and \eqref{eq:LMM3.22}, we have
\begin{align*}
\delta \prs{1 + C' e^{8}} \geq C \prs{\ga - 1} \prs{\log \la }\min \sqg{ \log \la, 1 },
\end{align*}
yielding the result.
\end{proof}
\section{Closeness of $\ga$-connections in the $W^{2,p}$-norm}\label{s:closenessW2p}
Now we prove a refinement of Proposition \ref{eq:LMM3.1} which demonstrates closeness between $\vp^* \N$ and $\tN$ in $W^{2,p}$ for $p \in (2, \tfrac{12}{5}]$, with the further restriction that $\N$ is a Yang-Mills $\ga$-connection. The determination of this range will be clarified in Proposition \ref{prop:LMM5.1} (the fundamental reason being to apply necessary Sobolev embeddings). At this point in the proof we reach a key fundamental difference between the arguments of \cite{LMM} and ours. Since we are working in the four rather than two dimensional setting, our Sobolev embeddings are not as favorable in that we require more degrees of differentiability. To address this we introduce the notion of Morrey space of maps.
\begin{defn}Take $\Omega \subset \mathbb{S}^4$ and set $\Omega \prs{\zeta_0, \rho} : = \Omega \cap B_{\rho} \prs{\zeta_0}$ and for every $p \in \brk{1,+\infty}$, $\la \geq 0$ set,
\begin{equation*}
\mathcal{M}^{p}_{\la} \prs{\Omega} := \sqg{ u \in L^p \prs{\Omega}:  \sup_{\underset{\rho>0}{\zeta_0 \in \Omega}} \rho^{-\la} \int_{\Omega \prs{\zeta_0,\rho}} \brs{u}^p \, dV_{\mathring{g}}< \infty}.
\end{equation*}
with associated norm defined by
\begin{align*}
\brs{\brs{u}}_{\mathcal{M}^{p}_{\la}} := \prs{\sup_{\underset{\rho>0}{\zeta_0 \in \Omega}} \rho^{-\la} \int_{\Omega\prs{\zeta_0,\rho}} \brs{u}^p \, dV_{\mathring{g}}}^{1/p} < \infty.
\end{align*}
\end{defn}
\noindent Using this space we intend to use the following result of Morrey.
\begin{customthm}{1.1 of \cite{Giaquinta}}\label{thm:MLemma}
Assume $p \geq n$ and let $u \in W^{1,p}_{\text{loc}}$, $D u \in \mathcal{M}^{p,n-p+\varepsilon}_{\text{loc}}$ for some $\varepsilon > 0$. Then $u \in C_{\text{loc}}^{0,\frac{\varepsilon}{p}}$.
\end{customthm}
\begin{rmk}[Notational conventions]\label{rmk:notation} For the proofs of this following proposition as well as that of Lemma \ref{lem:MorreyUp}, we will indicate applications of either Sobolev embeddings or Poincar\'{e} inequality with a subscript $S$ or $P$ on constants. This is nonstandard but will help the reader follow manipulations.
\end{rmk}
\begin{prop}\label{prop:LMM4.1} There exists $\ga_0 > 1$, $\delta_0, C = C\prs{\ga_0,\delta_0} >0$ depending on only $\ga_0$ and $\delta_0$ such that for every $\ga \in (1,\ga_0]$, every $\delta \in (0,\delta_0]$ and for every critical point $\fN \in W^{1,2 \ga}(\mathcal{A}_{E}(\mathbb{S}^4))$ of $\mathcal{YM}_{\ga,\la}$ satisfying \eqref{eq:LMM3.1} and \eqref{eq:LMM3.2} then for $p \in \left( 2, \tfrac{12}{5} \right]$ sufficiently small,
\begin{equation}\label{eq:LMM4.1}
\brs{ \brs{ \widetilde{\Pi} \brk{\fN} - \tN }}_{L^{\infty}} + \brs{\brs{ F_{\widetilde{\Pi} \brk{\fN}} - F_{\tN}}}_{W^{1,p}} \leq C \prs{\delta + \prs{\ga -1}}.
\end{equation}
\begin{proof} We summarize the proof. Via Proposition \ref{prop:LMM3.1} we have that the difference of curvatures is small, i.e. \eqref{eq:curvsmall}. For notational convenience set $\N \equiv \widetilde{\Pi} \brk{\N}$. Initially, we demonstrate
\begin{equation}\label{eq:sobconndif}
\brs{\brs{\fN- \tN}}_{W^{2,2}} + \brs{\brs{F_{\fN}- F_{\tN}}}_{W^{1,2}} < C \prs{\delta + \prs{\ga - 1}}.
\end{equation}
Using this, we will apply a hole-filling argument to obtain the necessary Morrey type estimates, and apply appropriate embeddings to obtain the main result. We will first conclude \eqref{eq:sobconndif} by obtaining bounds on the first derivative of $F_{\fN} - F_{\tN}$, and second derivative of $\fN-\tN$ (in fact, the curvature bounds follow from the polarization of curvature). Now, recall that the Yang-Mills $\alpha$-energy is also preserved under gauge transformation, so using Proposition \ref{prop:LMM2.1}, since $\fN$ is a critical point of $\mathcal{YM}_{\ga,\la}$ and $\tN$ is a critical point of $\mathcal{YM}$ we take the difference of the corresponding equations,
\begin{equation}\label{eq:LMM2.1}
0 =\brk{D^*_{\fN} F_{\fN} - D^*_{\tN} F_{\tN}}  +  \Theta_1(\fN) + \Theta_2(\fN).
\end{equation}
Our goal is to first estimate $\gY := \fN - \tN$ in the $W^{2,2}$ sense. To do so we first identify key pointwise estimates on $\Theta_i$ for $i \in \sqg{1,2}$.

\begin{align}
\begin{split}\label{eq:Theta1approx}
\brs{\Theta_1\prs{\fN} }_{\mathring{g}}& \leq 4 \chi_{\la} \prs{\ga - 1}\tfrac{\brs{\fN F_{\fN}}_{\mathring{g}} \brs{F_{\fN}}_{\mathring{g}}^2}{\brs{3 + \chi_{\la} \brs{F_{\fN}}^2_{\mathring{g}}}}\\ & \leq 4 \chi_{\la} \prs{\ga - 1}\tfrac{\brs{\fN F_{\fN}}_{\mathring{g}} \brs{F_{\fN}}_{\mathring{g}}^2}{\brs{\chi_{\la} \brs{F_{\fN}}^2_{\mathring{g}}}}\\
& \leq C \prs{\ga - 1}\brs{\fN F_{\fN}}_{\mathring{g}}\\
& \leq C \prs{\ga - 1}\prs{\brs{\tN F_{\fN}}_{\mathring{g}} + \brs{ \brk{\gY, F_{\fN}}}_{\mathring{g}}} \\
& \leq C \prs{\ga - 1} \prs{\brs{\tN^{(2)} \gY}_{\mathring{g}} + \brs{\tN \gY}_{\mathring{g}} \brs{\gY}_{\mathring{g}} + \brs{\gY}_{\mathring{g}}^3 }.
\end{split}
\end{align}
\begin{align}
\begin{split}\label{eq:Theta2approx}
\brs{\Theta_2\prs{\fN}}_{\mathring{g}} &\leq \tfrac{ 2  \chi_{\la}\prs{\ga -1}}{\brs{3 + \chi_{\la} \brs{F_{\fN}}^2_{\mathring{g}}}} \brs{ \fN \log \chi_{\la}}_{\mathring{g}} \brs{ F_{\fN} }_{\mathring{g}}\brs{F_{\fN}}_{\mathring{g}}^2 \\
& \leq \tfrac{ 2  \chi_{\la}\prs{\ga -1}}{\brs{\chi_{\la} \brs{F_{\fN}}^2_{\mathring{g}}}} \brs{ \fN \log \chi_{\la}}_{\mathring{g}} \brs{F_{\fN}}_{\mathring{g}}^3 \\
& \leq C  \prs{\ga -1}\brs{ \fN \log \chi_{\la}}_{\mathring{g}} \brs{ F_{\fN} }_{\mathring{g}} \\
&\leq C \prs{\ga-1} \brs{ \N \log \chi_{\la}}_{\mathring{g}} \prs{ 1 + \brs{\tN \gY}_{\mathring{g}} + \brs{\gY}^2_{\mathring{g}}}.
\end{split}
\end{align}
Via integration by parts combined with \eqref{eq:Fcomm} and applying the formula of the curvature tensor,
\begin{align}
\begin{split}\label{eq:hessionest}
\brs{\brs{ \tN^{(2)} \gY }}_{L^2}^2
&= - \int_{\mathbb{S}^4} \ip{ \tN \gY, \widetilde{\lap} \tN \gY }_{\mathring{g}} dV_{\mathring{g}} \\ 
&=  \brs{\brs{ \widetilde{\lap} \gY }}_{L^2}^2 - \int_{\mathbb{S}^4} \ip{ \tN_l \gY_i, \tN_k \brk{\tN_k, \tN_l} \gY_i }_{\mathring{g}} dV_{\mathring{g}}  - \int_{\mathbb{S}^4} \ip{ \tN_l \gY_i, \brk{\tN_k,  \tN_l} \tN_k \gY_i }_{\mathring{g}} dV_{\mathring{g}}  \\ 
&=  \brs{\brs{ \widetilde{\lap} \gY }}_{L^2}^2 - \int_{\mathbb{S}^4} \ip{ \tN_l \gY_i, \tN_k \brk{\Rm_{klis}^{\mathbb{S}^4} \gY_s} + \Rm_{klks}^{\mathbb{S}^4} \tN_s \gY_i  + \Rm_{klis}^{\mathbb{S}^4} \tN_k \gY_s }_{\mathring{g}} \, dV_{\mathring{g}}\\
&\hsp - 2\int_{\mathbb{S}^4} \ip{ \tN_l \gY_i, \brk{\widetilde{F}_{kl}, \tN_k\gY_i} }_{\mathring{g}} dV_{\mathring{g}} \\
&\leq \brs{\brs{ \widetilde{\lap} \gY }}_{L^2}^2 + 11 \brs{\brs{ \tN \gY}}_{L^2}^2 - 2 \int_{\mathbb{S}^4} \ip{\brk{\tN_i, \tN_l} \gY_i,  \gY_l  }_{\mathring{g}} \, dV_{\mathring{g}}\\
&= \brs{\brs{ \widetilde{\lap} \gY }}_{L^2}^2 + 11 \brs{\brs{ \tN \gY}}_{L^2}^2 - 2 \int_{\mathbb{S}^4} \ip{ \Rm_{ilis}^{\mathbb{S}^4}\gY_s + \brk{\widetilde{F}_{il}, \gY_s},  \gY_l  }_{\mathring{g}} \, dV_{\mathring{g}} \\
&\leq  \brs{\brs{ \widetilde{\lap} \gY }}_{L^2}^2 + 11 \brs{\brs{ \tN \gY}}_{L^2}^2 + 8 \brs{\brs{\gY}}_{L^2}^2 \\
&< C \delta +  \brs{\brs{ \widetilde{\lap} \gY }}_{L^2}^2.
\end{split}
\end{align}
We will estimate the latter term. Recall from Proposition \ref{cor:D*Fpol} in the appendix combined with \eqref{eq:LMM2.1},
\begin{align}
\begin{split}\label{eq:D*Fdiff}
\widetilde{\lap} \gY_{i\theta}^{\gb}  &=  - 3 \gY_{i \theta}^{\gb}  - 2\gY_{k \mu}^{\gb} \widetilde{F}_{k i \theta}^{\mu} + 2\widetilde{F}_{ki \mu}^{\gb}\gY_{k \theta}^{\mu}
 \\
 & \hsp - \gY_{i \gz}^{\gb} \gY_{k \mu}^{\gz} \gY_{k \theta}^{\mu}  + 2 \gY_{k \mu}^{\gb} \gY_{i \gz}^{\mu} \gY_{k \theta}^{\gz}   -  \gY_{k \mu}^{\gb} \gY_{k \gz}^{\mu} \gY_{i \theta}^{\gz}  + \prs{\Theta_1 + \Theta_2}_{i \ga}^{\gb} \\
 & \hsp  -  \prs{\tN_i \gY_{k\mu}^{\gb}} \gY_{k \theta}^{\mu} + \gY_{k \mu}^{\gb}  \prs{\tN_i \gY_{k\theta}^{\mu}}   -  2\gY_{k \mu}^{\gb} \prs{\tN_k \gY_{i \theta}^{\mu}}  +   2\prs{\tN_k\gY_{i \mu}^{\gb}} \gY_{k \theta}^{\mu}.
 \end{split}
\end{align}
With this in mind we compute each term of
\begin{align*}
\brs{\brs{ \widetilde{\lap} \gY}}_{L^2}^2 &= \int_{\mathbb{S}^4} \prs{-3 \ip{ \gY, \widetilde{\lap} \gY} + 2 \ip{ \brk{\widetilde{F}_{ki}, \gY_k}, \widetilde{\lap} \gY_i }} \, dV_{\mathring{g}} + \int_{\mathbb{S}^4} \ip{ \widetilde{\lap} \gY , \gY^{\ast_{\mathring{g}} 3}}_{\mathring{g}} \, dV_{\mathring{g}} \\
& \hsp + \int_{\mathbb{S}^4} \prs{\tN \gY \ast_{\mathring{g}} \gY \ast_{\mathring{g}} \widetilde{\lap} \gY } \, dV_{\mathring{g}} + \int_{\mathring{S}^4} \ip{ \prs{\Theta_1 + \Theta_2}, \widetilde{\lap} \gY }_{\mathring{g}} \, dV_{\mathring{g}}.
\end{align*}
For the first term we apply integration by parts combined with \eqref{eq:Fcomm}, noting $F_{\tN}$ is $\tN$-parallel
\begin{align*}
\int_{\mathbb{S}^4} \prs{-3 \ip{ \gY, \widetilde{\lap} \gY}^2 + 2 \ip{ \brk{\widetilde{F}_{ki}, \gY_k}, \widetilde{\lap} \gY_i }} \, dV_{\mathring{g}} &= \int_{\mathbb{S}^4} 3 \brs{ \tN \gY}^2 + 2 \ip{ \brk{ \tN_{j} \gY_i, \widetilde{F}_{ki} }, \tN_{j} \gY_i} \, dV_{\mathring{g}} \\
&\leq C \prs{\prs{\ga - 1} + \gd},
\end{align*}
For the second term, via integration by parts, H\"{o}lder's inequality and Sobolev embedding
\begin{align}
\begin{split}\label{eq:T2manip}
\int_{\mathbb{S}^4} \ip{ \widetilde{\lap} \gY , \gY^{\ast 3}}_{\mathring{g}} \, dV_{\mathring{g}} &\leq C \int_{\mathbb{S}^4} \brs{\N \gY}^2_{\mathring{g}}  \brs{\gY}^2_{\mathring{g}} \, dV_{\mathring{g}} \\
& \leq  C \prs{ \int_{\mathbb{S}^4} \brs{\tN \gY}^4_{\mathring{g}} \, dV_{\mathring{g}}}^{1/2} \prs{ \int_{\mathbb{S}^4} \brs{\gY}^4_{\mathring{g}} \, dV_{\mathring{g}}}^{1/2} \\
& \leq C \prs{\prs{\ga - 1} + \delta} \prs{ C \prs{\prs{\ga - 1} + \delta} + \int_{\mathbb{S}^4} \brs{\tN^{(2)} \gY}^2_{\mathring{g}} \, dV_{\mathring{g}} } \\
& \leq C \prs{\prs{\ga - 1} + \delta} \int_{\mathbb{S}^4} \brs{\tN^{(2)} \gY}_{\mathring{g}}^2 \, dV_{\mathring{g}} + C \prs{\prs{\ga - 1} + \delta}.
\end{split}
\end{align}
Likewise we have that, using a weighted H\"{o}lder's inequality and then applying the same equation as above in \eqref{eq:T2manip},
\begin{align}
\begin{split}\label{eq:T3manip}
\int_{\mathbb{S}^4} \prs{\tN \gY \ast_{\mathring{g}} \gY \ast_{\mathring{g}} \widetilde{\lap} \gY } \, dV_{\mathring{g}} &\leq \nu \int_{\mathbb{S}^4} \brs{\widetilde{\lap} \gY}^2_{\mathring{g}} \, dV_{\mathring{g}} + \tfrac{C}{\nu} \int_{\mathbb{S}^4} \brs{ \gY }^2_{\mathring{g}} \brs{\tN \gY}^2_{\mathring{g}} \, dV_{\mathring{g}} \\
&\leq C \prs{\nu + \prs{\ga - 1} + \delta } \int_{\mathbb{S}^4} \brs{\widetilde{\lap} \gY}^2_{\mathring{g}} \, dV_{\mathring{g}} + C \prs{\prs{\ga - 1} + \delta}.
\end{split}
\end{align}
Decomposing
\begin{align*}
\int_{\mathbb{S}^4} \ip{\widetilde{\lap}\gY,\Theta_1}_{\mathring{g}} \, dV_{\mathring{g}} &\leq C \prs{\ga - 1} \brk{ \int_{\mathbb{S}^4} \brs{\tN^2 \gY}^2_{\mathring{g}} \, dV_{\mathring{g}} +\int_{\mathbb{S}^4} \brs{\tN \gY}_{\mathring{g}} \brs{\gY}_{\mathring{g}} \brs{\tN^{(2)} \gY}_{\mathring{g}}\, dV_{\mathring{g}} + \int_{\mathbb{S}^4} \brs{\gY}_{\mathring{g}} \brs{\widetilde{\lap} \gY}_{\mathring{g}}\, dV_{\mathring{g}}}\\
&\hsp + C (\ga -1 )\int_{\mathbb{S}^4} \brs{\widetilde{\lap} \gY}_{\mathring{g}} \brs{\gY}^3_{\mathring{g}} \, dV_{\mathring{g}}  \\
&= \prs{\ga - 1} \sum_{i=1}^4
T_i.
\end{align*}
For the second term,
\begin{align*}
T_2 &\leq \nu \int_{\mathbb{S}^4} \brs{\tN^{(2)} \gY}_{\mathring{g}}^2 \, dV_{\mathring{g}} + \tfrac{C}{\nu} \int_{\mathbb{S}^4} \brs{\tN \gY}_{\mathring{g}}^2 \brs{\gY}_{\mathring{g}}^2 \, dV_{\mathring{g}} \\
&\leq \nu \int_{\mathbb{S}^4} \brs{\tN^{(2)} \gY}^2_{\mathring{g}} \, dV_{\mathring{g}} + \prs{ \int_{\mathbb{S}^4} \brs{\tN \gY}^4_{\mathring{g}} \,  dV_{\mathring{g}} }^{1/2}\prs{ \int_{\mathbb{S}^4} \brs{\gY}^4_{\mathring{g}} \,  dV_{\mathring{g}} }^{1/2} \\
&\leq \nu \int_{\mathbb{S}^4} \brs{\tN^{(2)} \gY}_{\mathring{g}}^2 \, dV_{\mathring{g}}  + C \prs{\delta + \prs{\ga - 1}} \prs{ C \prs{\delta + \prs{\ga - 1}} + \int_{\mathbb{S}^4} \brs{\tN^{(2)} \gY}_{\mathring{g}}^2 \, dV_{\mathring{g}} } \\
& \leq  C \prs{\delta + \prs{\ga - 1} + \nu} \int_{\mathbb{S}^4} \brs{\tN^{(2)} \gY}_{\mathring{g}}^2 \, dV_{\mathring{g}} + C \prs{\delta + \prs{\ga - 1} }.
\end{align*}
For the third term,
\begin{align*}
T_3 &\leq \int_{\mathbb{S}^4} \brs{\gY}_{\mathring{g}} \brs{\tN^{(2)} \gY}_{\mathring{g}} \, dV_{\mathring{g}} \\
&\leq C \prs{ \int_{\mathbb{S}^4} \brs{\tN \gY}^2_{\mathring{g}} \, dV_{\mathring{g}} }^{1/2} \prs{ \int_{\mathbb{S}^4} \brs{\gY}^2_{\mathring{g}} \, dV_{\mathring{g}} }^{1/2} \\
&\leq C_P \int_{\mathbb{S}^4} \brs{\tN \gY}_{\mathring{g}}^2 \, dV_{\mathring{g}}.
\end{align*}
$T_4$ follows exactly as in \eqref{eq:T2manip} and \eqref{eq:T3manip} above.
\begin{align*}
\int_{\mathbb{S}^4} \ip{\widetilde{\lap}\gY,\Theta_2}_{\mathring{g}} \, dV_{\mathring{g}} &\leq \prs{\ga - 1} \brk{\int_{\mathbb{S}^4}\brs{\tN \log \chi_{\la}}_{\mathring{g}} \brs{\tN \gY}_{\mathring{g}} \brs{\widetilde{\lap} \gY}_{\mathring{g}} \, dV_{\mathring{g}} + \int_{\mathbb{S}^4} \brs{\tN \log \chi_{\la}}_{\mathring{g}} \brs{\tN \gY}_{\mathring{g}} \brs{\gY}^2_{\mathring{g}} \, dV_{\mathring{g}}} = \sum_{i=1}^2 Q_i.
\end{align*}
We estimate these two terms. First we estimate
\begin{align*}
Q_1 &\leq C \int_{\mathbb{S}^4} \prs{ \brs{ \tN \log \chi_{\la}}_{\mathring{g}}^2 \brs{\tN \gY}_{\mathring{g}}^2 + \brs{\widetilde{\lap} \gY}_{\mathring{g}}^2} \, dV_{\mathring{g}} \\
&\leq C \prs{ \int_{\mathbb{S}^4} \brs{\tN \log \chi_{\la}}_{\mathring{g}}^4 }^{1/2} \prs{ \int_{\mathbb{S}^4} \brs{\tN \gY}_{\mathring{g}}^4 }^{1/2} + C \int_{\mathbb{S}^4} \brs{\tN^{(2)} \gY}_{\mathring{g}}^2 \, dV_{\mathring{g}} \\
&\leq C \log \la \prs{ C \prs{ \ga - 1 + \delta} + \int_{\mathbb{S}^4} \brs{\tN^{(2)} \gY}_{\mathring{g}}^2 \, dV_{\mathring{g}} } + C \int_{\mathbb{S}^4} \brs{\tN^{(2)} \gY}^2_{\mathring{g}} \, dV_{\mathring{g}}.
\end{align*}
Next we estimate
\begin{align*}
Q_2 &\leq \int_{\mathbb{S}^4} \brs{\gY}^4_{\mathring{g}} \, dV_{\mathring{g}} + \int_{\mathbb{S}^4}  \brs{\tN \log \chi_{\la}}_{\mathring{g}}^2 \brs{\tN\gY}^2_{\mathring{g}} \, dV_{\mathring{g}} \\
&\leq C \prs{ \ga - 1 + \delta } + C \prs{ \int_{\mathbb{S}^4} \brs{\tN \log \chi_{\la}}^4_{\mathring{g}}  \, dV_{\mathring{g}}}^{1/2} \prs{\int_{\mathbb{S}^4} \brs{\tN \gY}^4_{\mathring{g}} \, dV_{\mathring{g}}}^{1/2} \\
&\leq C \prs{ \ga - 1 + \delta } + C \log \la \prs{C \prs{ \ga - 1 + \delta } + \int_{\mathbb{S}^4} \brs{\tN^{(2)} \gY}^2_{\mathring{g}} \, dV_{\mathring{g}}}.
\end{align*}
Therefore we have that, combining everything
\begin{align*}
\brs{\brs{ \tN^{(2)} \gY}}_{L^2}^2 \leq C \prs{\ga - 1 + \delta} \brs{\brs{\tN^{(2)} \gY}}_{L^2}^2 + C \prs{\ga - 1 + \delta}.
\end{align*}
Thus provided $\delta_0$, $\prs{\ga_0 - 1}$ are sufficiently small, we may absorb the $\tN^{(2)} \gY$ type terms into the left side of \eqref{eq:hessionest}, allowing us to conclude the control
\begin{equation}\label{eq:N2Up}
\brs{\brs{ \tN^{(2)} \gY}}_{L^2} \leq C \prs{\delta + \prs{ \ga - 1}}.
\end{equation}
It is not difficult to believe that derivatives of curvature differences are intrinsically related to derivatives of $\gY$. Consider the case for the $L^2$ control of the $\tN$ derivative of $(F_{\N} - F_{\tN})$. Applying Proposition \ref{eq:Fpolar},
\begin{align*}
\brs{ \tN \brk{F_{\fN} - F_{\tN}}}_{\mathring{g}} &\leq C \prs{ \brs{\tN^{(2)} \gY}_{\mathring{g}} + \brs{ \gY \ast \tN \gY}_{\mathring{g}} }.
\end{align*}
Consequently the estimates in $L^2$ on this term follow directly from those above.

Our next goal is to demonstrate that $\tN \gY \in \mathcal{M}^{4}_{\gb}$ for some $\gb > 0$. To do so we use the Sobolev embedding $\mathcal{M}^{1,2}_{\gb} \hookrightarrow   \mathcal{M}^{4}_{\gb}$ and focus our attention on demonstrating containment in $\mathcal{M}^{1,2}_{\gb}$ (for the sake of brevity we include this step in the appendix, Lemma \ref{lem:MorreyUp}). Provided the estimate $\brs{\brs{\gY}}_{\mathcal{M}^{4}_{\gb}} \leq C \prs{\ga- 1+\delta}$ we will translate our setting to a functional perspective to apply Theorem \ref{thm:MLemma}. In a $\tN$-adapted frame we note the control $\brs{\widetilde{\gG} } \leq C R$.
Therefore using the coordinate decomposition $\tN_i \gY_{j \ga}^{\gb} = \del_i \gY_{j\ga}^{\gb} + \brk{\widetilde{\gG}_i, \gY_j }_{\ga}^{\gb}$,
\begin{align*}
R^{-\gb} \int_{B_R} \brs{\del_i \gY_{j \ga}^{\gb}}^4 \, dV_{\mathring{g}}  &\leq R^{-\gb} \int_{B_R}\prs{ \brs{\tN_i \gY_{j \ga}^{\gb}}^4 + C R^4 \brs{\gY_{j \ga}^{\gb}}^4 } \, \dVe \\
&\leq R^{-\gb} \int_{B_R} \brs{\tN \gY}^4 \, \dVe + R^{4-\gb} \int_{B_R} \brs{\gY}^4 \, \dVe \\
&\leq C \prs{\prs{\ga - 1} + \delta}.
\end{align*}
Applying Theorem \ref{thm:MLemma} to each coefficient function of $\gY$, with $p = n = 4$ we thus have desired H\"{o}lder continuity for all coefficients of $\gY$, and so
\begin{align*}
\brs{\brs{ \gY }}_{C^{0, \frac{\gb}{4}}} \leq C \prs{ \prs{\ga - 1} + \delta}.
\end{align*}
This is a particularly strong and immediately implies that $\brs{\brs{\gY}}_{L^{\infty}} \leq C \prs{\ga-1 + \delta}$. Furthermore, as discussed on \cite{GM} pp.76-78 (combining Proposition 5.4 and Theorem 5.5 of the text) we have $C^{0,\frac{\gb}{4} }$ is equivalent to $\mathcal{M}^2_{\nu}$ in the sense of functions, where $\nu := \frac{\gb}{2} + 4$. Applying these results to the coefficient functions of $\tN$ on a local level as above in a $\tN$-adapted frame, we obtain that
\begin{align*}
\brs{\brs{ \gY}}_{\mathcal{M}^2_{\nu}} \leq C \prs{\prs{\ga - 1} + \delta}.
\end{align*}
Via polarization of $F_{\fN} - F_{\tN}$ in terms of $\gY$ it follows that
\begin{align*}
\brs{\brs{F_{\fN} - F_{\tN}}}_{\mathcal{M}^{1,2}_{\mu}} \leq C \prs{\prs{\ga-1} + \delta}, \quad \text{ where } \mu := \min \sqg{\nu,\gb}.
\end{align*}
Lastly, we use the fact that for $\varepsilon$ sufficiently small (in particular $\varepsilon \leq \frac{2\mu}{4 - \mu}$),
\begin{equation*}
\brs{\brs{F_{\fN} - F_{\tN}}}_{\mathcal{M}^{1,2 + \varepsilon}_{\gg}}  \leq C \brs{\brs{F_{\fN} - F_{\tN}}}_{\mathcal{M}^{1,2}_{\mu}}, \quad \text{ where } \gg \leq \mu \prs{1+\tfrac{\varepsilon}{2}} - 2 \varepsilon.
\end{equation*}
Provided $\ga_0$ is chosen sufficiently small so that the range of $p$ values lies in $2 + \varepsilon$, in particular, $\ga_0 < \frac{2 \varepsilon + 2}{(2-\varepsilon)}$ (cf. \eqref{eq:pderiv} for the reasoning for such choice) then we finally attain
\begin{align*}
\brs{\brs{ F_{\fN} - F_{\tN}}}_{W^{1,p}} \leq C \prs{\prs{\ga - 1} + \delta},
\end{align*}
and thus concluding \eqref{eq:LMM4.1}.
\end{proof}
\end{prop}

\section{Bound on $\la$}\label{s:labd}

We next demonstrate how the estimates \eqref{eq:LMM3.2} and \eqref{eq:LMM4.1} imply small growth of $\tfrac{\del}{\del \log \la}\brk{ \mathcal{YM}_{\ga,\la} (\tN)}$ which, when coupled with \eqref{eq:LMM3.12}, yields a bound on $\la$ (which is independent of how close $\ga$ is to $1$). We compute $\tfrac{\del}{\del \log \la}\brk{\mathcal{YM}_{\ga,\la} (\tN)}$ directly from \eqref{eq:chidefn} and \eqref{eq:LMM2.7}. 
\begin{lemma}
We have the following equalities.
\begin{align*}
\tfrac{\del \chi_{\la}}{\del \log \la}(\zeta) & = \chi_{\la}(\zeta) \prs{\tfrac{ 4 \prs{\la^2 \brs{\zeta}^2 - 1}}{\la^2 \brs{\zeta}^2 + 1}}.
\end{align*}
\begin{proof}
Recalling that $\chi_{\la}(\zeta) = \frac{1}{\la^4}  \prs{ \frac{1 + \brs{\la \zeta}^2}{\brs{\zeta}^2 + 1 } }^4$, we have
\begin{align*}
\prs{\log \chi_{\la}}(\zeta)
&= 4 \log \prs{1 + \la^2 \brs{\zeta}^2 } - 4 \log{\la} - 4 \log \prs{\brs{\zeta}^2 + 1}  \\
\tfrac{\del\brk{\log \brk{ \chi_{\la}}}}{\del \la} (\zeta), 
&= \tfrac{ 4 \prs{\la^2 \brs{\zeta}^2 - 1}}{\la^2 \brs{\zeta}^2 + 1}.
\end{align*}
Remanipulating this accordingly, one obtains $\tfrac{\del \chi_{\la}}{\del \log \la} = \chi_{\la} \tfrac{\del \brk{\log\chi_{\la}}}{\del \log \la}$, giving the result.
%
%
\end{proof}
\end{lemma}
\begin{prop} The following inequality holds.
\begin{align}
\begin{split}\label{eq:LMM5.4}
\tfrac{\del}{\del \log \la}& \brk{\mathcal{YM}_{\ga , \la} \prs{\tN} } - \tfrac{\del}{\del \log \la} \brk{\mathcal{YM}_{\ga , \la} \prs{\fN} } \\
&\leq C \prs{\ga - 1} \prs{1+ \la^{4( \ga - 1)} }\brs{\brs{F_{\tN} -  F_{\fN} }}_{L^2}\prs{\brs{\brs{F_{\tN}}}_{L^2} + \brs{\brs{ F_{\fN} }}_{L^2} } \\
& \hsp + C \prs{\ga - 1}^2 \prs{1+ \la^{4( \ga - 1)} } \prs{ \brs{\brs{F_{\tN}}}_{L^{2 \ga + 2}} + \brs{\brs{F_{\fN}}}_{L^{2 \ga + 2}} } \brs{\brs{F_{\tN} -  F_{\fN} }}_{L^{2 \ga + 2}}  \brs{ \brs{F_{\fN} }}_{L^{2\ga+2}}^{2 \ga}.
\end{split}
\end{align}
\begin{proof}
We differentiate and obtain
\begin{align*}
\tfrac{\del}{\del \log \la}\brk{ \mathcal{YM}_{\ga,\la}(\fN)}
&= \tfrac{1}{2} \tfrac{\del}{\del \log \la} \brk{ \int_{\mathbb{S}^4} \prs{ 3 + \chi_{\la} \brs{F_{\fN}}^2_{\mathring{g}} }^{\ga} \tfrac{1}{\chi_{\la}} \, dV_{\mathring{g}} } \\
&=\tfrac{1}{2} \int_{\mathbb{S}^4} \prs{3 + \chi_{\la} \brs{F_{\fN}}_{\mathring{g}}^2 }^{\ga-1} \prs{ (\ga-1) \brs{F_{\fN}}^2_{\mathring{g}} -  \tfrac{2}{\chi_{\la}} } \prs{ \tfrac{\del \chi_{\la}}{\del \log \la} } \tfrac{1}{\chi_{\la}} \, dV_{\mathring{g}} \\
&=  2 \int_{\mathbb{S}^4} \prs{3 + \chi_{\la} \brs{F_{\fN}}_{\mathring{g}}^2 }^{\ga-1} \prs{ (\ga-1) \brs{F_{\fN}}^2_{\mathring{g}} -  \tfrac{2}{\chi_{\la}} } \underline{\prs{\tfrac{ \la^2 \brs{\zeta}^2 - 1}{\la^2 \brs{\zeta}^2 + 1}}} \, dV_{\mathring{g}}.
\end{align*}
For computational ease, set $\mu(\zeta) := \prs{\frac{ \brs{\zeta}^2 - 1}{\brs{\zeta}^2 + 1}} \in [-1,1)$ so that the underlined quantity is $\mu(\la \zeta)$.  Then
\begin{align}
\begin{split}\label{eq:LMM5.1}
\tfrac{\del}{\del \log \la} \brk{\mathcal{YM}_{\ga,\la} (\tN) } &- \tfrac{\del}{\del \log \la} \brk{\mathcal{YM}_{\ga,\la} (\fN)}\\
& = - \int_{\mathbb{S}^4} \prs{ \prs{3 + 3 \chi_{\la}  }^{\ga - 1} - \prs{ 3 + \chi_{\la} \brs{F_{\fN}}^2_{\mathring{g}} }^{\ga - 1} } \tfrac{2 \mu(\la \zeta)}{\chi_{\la}} \, dV_{\mathring{g}} \\
& \hsp + \prs{\ga - 1} \int_{\mathbb{S}^4} \prs{ 3 \prs{3 + 3 \chi_{\la} }^{\ga - 1} - \brs{F_{\fN}}_{\mathring{g}}^2 \prs{3 +  \chi_{\la} \brs{F_{\fN}}_{\mathring{g}}^2}^{\ga - 1} } \mu(\la \zeta)\, dV_{\mathring{g}}.
\end{split}
\end{align}
Similar to the proof of Lemma \ref{lem:LMM3.3}, there exists some $f : \mathbb{S}^4 \to [0,\infty)$ whose value at $\zeta \in \mathbb{S}^4$ lies between $\brs{F_{\fN}\prs{\zeta}}^2_{\mathring{g}}$ and $3 = \brs{F_{\tN}\prs{\zeta}}^2_{\mathring{g}}$ such that
\begin{align*}
\prs{ 3^{\ga} \prs{1 + \chi_{\la}}^{\ga - 1} - \prs{3 + \chi_{\la} \brs{F_{\fN}}^2_{\mathring{g}} }^{\ga - 1}  }= \prs{\ga - 1} \prs{3 + f \chi_{\la} }^{\ga - 2} \chi_{\la} \prs{3 - \brs{F_{\fN}}^2_{\mathring{g}}}.
\end{align*}
Then multiplying through by $\brs{F_{\N}}_{\mathring{g}}^2$ and remanipulating,
\begin{align*}
3 \prs{3 + 3 \chi_{\la} }^{\ga - 1} - &\brs{F_{\fN}}^2_{\mathring{g}} \prs{3 + \chi_{\la} \brs{F_{\fN}}_{\mathring{g}}^2 }^{\ga - 1} \\&= \prs{3 + 3 \chi_{\la}}^{\ga - 1} \prs{3 - \brs{F_{\fN}}_{\mathring{g}}^2 } + \prs{\ga - 1}\prs{3 + f \chi_{\la}}^{\ga - 2 }\chi_{\la} \prs{3 - \brs{F_{\fN}}_{\mathring{g}}^2 } \brs{F_{\fN}}_{\mathring{g}}^2.
\end{align*}
Note for $\ga \leq 2$ one has $\prs{3 + f \chi_{\la}}^{\ga - 2} \leq 1$. Furthermore
\begin{align*}
\tfrac{\chi_{\la} \brs{F_{\N}}_{\mathring{g}}^2}{3 + f \chi_{\la}} &\leq 
\begin{cases}
\tfrac{1}{3} \brs{F_{\fN}}_{\mathring{g}}^2 &\text{ if } \brs{F_{\fN}}^2_{\mathring{g}} \geq 3 \\
1 &\text{ if } \brs{F_{\fN}}_{\mathring{g}}^2 \leq 3
\end{cases} \\
& \leq 1 + \brs{F_{\fN}}^2_{\mathring{g}}.
\end{align*}
and furthermore,
\begin{align*}
\prs{3 + 3 \chi_{\la}}^{\ga - 1} &= 3^{\ga - 1} \prs{1 + \tfrac{\chi_{\la}}{3}}^{\ga - 1} \leq 6^{\ga-1} \la^{4(\ga- 1)}, \quad
\prs{3 + f \chi_{\la} }^{\ga - 1} \leq 6^{\ga-1} \la^{4(\ga - 1)} \prs{1+ \brs{F_{\fN}}_{\mathring{g}}^{2\ga - 2} }.
\end{align*}
Then using the fact that $\brs{\mu} \leq 1$,
\begin{equation*}\label{eq:LMM5.2}
\brs{ \prs{\prs{3 + 3 \chi_{\la} }^{\ga - 1} - \prs{3 + \chi_{\la} \brs{F_{\fN}}_{\mathring{g}}^2 }^{\ga - 1} }\tfrac{2 \mu(\la \zeta)}{\chi_{\la}} } \leq 2 \prs{\ga -1} \brs{3- \brs{F_{\fN}}_{\mathring{g}}^2 },
\end{equation*}
and
\begin{align}\label{eq:LMM5.3}
\brs{3 \prs{3 + 3 \chi_{\la}}^{\ga -1} - \brs{F_{\fN}}_{\mathring{g}}^2 \prs{3 + \chi_{\la} \brs{F_{\fN}}_{\mathring{g}}^2 }^{\ga - 1}  \mu(\la \zeta) } & \leq C \la^{4 ( \ga - 1)} \brs{3 - \brs{F_{\fN}}_{\mathring{g}}^2 } \prs{1 + (\ga - 1) \brs{F_{\fN}}^{2 \ga}_{\mathring{g}} }.
\end{align}
Therefor, applying \eqref{eq:LMM5.2},  \eqref{eq:LMM5.3} into  \eqref{eq:LMM5.1}
\begin{align*}
\begin{split}
\tfrac{\del}{\del \log \la} & \brk{\mathcal{YM}_{\ga , \la} \prs{\tN} } - \tfrac{\del}{\del \log \la} \brk{\mathcal{YM}_{\ga , \la} \prs{\fN} } \\
& \leq C \prs{\ga - 1} \prs{1+ \la^{4 (\ga - 1)} } \int_{\mathbb{S}^4} \brs{3 - \brs{F_{\fN}}^2_{\mathring{g}} } \prs{1 + \prs{\ga - 1} \brs{F_{\fN}}_{\mathring{g}}^{2\ga}} \, dV_{\mathring{g}} \\
&\leq \brk{C \prs{\ga - 1} \prs{1+ \la^{4 (\ga - 1)} } \int_{\mathbb{S}^4} \brs{3 - \brs{F_{\fN}}^2_{\mathring{g}} }  \, dV_{\mathring{g}}}_{T_1} \\
& \hsp +
\brk{C \prs{\ga - 1}^2 \prs{1+ \la^{4 (\ga - 1)} } \int_{\mathbb{S}^4} \brs{3 - \brs{F_{\fN}}^2_{\mathring{g}} } \brs{F_{\fN}}_{\mathring{g}}^{2\ga} \, dV_{\mathring{g}}}_{T_2}.
\end{split}
\end{align*}
For $T_1$ we compute out, first applying H\"{o}lder's inequality followed by triangle inequality
\begin{align*}
\int_{\mathbb{S}^4} \brs{ \brs{F_{\tN}}_{\mathring{g}}^2 - \brs{F_{\fN}}_{\mathring{g}}^2 } \, dV_{\mathring{g}} &\leq \int_{\mathbb{S}^4} \prs{\brs{F_{\tN}}_{\mathring{g}} -  \brs{F_{\fN}}_{\mathring{g}} } \prs{ \brs{F_{\tN}}_{\mathring{g}} + \brs{F_{\fN}}_{\mathring{g}} } \, dV_{\mathring{g}} \\
&\leq \int_{\mathbb{S}^4} \brs{F_{\tN} -  F_{\fN} }_{\mathring{g}} \brs{F_{\tN} + F_{\fN}}_{\mathring{g}} \, dV_{\mathring{g}}\\
& \leq \brs{\brs{F_{\tN} -  F_{\fN} }}_{L^2}\brs{\brs{F_{\tN} +  F_{\fN} }}_{L^2} \\
& \leq \brs{\brs{F_{\tN} -  F_{\fN} }}_{L^2}\prs{\brs{\brs{F_{\tN}}}_{L^2} + \brs{\brs{ F_{\fN} }}_{L^2} }.
\end{align*}
For $T_2$ this follows in a similar fashion but rather than apply the standard H\"{o}lder's inequality we apply the triple version, namely, for $f,g,h \in C^1(\mathbb{S}^4)$,
\begin{equation*}
\int f g h \, dV = \brs{\brs{ f }}_{L^p}\brs{\brs{ g }}_{L^q}\brs{\brs{ h }}_{L^r}, \quad \tfrac{1}{p} + \tfrac{1}{q} + \tfrac{1}{r} = 1.
\end{equation*}
where in our case we will take $p = q = 2 \ga + 2$ and $r = \frac{\ga}{2 \ga + 2}$. Now we compute
\begin{align*}
\int_{\mathbb{S}^4} \brs{ \brs{F_{\tN}}_{\mathring{g}}^2 - \brs{F_{\fN}}_{\mathring{g}}^2 } \brs{F_{\fN}}^{2 \ga}_{\mathring{g}} \, dV_{\mathring{g}} &\leq \int_{\mathbb{S}^4} \brs{F_{\tN} + F_{\fN}}_{\mathring{g}} \brs{F_{\tN} -  F_{\fN} }_{\mathring{g}} \brs{F_{\fN}}^{2 \ga}_{\mathring{g}}  \, dV_{\mathring{g}}\\
 &\leq \brs{\brs{F_{\tN} + F_{\fN}}}_{L^{2 \ga + 2}} \brs{\brs{F_{\tN} -  F_{\fN} }}_{L^{2 \ga + 2}}  \brs{ \brs{\brs{F_{\fN}}^{2 \ga}_{\mathring{g}} }}_{L^{\frac{\ga+1}{\ga}}}  \\
 & \leq  \prs{ \brs{\brs{F_{\tN}}}_{L^{2 \ga + 2}} + \brs{\brs{F_{\fN}}}_{L^{2 \ga + 2}} } \brs{\brs{F_{\tN} -  F_{\fN} }}_{L^{2 \ga + 2}}  \brs{ \brs{F_{\fN} }}_{L^{2\ga+2}}^{2 \ga}.
\end{align*}
Note that the manipulation of the last quantity follows from the fact that
\begin{equation*}
\brs{ \brs{\brs{F_{\fN}}^{2 \ga}_{\mathring{g}} }}_{L^{\frac{\ga+1}{\ga}}} = \prs{ \int_{\mathbb{S}^4} \prs{\brs{F_{\fN}}_{\mathring{g}}^{2 \ga}}^{\frac{\ga + 1}{\ga}} \, dV_{\mathring{g}} }^{\frac{\ga}{\ga + 1}} = \prs{ \prs{ \int \brs{F_{\fN}}_{\mathring{g}}^{2\ga + 2} \, dV_{\mathring{g}}}^{\frac{1}{2\ga + 2}}}^{2\ga} = \brs{ \brs{F_{\fN} }}_{L^{2\ga+2}}^{2 \ga}.
\end{equation*}
Combining these estimates we conclude \eqref{eq:LMM5.4}.
\end{proof}
\end{prop}
\begin{prop}\label{prop:LMM5.1}There exist $\ga_0 > 1$, $\delta_0 >0$ possibly smaller than those in Proposition \ref{prop:LMM4.1} such that if $\fN \in W^{1,2\ga}\prs{\mathcal{A}_{E}(\mathbb{S}^4)}$ is a critical point of $\mathcal{YM}_{\ga,\la}$ satisfying \eqref{eq:LMM3.1} and \eqref{eq:LMM3.2}, with $\ga \in (1,\ga_0 ]$ and $\delta \in (0, \delta_0]$, then
\begin{equation}\label{eq:LMM5.5}
\log \la \leq C \prs{\delta + \ga -1}.
\end{equation}
\begin{proof} As in Proposition \ref{prop:LMM5.1} we will be considering the relationship between the connections $\N$ and $\tN$. Then we can apply a Sobolev embedding to obtain that
\begin{align*}
\brs{\brs{F_{
\widetilde{\Pi}\brk{\fN}} - F_{\tN}}}_{L^{2\ga + 2}} \leq C_{S,\ga} \brs{\brs{F_{\widetilde{\Pi}\brk{\fN}} - F_{\tN}}}_{W^{1,p}}, \quad \text{ where }p := \tfrac{4(\ga + 1)}{\ga +3}.
\end{align*}
This choice is explained as follows: we have that $W^{1,p} \hookrightarrow L^{p^*}$, so if we desire $p^* = 2 \ga + 2$ we compute
\begin{equation}\label{eq:pderiv}
\prs{\tfrac{1}{p^*} = \tfrac{1}{p} - \tfrac{1}{4}} \rightarrow \prs{ p = \tfrac{4(\ga + 1)}{\ga +3}}.
\end{equation}
Since we can assume that $\ga_0 \leq 2$, we have that $p \in (2, \frac{12}{5}]$, as in Proposition \ref{prop:LMM4.1}. Then $C_{S,\ga}$ can in fact then be chosen independent of $\ga$, so by taking $\ga_0$ and $\delta_0$ as in Proposition \ref{prop:LMM4.1}, we obtain that, from \eqref{eq:LMM4.1},
\begin{equation}\label{eq:LMM5.6}
\brs{\brs{ F_{\widetilde{\Pi} \brk{\fN}} -F_{\tN}}}_{L^{2 \ga + 2}} 
\leq C_S \prs{\delta + \ga - 1}.
\end{equation}
Consequently, we have that,
\begin{equation}\label{eq:curvL2ga}
\brs{\brs{ F_{\fN}}}_{L^{2 \ga + 2}} =  \brs{\brs{ F_{\widetilde{\Pi}\brk{\fN}} }}_{L^{2 \ga + 2}}\leq \brs{\brs{  F_{\widetilde{\Pi}\brk{\fN}}-F_{\tN} }}_{L^{2 \ga + 2}} + \brs{\brs{ F_{\tN}}}_{L^{2\ga+2}} \leq C.
\end{equation}
Furthermore by \eqref{eq:LMM3.2}
\begin{equation}\label{eq:LMM5.7}
\la^{4 \prs{\ga  - 1} } < \max \sqg{ e^{4 C \gd}, e^{4 \ga_0 - 4} }.
\end{equation}
Since $\N$ is a critical point of $\mathcal{YM}_{\ga,\la}$ we claim that $\frac{\del}{\del \log \tau} \brk{ \mathcal{YM}_{\ga, \tau} \prs{\N}}_{\tau = \la} = 0$. To see this, first note that, via \eqref{eq:LMM2.6},
\begin{equation*}
\mathcal{YM}_{\ga,\tau}(\fN) = \mathcal{YM}_{\ga}((\tau^{-1})^*\fN) =   \mathcal{YM}_{\ga,\la}( \la^*(\tau^{-1})^*\fN),
\end{equation*}
and thus
\begin{equation}\label{eq:Eprime0}
\tfrac{\del}{\del \log \tau} \brk{ \mathcal{YM}_{\ga,\tau}\prs{\fN)} }_{\tau = \la} = \left. \prs{\tau \tfrac{\del}{\del \tau} \brk{\mathcal{YM}_{\ga,\tau} (\fN) } } \right|_{\tau = \la} = \mathcal{YM}_{\ga,\la}'(\fN)(\Xi),
\end{equation}
where $\Xi \in \Lambda^1(\Ad E)$ is given by
\begin{equation*}
\Xi : = \left. \prs{\tau \tfrac{\del}{\del \tau} \brk{ \prs{\la \tau^{-1}}^* \fN}} \right|_{\tau = \la}.
\end{equation*}
But, $\N$ is a critical point of $\mathcal{YM}_{\ga,\la}$ and thus $\mathcal{YM}_{\ga,\la}'(\fN) = 0$, which forces \eqref{eq:Eprime0} to vanish.
Consequently it follows from combining \eqref{eq:LMM3.12} (for a lower bound), \eqref{eq:LMM5.4}, \eqref{eq:LMM5.6}, \eqref{eq:curvL2ga} and \eqref{eq:LMM5.7} (for the upper bound) that
\begin{equation}\label{eq:LMM5.8}
\tfrac{(\ga - 1)}{C'} \tfrac{\log \la}{1+ \log \la} \leq \tfrac{\del}{\del \log \la} \brk{ \mathcal{YM}_{\ga,\la}(\tN) } \leq C \prs{\ga - 1} \prs{\gd + \ga - 1}. 
\end{equation}
We obtain the estimate \eqref{eq:LMM5.5} by taking $\ga_0$ and $\delta$ sufficiently close to $1$ and $0$ respectively.
\end{proof}
\end{prop}

\section{Proof of Main Result
}\label{s:optlacloseW2p}

Ultimately our goal is to prove the dilation factor of the conformal automorphism is precisely $1$, though at this point in our argument, we cannot to better than \eqref{eq:LMM5.5}; a bound dependent on the closeness in curvature and $\ga$-value. Proposition \ref{prop:LMM3.1} hints to choose a $\vp \in \SO(5,1)$ minimizing
\begin{equation}\label{eq:curvdiffmin}
\brs{\brs{  F_{\widetilde{\Pi} \brk{\vp^* \N}} - F_{\tN} }}_{L^2,(\mathbb{S}^4, \mathring{g})}^2 = \brs{\brs{  F_{(\vp^{-1})^* \vp^* \N} - F_{(\vp^{-1})^* \tN} }}_{L^2,(\mathbb{S}^4,(\vp^{-1})^*\mathring{g})}^2  = \brs{\brs{  F_{\N} - F_{(\vp^{-1})^* \tN} }}_{L^2,(\mathbb{S}^4, \mathring{g})}^2.
\end{equation}
For our purposes, however, we will be even more selective in our choice of minimizer to help ourselves in later computations. Noting the relationship of Theorem \ref{thm:curvcontrol} between connection and curvature difference, for a fixed $\N \in \mathcal{A}_{E} \prs{\mathbb{S}^4}$ we may instead choose to minimize
\begin{equation}\label{eq:Zminimization}
\mathcal{Z}^{\N} \prs{\vp} := \brs{\brs{ F_{\widetilde{\Pi} \brk{\vp^* \N}} - F_{\tN}}}^2_{L^2, \prs{\mathbb{S}^4, \mathring{g}}} + \brs{\brs{ \widetilde{\Pi} \brk{\vp^*\N} - \tN}}^2_{L^2, \prs{\mathbb{S}^4, \mathring{g}}}.
\end{equation}
To justify that this minimization is possible within $\SO(5,1)$, we prove the following.
\begin{lemma}\label{lem:Ksec7}
It is sufficient to minimize $\mathcal{Z}^{\N}\prs{\vp}$ over a compact subset of $\SO(5,1)$.
\begin{proof}
To do this, we first note that minimizing $\mathcal{Z}^{\N}\prs{\varphi}$ is equivalent to minimizing \eqref{eq:curvdiffmin}. We make the key observation that (up to rotations), $\vp$ goes to infinity only if it approaches a dilation from the south pole towards the north pole by large $\la > 0$, so that the energy of the dilation portion of $\vp$, given by $\la$, is concentrating on a small disk $\overline{B}$ centered at the south pole.  Since we will be working in various regions of integration we will denote these in our subscripts.

Consider $\overline{B}$ so small that $\brs{\brs{ F_{\N} }}_{L^2,(\overline{B})}^2 < \ge$. We separate out and divide up the equality, applying the conformal and gauge invariance of the norm,
\begin{align}
\begin{split}\label{eq:Fdifyetagain}
\brs{ \brs{ F_{\widetilde{\Pi} \brk{\N}} - F_{(\vp^{-1})^* \tN} }}_{L^2}^2 &= \brs{\brs{ F_{\widetilde{\Pi} \brk{\N}} }}_{L^2}^2 + 2 \ip{ F_{\widetilde{\Pi} \brk{\N}}, F_{(\vp^{-1})^*\tN} }_{L^2} + \brs{\brs{  F_{(\vp^{-1})^*\tN}}}_{L^2}^2 \\
&= \brs{\brs{ F_{\N} }}_{L^2}^2 + 2 \ip{ F_{\widetilde{\Pi} \brk{\N}}, F_{(\vp^{-1})^*\tN} }_{L^2} + \brs{\brs{  F_{\tN}}}_{L^2}^2.
\end{split}
\end{align}
Furthermore, we can decompose the middle term into their corresponding parts on or outside of $\overline{B}$:
\begin{align*}
\ip{ F_{\widetilde{\Pi} \brk{\N}}, F_{(\vp^{-1})^*\tN} }_{L^2} &\leq \brs{\brs{ F_{ \N }}}^2_{L^2, \prs{\overline{B}}} \brs{\brs{ F_{\tN} } }^2_{L^2 ,\prs{\overline{B}}} + \brs{ \ip{ F_{\widetilde{\Pi} \brk{\N}}, F_{(\vp^{-1})^*\tN} }_{L^2,(\mathbb{S}^4 - \overline{B})}  } \\
& \leq 8 \pi^2 \ge + (\la^{-1})^{2} \brs{ \ip{ F_{\widetilde{\Pi} \brk{\N}}, F_{\la^*(\vp^{-1})^*\tN} }_{L^2,(\mathbb{S}^4 - \overline{B})}  }\\
& \leq C \ge.
\end{align*}
This implies that $\ip{ F_{\widetilde{\Pi} \brk{\N}}, F_{(\vp^{-1})^* \tN}}_{L^2}$ is small, and so in particular we can update \eqref{eq:Fdifyetagain},
\begin{align*}
\brs{ \brs{ F_{\widetilde{\Pi} \brk{\N}} - F_{(\vp^{-1})^* \tN} }}_{L^2}^2 \geq  16 \pi^2 - C \ge \ggg 0,
\end{align*}
concluding that the minimization of \eqref{eq:Zminimization} occurs on a compact subset of $\SO(5,1)$, away from conformal infinity, as desired.
\end{proof}
\end{lemma}
Recall the Yang-Mills energy Jacobi operator given by
\begin{equation}\label{eq:Joperator}
\prs{\mathcal{J}^{\N}(\Xi)}_{i \mu}^{\gb} := -(D^*_{\N} D_{\N} \Xi)_{i\mu}^{\gb}- \prs{F_{\N}}_{ki \delta}^{\gb} \Xi_{k \mu}^{\delta} + \Xi_{k \delta}^{\gb} \prs{F_{\N}}_{ki\mu}^{\delta} .
\end{equation}
This can be derived as follows. Varying the Euler-Lagrange equation for Yang-Mills energy,
\begin{align*}
\left. \tfrac{\del}{\del s} \brk{D^*_{\N_s} F_{\N_s} }_{i} \right|_{\N_s = \N}
&= -\left. \tfrac{\del}{\del s} \brk{ \N_k F_{k i} } \right|_{\N_s = \N}
= (D^*_{\N} D_{\N} \dot{\gG})_{i} + \brk{F_{ki},\dot{\gG}_{k}}.
\end{align*}
Using the the Bochner formula, which we substitute into \eqref{eq:Joperator},
\begin{align*}
\lap_{D} \Xi_{i} &= - \lap \Xi_{i} + 3 \Xi_{i}+ \brk{ F_{ki}, \Xi_k},
\end{align*}
we can write the Jacobi operator in the form
\begin{align*}
\prs{\mathcal{J}^{\N}\prs{\Xi}}_{i} =  \lap \Xi_{i} + \prs{DD^* \Xi}_{i} - 3\Xi_{i}- 2\brk{F_{ki}, \Xi_k }.
\end{align*}
%
  
The kernel of the Jacobi operator is precisely the tangent space to the moduli space of instantons (this follows from Proposition 4.2.23 of \cite{DK}), which in turn is precisely the tangent space of the orbit of $\tN$ under the action of $\SO(5,1)$. Let $\overline{A}$ denote the orthogonal projection of the $1$-form $A$ onto the kernel of $\mathcal{J}^{\tN}$ with respect to the $L^2$ inner product. Then from ellipticity of the system,
\begin{equation}\label{eq:W2pest}
\brs{\brs{ A }}_{W^{2,p}} \leq C \prs{ \brs{\brs{\mathcal{J}^{\tN} A}}_{L^p} + \brs{\brs{\overline{A}}}_{L^p} }.
\end{equation}
Now, assume that $\vp$ minimizes $\mathcal{Z}^{\N}\prs{\vp}$, and again set $\fN \equiv \vp^* \N$. We now estimate the two terms on the right side of the inequality in the case $A \equiv \gY (= \fN - \tN)$, keeping in mind two main identities:
\begin{align*}
D^*_{\tN} \gY = 0 \text{ (due to $\tN$-Coulomb gauge),}\qquad \mathcal{J}^{\tN} \prs{\overline{\gY}} = 0 \text{ (due to $\gY$ in $\Ker \mathcal{J}^{\tN}$)}.
\end{align*}
\begin{prop}\label{prop:W2pnorm}Assuming the results above
\begin{equation}\label{eq:LMM6.8}
\brs{\brs{ \gY }}_{W^{2,p}} \leq C \prs{\ga - 1} \log \la,
\end{equation}
\begin{proof}
We first estimate the $L^p$ norm of $\overline{\gY}$. From the minimizing property of $\mathcal{Z}^{\N} \prs{\varphi}$ we have that,
\begin{align*}
0 &=  \int_{\mathbb{S}^4} \ip{F_{\fN} -F_{\tN} , \N \overline{\gY}}_{\mathring{g}} \, dV_{\mathring{g}} + \int_{\mathbb{S}^4} \ip{ \gY , \overline{\gY} }_{\mathring{g}} \, dV_{\mathring{g}} \\
&=   \int_{\mathbb{S}^4} \ip{F_{\fN} -F_{\tN} , \tN \overline{\gY} + \brk{\gY, \overline{\gY}}}_{\mathring{g}} \, dV_{\mathring{g}}+ \int_{\mathbb{S}^4} \ip{ \gY , \overline{\gY} }_{\mathring{g}} \, dV_{\mathring{g}} \\
&=  \brk{\int_{\mathbb{S}^4} \ip{F_{\fN} -F_{\tN} , \tN \overline{\gY}}_{\mathring{g}} \, dV_{\mathring{g}} }_{T_1} + \brk{ \int_{\mathbb{S}^4} \ip{F_{\fN} -F_{\tN} ,\brk{\gY, \overline{\gY}}}_{\mathring{g}} \, dV_{\mathring{g}}}_{T_2} + \int_{\mathbb{S}^4} \ip{ \gY , \overline{\gY} }_{\mathring{g}} \, dV_{\mathring{g}}.
\end{align*}
Using the polarization identity of the curvature, Proposition \ref{prop:curvpol}, we break the terms $T_1$ and $T_2$ apart into four labelled integrals:
\begin{align*}
 \brs{\brs{ \overline{\gY}}}_{L^2}^2  &=  \brk{ \int_{\mathbb{S}^4} \prs{\widetilde{D}_i \gY_{j \gz}^{\gb}} \prs{\tN_i \overline{\gY}_{j \gb}^{\gz}} \, dV_{\mathring{g}}}_{T_{11}} + \brk{\int_{\mathbb{S}^4} \prs{\widetilde{D}_i \gY_{j \gz}^{\gb}}  \brk{\gY_i, \overline{\gY}_j }_{\gb}^{\gz} \, dV_{\mathring{g}}}_{T_{21}} \\
& \hsp + \brk{\int_{\mathbb{S}^4} \brk{\gY_i,\gY_j}_{\gz}^{\gb} \prs{\tN_i \overline{\gY}_{j \gb}^{\gz}} \, dV_{\mathring{g}}}_{T_{12}} + \brk{\int_{\mathbb{S}^4} \brk{\gY_i,\gY_j}_{\gz}^{\gb} \brk{\gY_i, \overline{\gY}_j }_{\gb}^{\gz} \, dV_{\mathring{g}}}_{T_{22}}.
\end{align*}
Let's manipulate these each separately. We have that
\begin{align*}
T_{11} &= \int_{\mathbb{S}^4} \prs{\widetilde{D}_i \gY_{j \gz}^{\gb}} \prs{\tN_i \overline{\gY}_{j \gb}^{\gz}} \, dV_{\mathring{g}}= \int_{\mathbb{S}^4} \prs{\tN_i \gY_{j \gz}^{\gb}} \prs{\widetilde{D}_i \overline{\gY}_{j \gb}^{\gz}} \, dV_{\mathring{g}}= \int_{\mathbb{S}^4} \gY_{j \gz}^{\gb} \prs{\widetilde{D}^*\widetilde{D} \overline{\gY}}_{j \gb}^{\gz} \, dV_{\mathring{g}}.
\end{align*}
Next we approach $T_{21}$. Here we have
\begin{align*}
T_{21} &= \int_{\mathbb{S}^4} \prs{\widetilde{D}_i \gY_{j \gz}^{\gb}}  \brk{\gY_i, \overline{\gY}_j }_{\gb}^{\gz} \, dV_{\mathring{g}} = 2 \int_{\mathbb{S}^4} \prs{\tN_i \gY_{j \gz}^{\gb}}  \brk{\gY_i, \overline{\gY}_j }_{\gb}^{\gz} \, dV_{\mathring{g}} = 2\int_{\mathbb{S}^4} \gY_{i\gz}^{\gb} \brk{ \gY_j, \prs{\tN_i \overline{\gY}_{j}} }_{\gz}^{\gb} \, dV_{\mathring{g}}. 
\end{align*}
Note for the following that we apply the fact that $D^* \gY \equiv 0$.
\begin{align*}
T_{21} 
&= - \tfrac{1}{2} T_{12}.
\end{align*}
For $T_{22}$, we use the skew symmetry with respect to the bundle indices,
\begin{align*}
T_{22} &= \int_{\mathbb{S}^4} \brk{\gY_i,\gY_j}_{\gz}^{\gb} \brk{\gY_i, \overline{\gY}_j }_{\gb}^{\gz} \, dV_{\mathring{g}} = 2 \int_{\mathbb{S}^4} \gY^{\gb}_{i \gz}\gY_{j \gd}^{\gz} \brk{\gY_i, \overline{\gY}_j }_{\gb}^{\gd} \, dV_{\mathring{g}}.
\end{align*}
Combining these together, we have that
\begin{align*}
  \brs{\brs{ \overline{\gY}}}_{L^2}^2 = \sum_{i,j=1}^2 T_{ij} &= \int_{\mathbb{S}^4}  \gY_{j \mu}^{\gb} \prs{\widetilde{D}^*\widetilde{D} \overline{\gY}}_{j \gb}^{\mu}  +  \gY_{i\mu}^{\gb} \brk{ \gY_j, \prs{\tN_i \overline{\gY}_{j}} }_{\mu}^{\gb} + 2 \gY^{\gb}_{i \gz}\gY_{j \mu}^{\gz} \brk{\gY_i, \overline{\gY}_j }_{\gb}^{\mu}
\end{align*}
Applying \eqref{eq:Joperator} noting that $\mathcal{J}^{\tN}\prs{\overline{\gY}} \equiv 0$ to the first term on the right, we have that
\begin{align}
\begin{split}\label{eq:overgYcontrol}
\brs{\brs{ \overline{\gY}}}_{L^2}^2 &=  \int_{\mathbb{S}^4} - \gY_{j \mu}^{\gb} \brk{ \widetilde{F}_{kj}, \overline{\gY}_k}_{\gb}^{\mu} + \gY_{i\mu}^{\gb} \brk{ \gY_j, \prs{\tN_i \overline{\gY}_{j}} }_{\mu}^{\gb} + 2 \gY^{\gb}_{i \gz}\gY_{j \mu}^{\gz} \brk{\gY_i, \overline{\gY}_j }_{\gb}^{\mu}\\
 & \leq C \brs{\brs{ \gY}}_{L^2}\brs{\brs{ \overline{\gY}}}_{L^2} +  \brs{\brs{ \gY} }_{L^{\infty}}^2\brs{\brs{ \tN \overline{\gY}} }_{L^2} + \brs{\brs{ \overline{\gY}}}_{L^2} \brs{\brs{ \gY^{\ast 3}}}_{L^2}\\
 & \leq C \delta \prs{\brs{\brs{ \overline{\gY}}}_{L^2} +\brs{\brs{ \tN \overline{\gY}}}_{L^2}}.
\end{split}
\end{align}
Now, we claim furthermore that $\brs{\brs{ \tN \overline{\gY}}}_{L^2} \leq C \brs{\brs{ \overline{\gY} }}_{L^2 }$. To see this, consider the ratio $\brs{\brs{ \tN\overline{\gY} }}_{L^2}\brs{\brs{ \overline{\gY} }}_{L^2}^{-1}$. This is scale invariant in $\overline{\gY}$ and when restricted to $\{ A \in \Lambda^1 \prs{\Ad E} : \brs{A}_{\mathring{g}} \equiv 1  \}$ has a maximum and minimum, and is thus bounded.
So we can update \eqref{eq:overgYcontrol} by applying this estimate and dividing through by $\brs{\brs{ \overline{\gY} }}_{L^2}$ to conclude that $\brs{\brs{ \overline{\gY} }}_{W^{1,2}} \leq C \delta$, and using the Sobolev embedding $W^{1,2} \hookrightarrow L^{p}$, we conclude $\brs{\brs{ \overline{\gY} }}_{L^p}^2 \leq C_S  \delta$.

Now, we estimate the Jacobi operator term, by observing that, by subtracting $\Theta_1$ and $\Theta_2$ from both sides of \eqref{eq:LMM2.1} and inserting in \eqref{eq:D*Fdiff}, and then finally observing the presence of the terms of the $\mathcal{J}^{\tN}$ (this is the first line of the right hand side of \eqref{eq:D*Fdiff}), with rearrangement we obtain
\begin{align*}
\begin{split}\label{eq:D*Fdiff}
 \prs{\mathcal{J}^{\tN} \prs{\gY}}_{i\theta}^{\gb} 
&=  \prs{\tN_i \gY_{k\mu}^{\gb}} \gY_{k \theta}^{\mu} -\gY_{k \mu}^{\gb}  \prs{\tN_i \gY_{k\theta}^{\mu}} +  2 \prs{\tN_k\gY_{i \mu}^{\gb}} \gY_{k \theta}^{\mu} \\
 & \hsp  - \gY_{i \gz}^{\gb} \gY_{k \mu}^{\gz} \gY_{k \theta}^{\mu}   +  \gY_{k \mu}^{\gb} \gY_{k \gz}^{\mu} \gY_{i \theta}^{\gz} + \prs{\Theta_1 \prs{\N}}_{i \theta}^{\gb} + \prs{\Theta_2 \prs{\N}}_{i \theta}^{\gb} .
 \end{split}
\end{align*}
In fact, we actually obtained all of these term types from the proof of Proposition \ref{prop:LMM4.1}. Therefore
\begin{equation*}
\brs{\brs{ \mathcal{J}^{\tN} \prs{\gY}}}_{W^{1,2} } \leq C \prs{\delta + \prs{\ga - 1}}\prs{1 + \log \la }.
\end{equation*}
Applying the Sobolev embedding, $L^p \hookrightarrow W^{1,2}$, by taking $\delta_0$ and $\prs{\ga_0 - 1}$ both sufficiently small, we may conclude that by inserting our estimates into \eqref{eq:W2pest} and absorbing proper terms across the inequality, we obtain \eqref{eq:LMM6.8}.
\end{proof}
\end{prop}
\begin{proof}[Proof of Theorem \ref{thm:LMM1.2}]
We refer back to the proof of Proposition \ref{prop:LMM5.1} using the improved estimate in \eqref{eq:LMM6.8} obtain%
\begin{align*}
\brs{\brs{ F_{\N} - F_{\tN}}}_{L^{2\ga + 2}} \leq C \prs{\ga - 1} \log \la.
\end{align*}
Then the inequalities in \eqref{eq:LMM5.8} transform to (in particular to \eqref{eq:LMM3.12})
\begin{align*}
\tfrac{ \prs{\ga - 1}}{C'} \tfrac{\log \la}{1+ \log \la} \leq \tfrac{\del}{\del \log \la} \brk{\mathcal{YM}_{\ga,\la}\prs{\tN} } \leq C \prs{\ga - 1}^2 \log \la.
\end{align*}
By taking $\ga$ sufficiently close to (but not equal to $1$) and knowing that $\lambda$ is greater than $1$ and bounded, conclude that $\la \equiv 1$. However, using \eqref{eq:LMM6.8} we conclude that $\gY$ vanishes, that is, $\N \equiv \tN$, which yields the result.
\end{proof}

\section{Appendix}\label{s:appendix}
\subsection{Background material}
We state a variety of key polarization identities, which require elementary computation to show. For the following let $\N,\gN \in C^2 \prs{\mathcal{A}_E(M)}$, and $\Upsilon := \N- \gN$. We record more general formulas; not assuming $M = \mathbb{S}^4$.
\begin{prop}[Curvature polarization]\label{prop:curvpol}
We have
\begin{equation}\label{eq:Fpolar}
\prs{F_{\N} - F_{\gN}}_{ij\theta}^{\gb} = \prs{\gN_i \gY_{j\theta}^{\gb}}  - \prs{\gN_j \gY_{i \theta}^{\gb}} + \gY_{i \mu}^{\gb} \gY_{j \theta}^{\mu}   -  \gY_{j \mu}^{\gb} \gY_{i \theta}^{\mu}.
\end{equation}
\end{prop}
\begin{prop}[$D^*_{\N} F_{\N}$ polarization]\label{cor:D*Fpol}
We have that
\begin{align*}
-\prs{D^*_{\N} F_{\N}- D_{\gN}^* F_{\gN}}_{i \theta}^{\gb}
  &= \prs{ \gN_k  \gN_k \gY_{i\theta}^{\gb}}  - \prs{ \gN_i  \gN_k \gY_{k \theta}^{\gb}} - \Rm_{kik}^p \gY_{p \theta}  + 2\gY_{k \mu}^{\gb} \prs{F_{\gN}}_{k i \theta}^{\mu} - 2\prs{F_{\gN}}_{ki \mu}^{\gb}\gY_{k \theta}^{\mu}
 \\
 & \hsp + \gY_{i \gz}^{\gb} \gY_{k \mu}^{\gz} \gY_{k \theta}^{\mu}  - 2 \gY_{k \mu}^{\gb} \gY_{i \gz}^{\mu} \gY_{k \theta}^{\gz}   +  \gY_{k \mu}^{\gb} \gY_{k \gz}^{\mu} \gY_{i \theta}^{\gz} \\
 & \hsp  +  \prs{\gN_i \gY_{k\mu}^{\gb}} \gY_{k \theta}^{\mu} - \gY_{k \mu}^{\gb}  \prs{\gN_i \gY_{k\theta}^{\mu}}   +  2\gY_{k \mu}^{\gb} \prs{\gN_k \gY_{i \theta}^{\mu}}  -   2\prs{\gN_k\gY_{i \mu}^{\gb}} \gY_{k \theta}^{\mu}  \\
    & \hsp   +  \prs{\gN_k \gY_{k \mu}^{\gb}} \gY_{i \theta}^{\mu}    -  \gY_{i \mu}^{\gb} \prs{\gN_k\gY_{k \theta}^{\mu}}.
\end{align*}
%
\end{prop}
\subsection{Various technical computations}\label{ss:tech}

\begin{lemma}[Estimate for $\chi_{\la}$]\label{lem:chiest}
There is a constant $C>0$ which is independent of $\la \geq 1$ so that
\begin{equation}\label{eq:logchila}
\brs{\brs{\N \log \chi_{\la}}}_{L^2} + \brs{\brs{ \N^{(2)} \log \chi_{\la}}}_{L^2} \leq
\begin{cases}
C \log \la & \text{ when } \la \in \brk{1,e} \\
C \prs{\log \la}^{1/2} & \text{ when } \la \in [e,\infty) \\
\end{cases}.
\end{equation}
\begin{proof}
To apply spherical coordinates we use the appropriate change of variables
\begin{align*}
\zeta_1 = r \sin \vartheta_1 \sin \vartheta_2 \cos \vartheta_3, \quad \zeta_2 = r \sin \vartheta_1 \sin \vartheta_2 \sin \vartheta_3, \quad\zeta_3 = r \sin \vartheta_1 \cos \vartheta_2, \quad \zeta_4 = r \cos \vartheta_4 \\
g_{rr} = 1, \quad g_{11} :=g_{\vartheta_1 \vartheta_1}  = r^2, \quad g_{22}:=g_{\vartheta_2 \vartheta_2} = r^2 \sin^2 \vartheta_1, \quad g_{33}:=g_{\vartheta_3 \vartheta_3} = r^2 \sin^2 \vartheta_1 \sin^2 \vartheta_2. 
\end{align*}
Inserting these into the formula for the Levi Civita connection, we have
%
\begin{align}\label{eq:gGstuff}
\gG_{ij}^r = 0 \, (i \neq j), \quad \gG_{11}^r = r, \quad \gG_{22}^r = r \sin^2 \vartheta_1, \quad \gG_{33}^r = r \sin^2 \vartheta_1 \sin^2 \vartheta_2,
\end{align}
Recalling the formula for $\chi_{\la}(\zeta) = \frac{1}{\la^4}  \prs{ \frac{1 + \brs{\la \zeta}^2}{\brs{\zeta}^2 + 1 } }^4$ via \eqref{eq:chidefn}, take $r=\brs{\zeta}$.
%
%
Observe that
\begin{align}
\begin{split}\label{eq:logchila}
\N \log \chi_{\la} &= \tfrac{\del}{\del r} \brk{ \log \chi_{\la} } \quad \text{and} \quad
\N^{(2)} \log \chi_{\la} = \tfrac{\del^2}{\del r^2} \log \chi_{\la} - \prs{\sum_{i=1}^3 \gG_{ii}^r } \tfrac{\del}{\del r} \brk{\log \chi_{\la}}.
\end{split}
\end{align}
First we compute
\begin{align*}
\tfrac{\del}{\del r} \brk{\log \chi_{\la}}
&=  \tfrac{ 8r(\la^2 -1)}{\prs{1+ \la^2 r^2}(1+r^2)}, \quad \tfrac{\del^2}{\del r^2} \brk{\log \chi_{\la}} = - \tfrac{8 \prs{\la^2-1} \prs{3 \la^2 r^4 + \prs{\la^2+1}r^2 -1}}{\prs{r^2+1}^2 \prs{\la^2 r^2 + 1}^2}.
\end{align*}
Thus we have that, for the first derivative of $\log \chi_{\la}$,
\begin{align*}
\brs{\brs{ \N \log \chi_{\la}}}_{L^2}^2 &= \prs{\int_0^{2\pi} \int_0^{\pi} \int_0^{\pi}  \, d \vartheta_1 \, d \vartheta_2 \, d \vartheta_3 }\prs{\int_0^{\infty} \tfrac{r^3}{\prs{1+r^2}^2}\prs{ \tfrac{\del}{\del r}\brk{\log \chi_{\la}} }^2 \, dr}\\ & = 2 \pi^3 \prs{\int_0^{1/\la}+\int_{1/\la}^1 + \int_1^{\infty}  } \tfrac{r^3}{\prs{1+r^2}^2}\prs{ \tfrac{\del}{\del r}\brk{\log \chi_{\la}} }^2 \, dr \\
& = 2^7 \pi^3  \prs{\la^2-1}^1\prs{\int_0^{1/\la} r^5 \, dr +  \int_{1/\la}^1 \tfrac{r}{\la^4}  \, dr  +  \int_0^{1/\la} \tfrac{1}{\la^4 r^{7}} \, dr}\\
&= 2^7 \pi^3  \prs{\la^2 - 1}^2 \prs{\tfrac{1}{6 \la^6} + \tfrac{1}{2\la^4} - \tfrac{1}{2 \la^6} + \tfrac{1}{6 \la^4}}\\
&=  2^7 \pi^3  \prs{\tfrac{\la - 1}{\la}  \cdot \tfrac{\la + 1}{\la}}^{2} \prs{\tfrac{2}{3} - \tfrac{1}{3 \la^2}}\\
&=  \tfrac{2^7}{3} \pi^3  \prs{\tfrac{\la - 1}{\la} }^{2}\prs{\tfrac{\la + 1}{\la}}^{2}.
\end{align*}
Now we approach the second derivative of $\log \chi_{\la}$. Utilizing \eqref{eq:logchila} above, we note that
\begin{equation}\label{eq:N2decomp}
\brs{\brs{ \N^{(2)} \log \chi_{\la} }}_{L^2}^2 \leq \brs{\brs{ \tfrac{\del^2}{\del r^2} \brk{\log \chi_{\la}}}}_{L^2}^2  + \brs{\brs{ \gG^r \del_r \log \chi_\la }}_{L^2}^2.
\end{equation}
In this case we have that
\begin{align*}
&\brs{\brs{ \tfrac{\del^2}{\del r^2} \brk{\log \chi_{\la}}}}_{L^2}^2  \\
\hsp\hsp &= \prs{\int_0^{2\pi} \int_0^{\pi} \int_0^{\pi} \, d \vartheta_1 \, d \vartheta_2 \, d \vartheta_3}\prs{\int_0^{\infty} \tfrac{r^3}{\prs{1+r^2}^4}\prs{ \tfrac{\del^2}{\del r^2}\brk{\log \chi_{\la}} }^2 \, dr }\\
\hsp\hsp &= 2 \pi^3 \prs{\int_0^{1/\la} + \int_{1/\la}^1+\int_1^{\infty}} \tfrac{r^3}{\prs{1+r^2}^4}\prs{ \tfrac{\del^2}{\del r^2}\brk{\log \chi_{\la}} }^2  \, dr \\
\hsp\hsp &\leq 2^7 \pi^3 \prs{\la^2-1}^2 \int_0^{1/\la} \prs{r^{11} 9 \la^4 + r^9 \prs{6 \la^4 + 6 \la^2} + r^7 \prs{\la^4 - 4 \la^2 + 1} + r^5 \prs{-2 \la^2 - 2} + r^3} \, dr\\
\hsp\hsp& \hsp + 2^7 \pi^3 \prs{\la^2-1}^2 \int_{1/\la}^1 \tfrac{1}{r^{8} \la^8} \prs{r^{11} 9 \la^4 + r^9 \prs{6 \la^4 + 6 \la^2} + r^7 \prs{\la^4 - 4 \la^2 + 1} + r^5 \prs{-2 \la^2 - 2} + r^3} \, dr\\
\hsp\hsp& \hsp+ 2^7 \pi^3 \prs{\la^2-1}^2 \int_1^{\infty} \tfrac{1}{r^{24} \la^8} \prs{r^{11} 9 \la^4 + r^9 \prs{6 \la^4 + 6 \la^2} + r^7 \prs{\la^4 - 4 \la^2 + 1} + r^5 \prs{-2 \la^2 - 2} + r^3} \, dr\\
\hsp\hsp &=  2^7 \pi^3 \prs{ \tfrac{\la-1}{\la} }^{2} \prs{ \tfrac{\la+1}{\la} }^{2} \prs{ - \tfrac{2177}{720} \tfrac{1}{\la^4} - \tfrac{209}{1260} \tfrac{1}{\la^2} + \tfrac{1943}{336}} \\
&= 2^7 \tfrac{1307}{504} \pi^3 \prs{ \tfrac{\la-1}{\la} }^{2} \prs{ \tfrac{\la+1}{\la} }^{2}.
\end{align*}
%
%
For the second component of \eqref{eq:N2decomp}, we have
\begin{align*}
\int_0^{2\pi} \int_0^{\pi} \int_0^{\pi}& \int_0^{\infty} \tfrac{r^3}{\prs{1+r^2}^4}\prs{ \gG^r \tfrac{\del}{\del r}\brk{\log \chi_{\la}} }^2 \, dr \, d \vartheta_1 \, d \vartheta_2 \, d \vartheta_3 \\&=  \prs{\int_0^{2\pi} \int_0^{\pi} \int_0^{\pi}  \prs{1+\sin^2 \vartheta_1 + \sin^2 \vartheta_1 \sin^2 \vartheta_2}^2 \, d \vartheta_1 \,  d \vartheta_2 \, d \vartheta_3 }\prs{ \int_0^{\infty} \tfrac{64 \prs{\la^2 -1}^2 r^9}{\prs{r^2+1}^6 \prs{\la^2 r^2+1}^2} \, dr} \\
&= \tfrac{217 \pi^2}{32} \prs{ \int_0^{\infty} \tfrac{64 \prs{\la^2 -1}^2 r^9}{\prs{r^2+1}^6 \prs{\la^2 r^2+1}^2} \, dr}\\
 &\leq  2 \cdot 217 \pi^2 \prs{\la^2 -1}^2 \prs{\int_0^{1/\la} r^9 \, dr +  \int_{1/\la}^1  \tfrac{r^5}{\la^4} \, dr + \int_1^{\infty}  \tfrac{1}{\la^4 r^7} \, dr} \\
&=  2 \cdot 217 \pi^2 \prs{ \tfrac{\la - 1}{\la}}^2\prs{ \tfrac{\la + 1}{\la}}^2 \prs{\tfrac{1}{3} - \tfrac{1}{15 \la^6}} \\
&=  \tfrac{2^3}{15} \cdot 217 \pi^2 \prs{ \tfrac{\la - 1}{\la}}^2\prs{ \tfrac{\la + 1}{\la}}^2.
\end{align*}
Using \eqref{eq:gGstuff} above we have that
\begin{align*}
\brs{ \gG^r \del_r \chi_{\la}}^2 = r^2 \prs{1+\sin^2 \vartheta_1+\sin^2 \vartheta_1 \sin^2 \vartheta_2 }^2 \brs{\tfrac{\del}{\del r} \chi_{\la}}^2.
\end{align*}
With this we compute
\begin{align*}
\brs{\brs{ \gG^r \del_r \chi_{\la}}}_{L^2}^2 &= 2 \pi \prs{\int_0^\pi \int_0^{\pi} \prs{1+\sin^2 \vartheta_1 + \sin^2 \vartheta_1 \sin^2 \vartheta_2}^2  \, d\vartheta_1 \, d\vartheta_2} \prs{\int_0^{\infty} r^5 \brs{\chi_{\la}}^2 \, dr}\\
&= \tfrac{\pi}{64} \int_0^{\pi}\prs{-68 \cosh \prs{2 \vartheta_2} + 3 \cosh(4 \vartheta_2) + 217} \, d \vartheta_2 \prs{ \int_0^{\infty} r^5 \brs{\chi_{\la}}^2 \, dr} \\
&= \tfrac{217 \pi^2}{64} \int_0^{\infty} r^5 \brs{\chi_{\la}}^2 \, dr \\
&= 217 \pi^2 (\la^2-1)^2 \prs{\int_0^{\infty} \tfrac{r^7}{\prs{1+\la^2 r^2}^2 \prs{1+r^2}^6} \, dr } \\
&= 217  \pi^2 (\la^2-1)^2 \prs{\int_0^{1/\la} r^7 \, dr + \tfrac{1}{\la^4} \int_{1/\la}^1 r^3 \, dr + \tfrac{1}{\la^4} \int_1^{\infty} r^{-9} \, dr} \\
&=\tfrac{217}{8} \pi^2 \prs{\tfrac{\la-1}{\la}}^2 \prs{\tfrac{\la +1}{\la}}^2  \prs{\tfrac{3}{8} -\tfrac{1}{8 \la^4} }\\
&\leq \tfrac{217}{2^7} \pi^2 \prs{\tfrac{\la-1}{\la}}^2\prs{\tfrac{\la +1}{\la}}^2.
\end{align*}
%

%
To all of these estimates, we note that $\prs{\tfrac{\la + 1}{\la}} \leq 1 + \tfrac{1}{\la} \leq 2$. It is a standard fact that
\begin{equation*}
\frac{\la - 1}{\la} \leq 
\begin{cases}
\prs{\log \la} \text{ when }\la \in [1,\infty), \\
\prs{\log \la}^{1/2} \text{ when }\la \in [e,\infty).
\end{cases}
\end{equation*}
Applying these, we conclude \eqref{eq:logchila}.
The result follows.
\end{proof}
\end{lemma}
\subsection{\boldmath${\ga}$-connection concentration compactness result}
%
%
%
%

\begin{thm}\label{thm:compactstate}
Let $\sqg{\ga_i} \subset [1,2)$ with $\lim_{i \to \infty} \ga_i = 1$. Given corresponding Yang-Mills $\ga_i$-energy minimizing connections $\sqg{\N^{i}}$, there exists sequences $\sqg{ \vp_i} \subset \SO(5,1)$ and $\sqg{\gs_i} \subset \mathcal{G}_{E}$ such that there is a subsequence $\sqg{\sigma_{i_j} \brk{ \vp_{i_j}^* \N^{i_j} }}$ which converges strongly in $C^{\infty}$ to a antiself dual connection $\N^{\infty}$.
\begin{proof} First, assume that the pointwise curvature norms $\brs{F_{\N^{i}}}_{\mathring{g}}$ do not concentrate as $i \to \infty$. If so, then the derivatives of curvature are also controlled via the $\ge$-regularity and derivative estimates results of \cite{HTY} (Lemmata 3.5 and 3.6) (note that these results are for the Yang-Mills $\ga$-flow, but for our purposes we can apply them assuming the stationary setting). Thus up to gauge transformation, the sequence $\sqg{\sigma_{i_j} \brk{ \vp_{i_j}^* \N^{i_j} }}$ converges to a minimal energy critical value of the Yang-Mills energy, which implies antiself duality (cf. Proposition \ref{prop:W12AE}).

If the pointwise curvature norms concentrate as $i \to \infty$, then we do a maximal blowup along the sequence, to identify a sequence of points $\sqg{\zeta_{i'}}$ with $\zeta_{i'}$ admitting supremal pointwise curvature norm for all $k \leq i'$. There exists a further subsequence $\sqg{\zeta_{i''}}$ converging to $\zeta_{\infty} \in \mathbb{S}^4$, so that
\begin{equation*}
\lim_{i'' \to \infty}\brs{F_{\N^{i''}}\prs{\zeta_{i''}}}_{\mathring{g}} = \infty.
\end{equation*}
Stereographically projecting $\mathbb{S}^4$ onto $\mathbb{H}^1$, with $\zeta_{\infty}$ as the center point, we see that on $\mathbb{H}^1$, dilation centered at the origin is equivalent to performing a conformal automorphism on $\mathbb{S}^4$. Consequently, we can normalize the curvature via dilations in the blowup so that one has (identifying back to the corresponding setting on $\mathbb{S}^4$),
\begin{equation*}
\lim_{i'' \to \infty} \brs{ F_{\vp_{i''}^* \N^{i''}} \prs{\zeta_{i''}}}_{\mathring{g}} = 1.
\end{equation*}
This modified sequence satisfies the initial case, above, namely that the pointwise norms do not concentrate, which concludes the result.
\end{proof}
\end{thm}
\subsection{Poincar\'{e} Inequalities}\label{ss:Poincare}
In this section we will compute global and localized $\tN$-Poincar\'{e} inequalities which rely heavily on the strong structure of the $\tN$ and its corresponding curvature. To do so, we first need to establish Lemma \ref{eq:commutatorbd}, namely, pointwise bounds on commutator type terms.
\begin{proof}[Proof of Lemma \ref{eq:commutatorbd}]
Using the formula for $\widetilde{F}_{ij}$ written out in coordinates, identifying $\Ad E$ with $\Im \mathbb{H}^1$ (as discussed in \cite{Naber}), we see that on $\mathbb{H}^1$,
\begin{align*}
F_{\tN} = \tfrac{2}{\prs{1 + \brs{\zeta}^2}^2} \prs{\prs{ d\zeta^{12} - d\zeta^{34} }\mathbf{i} + \prs{ d\zeta^{13} + d\zeta^{24} }\mathbf{j} + \prs{ d\zeta^{14} - d\zeta^{23} }\mathbf{k}  },
\end{align*}
where here $d\zeta^{ij} := d \zeta^i \w d \zeta^j$. We aim to compute
\begin{align*}
\ip{\widetilde{F}_{ij}, \brk{A_i, A_j}} = \widetilde{F}_{ij}^{\mathbf{i}} \brk{A_i, A_j}^{\mathbf{i}} + \widetilde{F}_{ij}^{\mathbf{j}} \brk{A_i, A_j}^{\mathbf{j}}+ \widetilde{F}_{ij}^{\mathbf{k}} \brk{A_i, A_j}^{\mathbf{k}}.
\end{align*}
Computing strictly in coordinates of $\mathbb{H}^1$, we compute the following.
\begin{align*}
\brk{A_i, A_j}^{\mathbf{i}} &= 2 \prs{ A_i^{\mathbf{j}} A_j^{\mathbf{k}} - A_i^{\mathbf{k}} A_j^{\mathbf{j}} }.
\end{align*}
Then we have that
\begin{align*}
\widetilde{F}_{ij}^{\mathbf{i}} \brk{A_i,A_j}^{\mathbf{i}} = \tfrac{4}{\prs{1+\brs{\zeta}^2}^2} \prs{ \prs{A_1^{\mathbf{j}} A_2^{\mathbf{j} } -  A_2^{\mathbf{j}} A_1^{\mathbf{j} } } - \prs{ A_3^{\mathbf{j}} A_4^{\mathbf{j} } -  A_4^{\mathbf{j}} A_3^{\mathbf{j} } }}.
\end{align*}
We have that for $\mathbf{l},\mathbf{m} \in \Im \mathbb{H}$ then $2 \brs{A_i^{\mathbf{l}} A_j^{\mathbf{m}}} \leq \brs{A_i^{\mathbf{l}}}^2 + \brs{A_j^{\mathbf{m}}}^2$, so we have
\begin{equation*}
\brs{\widetilde{F}_{ij}^{\mathbf{i}} \brk{A_i,A_j}^{\mathbf{i}}}  \leq  \tfrac{2}{\prs{1 + \brs{\zeta}^2}^2} \prs{ \brs{A_1^{\mathbf{j}}}^2 + \brs{A_1^{\mathbf{k}}}^2 +  \brs{A_2^{\mathbf{j}}}^2 + \brs{A_2^{\mathbf{k}}}^2 + \brs{A_3^{\mathbf{j}}}^2 + \brs{A_3^{\mathbf{k}}}^2 + \brs{A_4^{\mathbf{j}}}^2 + \brs{A_4^{\mathbf{k}}}^2  }.
\end{equation*}
Consequently it follows that (applying the round metric)
%
%
%
\begin{align*}
\brs{\ip{\widetilde{F}_{ij}, \brk{A_i,A_j}}_{\mathring{g}} } & \leq \brs{A}_{\mathring{g}}^2.
\end{align*}
The second inequality of \eqref{eq:Fcomm} follows similarly (noting the contraction within the commutator adds an extra dimensional factor of $4$).
\end{proof}
\begin{prop}[Localized $\tN$-Poincar\'{e} inequalities]\label{prop:poinlocal} For $R > 0$, $\ell \in \mathbb{N}$, and $A \in \prs{\Lambda^1 \prs{B_R} \ten \Ad E}$ where $B_R \subset \mathbb{S}^4$ there exists $C_P > 0$ such that
\begin{equation}\label{eq:NgYineq}
\int_{B_R} \brs{ A}^{\ell}_{\mathring{g}} \, dV_{\mathring{g}}\leq C_P R^{\ell} \int_{B_R} \brs{\tN A}^{\ell}_{\mathring{g}} \, dV_{\mathring{g}}, \text{ and } \int_{B_R} \brs{\tN A}^{\ell}_{\mathring{g}} \, dV_{\mathring{g}} \leq C_P R^{\ell} \int_{B_R} \brs{\tN^{(2)} A}^{\ell}_{\mathring{g}} \, dV_{\mathring{g}}.
\end{equation}
\begin{proof}
We provide a proof by contradiction. If the inequality above were false, we can find a normalized sequence $\sqg{A^i}$ satisfying
\begin{equation*}
\int_{B_R} \brs{\tN A^i}_{\mathring{g}}^{\ell} \, dV_{\mathring{g}} \to 0, \quad \int_{B_R} \brs{ A^i}^{\ell}_{\mathring{g}} \, dV_{\mathring{g}} = 1.
\end{equation*}
Via theorems of Rellich and Banach-Alaoglu, we choose a normalized subsequence $\sqg{ A^{i'} }$ satisfying
\begin{equation*}
A^{i'} \overset{L^p}{\rightarrow} A, \quad A^{i'} \overset{W^{1,p}}{\rightharpoonup} A \text{ and so }\tN A \equiv 0, \text{ and } \brs{\brs{ A }}_{L^{\ell}\prs{\Lambda^1 ( B_R) \ten \Ad E}} \equiv 1.
\end{equation*}
It follows from the Ambrose-Singer Theorem (\cite{AS} Theorem 2) that this cannot hold. Implicitly the theorem relates the curvature of the connection to its holonomy. Thus, if one finds a local parallel section of $\Lambda^1 \prs{B_R} \ten \Ad E$ in a neighborhood of some point, the holonomy is reduced, which is a contradiction since $\tN$ has full holonomy. This concludes the first inequality of \eqref{eq:NgYineq}.

For the second inequality of \eqref{eq:NgYineq}, we again perform a proof by contradiction and construct a normalizing sequence $\sqg{A_j}$, this time satisfying
\begin{align*}
\int_{B_R} \brs{ \tN^{(2)} A_j }^2_{\mathring{g}} \, dV_{\mathring{g}} \to 0, \quad \int_{B_R} \brs{\tN A}^2_{\mathring{g}} \, dV_{\mathring{g}} = 1.
\end{align*}
Again by Rellich's and Banach-Alaoglu's theorems there is a further subsequence $\sqg{A_{j'}}$ such that
\begin{align*}
A_{j'} \overset{W^{1,2}}{\to} A, \quad A_{j'} \rightharpoonup A \text{ and so } \tN^{(2)} A \equiv 0, \text{ and } \brs{\brs{ \tN A }}_{L^{\ell} \prs{\prs{\Lambda^1 \prs{B_R}}^{\ten 2} \ten \Ad E}} \equiv 1.
\end{align*}
In particular $\tN^{(2)} A \equiv 0$. Since this is true on the coordinate level, we also have
\begin{align*}
0 = \tN_i \tN_j A_k - \tN_j \tN_i A_k = \brk{\tN_i, \tN_j} A_k. 
\end{align*}
In particular, this implies that (using \eqref{eq:commutatorbd})
\begin{align*}
0 &= \ip{\brk{\tN_i, \tN_j } A_i, A_j}_{\mathring{g}} = -3 \brs{A}_{\mathring{g}}^2 + \ip{\brk{\widetilde{F}_{ij}, A_i}, A_j}_{\mathring{g}} < 0,
\end{align*}
an obvious contradiction.
\end{proof}
\end{prop}
\noindent Note that Proposition \ref{prop:poincare} follows naturally from a simple covering argument over $\mathbb{S}^4$.

\subsection{Properties of $\tN$-Coulomb gauge }

Here we include a proof of Theorem \ref{thm:curvcontrol}, an global adaptation of Tao and Tian's local result (\cite{TT} Theorem 4.6) which in turn was inspired by Theorem 1.3 of \cite{Uhlenbeck1}. Set $K > 1$ to be an absolute constant we to be chosen later, and define two sets
\begin{align*}
\mathcal{U}_{\ge} &:= \sqg{\N \in \mathcal{A}_E\prs{\mathbb{S}^4} : \inf_{\vs \in \mathcal{G}_E} \brs{\brs{ F_{\vs \brk{\N}} - F_{\tN} } }_{L^2}\leq \ge }\\
\mathcal{U}_{\ge}^* &:= \sqg{\N \in \mathcal{U}_{\ge}\prs{\mathbb{S}^4} : \brs{\brs{ \widetilde{\Pi} \brk{\N} - \tN } }_{W^{1,2}}\leq K \ge  }.
\end{align*}
Our goal is to show $\mathcal{U}^*_{\ge} \equiv \mathcal{U}_{\ge}$, thus establishing Theorem \ref{thm:curvcontrol}. The preliminary step needed in the proof is a bootstrap estimate. Throughout the proof, we warn the reader to be cautious of the meaning of $\gY$, as it changes periodically throughout the argument (it will either be $\N - \tN$ \emph{or} $\widetilde{\Pi}\brk{\N} - \tN$).
\begin{lemma}[Bootstrap estimate]\label{lem:bootstrap}
For any $\N \in \mathcal{U}_{\ge}^*$, the following estimate may be bootstrapped
\begin{equation*}
\brs{\brs{ \widetilde{\Pi} \brk{\N} - \tN}}_{L^{2}} \leq K \ge,
\end{equation*}
to instead obtain
\begin{equation*}
\brs{\brs{ \widetilde{\Pi} \brk{\N} - \tN}}_{L^{2}} \leq K \tfrac{\ge}{2},
\end{equation*}
additionally, \eqref{eq:bootstrap} holds.
\begin{proof} Set $\gY :=\widetilde{\Pi}\brk{\N} - \tN$. Via Proposition \ref{prop:poincare} and Proposition \ref{prop:curvpol},
\begin{align*}
\brs{\brs{ \gY }}_{L^{2}}  &\leq C\brs{\brs{ D_{\tN} \gY }}_{L^{2}} \\ &= C \brs{\brs{ F_{\widetilde{\Pi}\brk{\N}} - F_{\tN} + \brk{\gY,\gY}}}_{L^{2}} \\
&\leq C \brs{\brs{ F_{\widetilde{\Pi}\brk{\N}} - F_{\tN} }}_{L^{2}}+ C \brs{\brs{\gY } }^2_{L^{2}} \\
&\leq C \ge + C \brs{\brs{\gY } }^2_{L^{2}}.
\end{align*}
Applying the estimates yields the desired results.
\end{proof}
\end{lemma}
Now take $\N \in \mathcal{U}_{\ge}$ and consider the one-parameter family of connections for $s \in \brk{0,1}$ by
\begin{align*}
\N_s \prs{x} :=\tN \prs{x} + s \prs{ \N - \tN}\prs{sx}.
\end{align*}
One can see that for $s = 0$, $\N_0 \equiv \tN$ and $\N_1 \equiv \N$. We next verify this entire family lies inside of $\mathcal{U}_{\ge}$.

\begin{prop}Using the notation above $\N_s \in \mathcal{U}_{\ge}$.
\begin{proof}
Take $\gY_s := \N_s - \tN$ and observe that
\begin{align*}
\left[ F_{\N_s} - F_{\tN} \right|_{x} &= \left[ D_{\tN} \prs{\gY_s } + \brk{\gY_s, \gY_s } \right|_{x}\\
&= s\left[  D_{\tN} \prs{\gY } + \brk{\gY, \gY } \right|_{sx}\\
&= s \left[ F_{\N} - F_{\tN} \right|_{sx}.
\end{align*}
Consequently
\begin{align*}
\brs{\brs{ F_{\N_s} - F_{\tN} }}_{L^{2}} &= s\brs{\brs{ F_{\N} - F_{\tN} }}_{L^{2}} \leq \tfrac{s \ge}{2},
\end{align*}
as desired.
\end{proof}
\end{prop}
\begin{prop}[Continuity of the Coulomb gauge construction in smooth norms]\label{prop:Coulgauge} Let $X \in \prs{0,\infty}$, $p \in \prs{2,4}$ and let $\N \in \mathcal{U}_{\ge}$ be such that
\begin{equation}\label{eq:TT92}
\brs{\brs{ \N - \tN }}_{W^{1,p}} \leq X.
\end{equation}
Then there exists a quantity $\delta_X > 0$ depending only on $X, \mathcal{G}_{E}, p, \ge$, such that
\begin{equation*}
\sqg{ \N + A \in \mathcal{U}_{\ge} : \brs{\brs{ A }}_{W^{1,p}} \leq \delta_X } \subset \mathcal{U}_{\ge}^*.
\end{equation*}
\begin{proof}
Fix $p$ (all constants are allowed to depend on $p$), and let $C_X \geq 0$ denote quantities dependent on $X$, which can be updated as necessary. As in \cite{TT}, the argument consists of three steps.\\

\noindent \textbf{Step 1. Estimation of the $\tN$-Coulomb gauge in smooth norms.} Note that via Proposition \ref{prop:curvpol} combined with H\"{o}lder's inequality, there exists a constant $C_X \geq 0$ such that
\begin{equation*}
\brs{\brs{ F_{\N} - F_{\tN} }}_{L^p} \leq C_X.
\end{equation*}
This constant $C_X$ can be updated as necessary to bound above
\begin{equation*}
\brs{\brs{ F_{\widetilde{\Pi}\brk{ \N }} - F_{\tN} }}_{L^p} \leq C_X.
\end{equation*}
Manipulating as in the lemma above,
\begin{align*}
\brs{\brs{ \widetilde{\Pi} \brk{\N} - \tN }}_{W^{1,p}} &\leq \brs{\brs{ F_{\widetilde{\Pi} \brk{\N}} - F_{\tN} }}_{L^p} + C \brs{\brs{ \prs{\widetilde{\Pi} \brk{\N} - \tN } \w \prs{\widetilde{\Pi} \brk{\N} - \tN } } }_{W^{1,p}}
 \leq C_X.
\end{align*}
Let $\vs$ denote the gauge transformation such that $\vs \brk{ \N } = \widetilde{\Pi} \brk{ \N }  = \del + \Sigma$. Then
\begin{equation*}
\Sigma_{i \theta}^{\gb} := \prs{\vs^{-1}}_{\gd}^{\gb} \prs{\del_i \vs_{\theta}^{\delta}} + \prs{\vs^{-1}}_{\delta}^{\gb} \gG_{i \gg}^{\gd} \vs_{\theta}^{\gamma}.
\end{equation*}
Remanipulating, and setting $\Upsilon := \vs \brk{ \N } - \tN$ and $\text{Y} := \N - \tN$  yields
\begin{align*}
\prs{\del_i \vs_{\theta}^{\rho}} &= \vs_{\gb}^{\rho}\Sigma_{i \theta}^{\gb} - \gG_{i \gg}^{\rho} \vs_{\theta}^{\gamma} \\
\prs{\tN_i \vs_{\theta}^{\rho}} &= \vs_{\gb}^{\rho}\Sigma_{i \theta}^{\gb} - \gG_{i \gg}^{\rho} \vs_{\theta}^{\gamma} - \vs_{\mu}^{\rho} \widetilde{\gG}_{i \theta}^{\mu}+ \widetilde{\gG}_{i \mu}^{\rho} \vs_{\theta}^{\mu} =  \vs_{\gb}^{\rho}\gY_{i \theta}^{\gb} - \text{Y}_{i \mu}^{\rho} \vs_{\theta}^{\mu}.
\end{align*}
Now we note that
\begin{align*}
\tN_j \tN_i \vs_{\theta}^{\rho} &= \prs{\tN_j \vs_{\gb}^{\rho} } \gY_{i \theta}^{\gb} +\vs_{\gb}^{\rho}  \prs{\tN_j  \gY_{i \theta}^{\gb}} - \prs{ \tN_j \text{Y}_{i \mu}^{\rho}} \vs_{\theta}^{\mu} - \text{Y}_{i \mu}^{\rho} \prs{ \tN_j \vs_{\theta}^{\mu}}\\
&= \prs{\vs_{\gg}^{\rho} \gY_{j \gb}^{\gg}- \text{Y}_{j \gg}^{\rho} \vs_{\gb}^{\gg}
 } \gY_{i \theta}^{\gb} +\vs_{\gb}^{\rho}  \prs{\tN_j  \gY_{i \theta}^{\gb}} - \prs{ \tN_j \text{Y}_{i \mu}^{\rho}} \vs_{\theta}^{\mu} - \text{Y}_{i \mu}^{\rho} \prs{ \vs_{\zeta}^{\mu} \gY_{j \theta}^{\zeta} - \text{Y}_{j \zeta}^{\mu} \vs_{\theta}^{\zeta}}.
\end{align*}
We thus have that
\begin{align*}
\brs{\tN^{(2)} \vs}_{\mathring{g}} &\leq \brs{\vs}_{\mathring{g}} \prs{ \brs{\gY}_{\mathring{g}} + \brs{\text{Y}}_{\mathring{g}} \brs{\gY}_{\mathring{g}}} + \brs{\vs}_{\mathring{g}} \prs{ \brs{\tN \gY}_{\mathring{g}} + \brs{\tN \text{Y}}_{\mathring{g}}} \\
&\leq C \brs{\vs}_{\mathring{g}} \prs{ \brs{\gY}_{\mathring{g}} +  \brs{\gY}^2_{\mathring{g}} + \brs{\tN \gY}_{\mathring{g}} + \brs{\text{Y}}^2_{\mathring{g}}\brs{\tN\text{Y}}_{\mathring{g}}}.
\end{align*}
Consequently we have
\begin{align*}
\brs{\vs}_{\mathring{g}} + \brs{ \tN \vs}_{\mathring{g}}  + \brs{ \tN^{(2)} \vs}_{\mathring{g}} \leq C \brs{\vs}_{\mathring{g}} \prs{ \brs{\gY}_{\mathring{g}} +  \brs{\gY}_{\mathring{g}}^2 + \brs{\tN \gY}_{\mathring{g}}  + \brs{\text{Y}}_{\mathring{g}} + \brs{\text{Y}}_{\mathring{g}}^2+ \brs{\tN \text{Y}}_{\mathring{g}}}.
\end{align*}
Combining these all together we conclude that $\brs{\brs{ \vs }}_{W^{2,p}} \leq C_X$, concluding the first step.\\

\noindent \textbf{Step 2. Pass to the $\tN$-Coulomb gauge.} As a consequence of \textbf{Step 1}, it follows that the action of a $\tN$-Coulomb gauge transformation is uniformly continuous in the sense of the $W^{2,p}$-topology in a small neighborhood of $\N$. Furthermore, since both $\mathcal{U}_{\ge}$ and $\mathcal{U}_{\ge}^*$ are in fact invariant under gauge transformation (the $\tN$-projection always overrides any gauge action) we can prove Proposition \ref{prop:Coulgauge} specifically in the setting $\N = \widetilde{\Pi} \brk{ \N}$. \\

\noindent \textbf{Step 3. Apply perturbation theory to the Coulomb gauge.} 
Fix the perturbation parameter $A$ as in Proposition \ref{prop:Coulgauge}. To show $\N + A \in \mathcal{U}_{\ge}^*$, we must construct a gauge transformation $\vs$ satisfying
\begin{equation}\label{eq:Ncoulgauge}
D_{\tN}^* \prs{ \vs \brk{ \N + A}- \tN } = 0.
\end{equation}
To do so, we give a perturbative argument. Set $\vs := e^{\gs}$, recall the formula \eqref{eq:commutator} of a gauge action on a connection.
We have that $\vs \brk{ \N + A } - \tN$ is given by
\begin{align*}
\prs{\vs \brk{ \N + A}- \tN}_{i \theta}^{\gb}  &= \prs{\vs^{-1}}_{\gd}^{\gb} \prs{\del_i \vs_{\theta}^{\gd}} + \prs{\vs^{-1}}_{\gd}^{\gb} \brk{ \gG - A}_{i \gg}^{\gd} \vs_{\theta}^{\gg} - \widetilde{\gG}_{i \theta}^{\gb} \\
&= \prs{\vs^{-1}}_{\gd}^{\gb} \prs{\tN_i \vs_{\theta}^{\gd}} -\prs{\vs^{-1}}_{\gd}^{\gb} \widetilde{\gG}_{i \tau}^{\gd} \vs_{\theta}^{\tau} + \prs{\vs^{-1}}_{\gd}^{\gb} \brk{ \gG + A }_{i \gg}^{\gd} \vs^{\gg}_{\theta} \\
&= \prs{\vs^{-1}}_{\gd}^{\gb} \prs{\tN_i \vs_{\theta}^{\gd} } + \prs{\vs^{-1}}_{\gd}^{\gb} \prs{\gY + A }_{i \gg}^{\gd} \vs_{\theta}^{\gg}.
\end{align*}
Converting this to be in terms of $\gs$ (and simultaneously defining our term $\mathcal{W}$) we have
\begin{align}\label{eq:WSNA}
\begin{split}
\mathcal{W} \prs{\gs, \N + A} &= \prs{ \vs \brk{\N + A} - \tN}_{i \theta}^{\gb} - \tN_i \gs_{\theta}^{\gb}\\
&= \prs{e^{-\gs}}_{\gd}^{\gb} \prs{\gY + A}_{i \gg}^{\gd} \prs{e^\gs}_{\theta}^{\gg} \\
&= \prs{\Id^{\gb}_{\gd} - \gs_{\gd}^{\gb} + \cdots} \prs{\gY + A}_{i \gg}^{\gd} \prs{\Id^{\gb}_{\gd} + \gs_{\gd}^{\gb}  + \cdots} \\
&= \prs{\gY + A}_{i\theta}^{\gb} - \gs_{\gd}^{\gb}\prs{\gY + A}_{i\theta}^{\gd}  + \prs{\gY + A}_{i\gd}^{\gb}\gs_{\theta}^{\gd} + \gs_{\gz}^{\gb} \prs{\gY + A}_{i\gd}^{\gz}\gs_{\theta}^{\gd} + \cdots.
\end{split}
\end{align}
Expanding out $0 = D^*_{\tN}\prs{\mathcal{W}\prs{\gs,\N +A} + \N \gs }$ gives
\begin{align*}
0 &=- \tN_i \brk{ \prs{\vs^{-1}}_{\gg}^{\gb} \prs{\del_i \vs_{\theta}^{\gg}} + \prs{\vs^{-1}}_{\delta}^{\gb} \prs{\gG + A}_{i \gg}^{\gd} \prs{\vs_{\theta}^{\gg}}  } \\
&=- \tN_i \brk{ \prs{\vs^{-1}}_{\gg}^{\gb} \prs{\tN_i \vs_{\theta}^{\gg} - \widetilde{\gG}_{i \gz}^{\gg} \vs^{\gz}_{\theta} + \vs^{\gg}_{\gz}  \widetilde{\gG}_{i \theta}^{\gz} } + \prs{\vs^{-1}}_{\delta}^{\gb} \prs{\gG + A}_{i \gg}^{\gd} \prs{\vs_{\theta}^{\gg}}  } \\
&= - \prs{\vs^{-1}}_{\gg}^{\gb} \widetilde{\lap} \vs_{\theta}^{\gg} + \prs{\vs^{-1}}_{\gw}^{\gb}\prs{\tN_i \vs_{\delta}^{\gw}}   \prs{\vs^{-1}}_{\gg}^{\delta} \prs{\tN_i \vs_{\theta}^{\gg}} + \prs{\vs^{-1}}_{\gg}^{\gb} \widetilde{\gG}_{i \gz}^{\gg} \prs{\tN_i \vs_{\theta}^{\gz}} - \prs{\vs^{-1}}_{\gg}^{\gb} \prs{\tN_i \vs_{\gz}^{\gg}} \widetilde{\gG}_{i \theta}^{\gz}\\
& \hsp + \prs{\vs^{-1}}_{\mu}^{\gb}\prs{ \tN_i \vs^{\mu}_{\nu}} \prs{\vs^{-1} }_{\delta}^{\nu}\prs{\gG + A}_{i \gg}^{\gd} \vs_{\theta}^{\gg} -  \prs{\vs^{-1}}_{\gd}^{\gb}  \tN_i \prs{\gG + A}_{i \gg}^{\gd} \vs_{\theta}^{\gg} - \prs{\vs^{-1}}_{\gd}^{\gb}\prs{\gG + A}_{i \gg}^{\gd}\prs{\tN_i \vs_{\theta}^{\gg}}.
\end{align*}
Inserting the fact that $\vs \equiv e^\gs$, we have that
\begin{align*}
0 &= -  \widetilde{\lap} \gs_{\theta}^{\gb} +\prs{\tN_i \gs_{\delta}^{\gb}} \prs{\tN_i \gs_{\theta}^{\delta}} + \prs{\exp (-\gs)}_{\gg}^{\gb} \widetilde{\gG}_{i \gz}^{\gg} (e^\gs)^{\gz}_{\gw} \prs{\tN_i \gs_{\theta}^{\gw}} -\prs{\tN_i \gs_{\gz}^{\gb}} \widetilde{\gG}_{i \theta}^{\gz} \\
& \hsp + \prs{ \tN_i \gs^{\mu}_{\gb}} \prs{ e^{-\gs} }_{\delta}^{\nu}\prs{\gG + A}_{i \gg}^{\gd} \prs{e^\gs}_{\theta}^{\gg} -  \prs{e^{-\gs}}_{\gd}^{\gb}  \tN_i \prs{\gG + A}_{i \gg}^{\gd} \prs{e^{\gs}}_{\theta}^{\gg}\\
& \hsp - \prs{e^{-\gs}}_{\gd}^{\gb}\prs{\gG + A}_{i \gg}^{\gd} \prs{e^\gs}^{\gg}_{\gz}\prs{\tN_i \gs_{\theta}^{\gz}}\\
&= -  \widetilde{\lap} \gs_{\theta}^{\gb} +\prs{\tN_i \gs_{\delta}^{\gb}} \prs{\tN_i \gs_{\theta}^{\delta}} + \prs{ \tN_i \gs^{\mu}_{\gb}}\brk{ \prs{ e^{-\gs} }_{\delta}^{\nu}\prs{\gG + A}_{i \gg}^{\gd} \prs{e^\gs}_{\theta}^{\gg}  - \widetilde{\gG}_{i \theta}^{\mu}} \\
& \hsp-  \prs{e^{-\gs}}_{\gd}^{\gb}  \tN_i \prs{\gY + A}_{i \gg}^{\gd} \prs{e^\gs}_{\theta}^{\gg} - \prs{e^{-\gs}}_{\gd}^{\gb} A_{i \gg}^{\gd} \prs{e^\gs}^{\gg}_{\gz}\prs{\tN_i \gs_{\theta}^{\gz}}.
\end{align*}
Therefore we have that,
\begin{align}\label{eq:lapSeqn}
\begin{split}
\widetilde{\lap} \gs_{\theta}^{\gb} &= \prs{\tN_i \gs_{\delta}^{\gb}} \prs{\tN_i \gs_{\theta}^{\delta}} + \prs{ \tN_i \gs_{\nu}^{\gb}}\brk{ \prs{ e^{-\gs} }_{\delta}^{\nu}\prs{\gY + A}_{i \gg}^{\gd} \prs{e^{\gs}}_{\theta}^{\gg} + \prs{ e^{-\gs} }_{\delta}^{\nu} \widetilde{\gG}_{i \gg}^{\gd} \prs{e^\gs}_{\theta}^{\gg}  - \widetilde{\gG}_{i \theta}^{\nu}}\\
& \hsp  -  \prs{e^{-\gs}}_{\gd}^{\gb}  \tN_i A_{i \gg}^{\gd} \prs{e^{\gs}}_{\theta}^{\gg} - \prs{e^{-\gs}}_{\gd}^{\gb}\prs{\gY + A}_{i \gg}^{\gd} \prs{e^\gs}^{\gg}_{\gz}\prs{\tN_i \gs_{\theta}^{\gz}}.
\end{split}
\end{align}
By combining \eqref{eq:WSNA} and \eqref{eq:Ncoulgauge}, and noting that here, $D^*_{\tN} \brk{\mathcal{W} \prs{\gs^{(j)},\N + A}}$ is precisely the right hand side of \eqref{eq:lapSeqn}, we conclude
\begin{equation*}
\widetilde{\lap} \gs = D_{\tN}^* \brk{\mathcal{W} \prs{\gs, \N + A}}.
\end{equation*}
Consider the following iteration scheme, with initial condition $\gs^{(0)} \equiv 0$,
\begin{align*}
\widetilde{\lap} \gs^{(\ell+1)} &:=  D_{\tN}^* \brk{\mathcal{W} \prs{\gs^{(\ell)}, \N + A}}.
\end{align*}
Note $\gs^{(\ell+1)}$ is uniquely defined by standard elliptic regularity. We will derive bounds on $\gs^{(\ell+1)}$. First, based off of the system above, we have
\begin{equation*}
\brs{\brs{ \gs^{(\ell+1)}}}_{W^{2,p}} \leq C \brs{\brs{ D_{\tN}^* \brk{\mathcal{W}\prs{\gs^{(\ell)}, \N + A}} }}_{W^{2,p}} + C \brs{\brs{\mathcal{W}\prs{\gs^{(\ell)},\tN + A}}}_{W^{1,p}}
\end{equation*}
We estimate each term on the right hand side. For both terms, we apply the exponential power series expansion for some $C > 0$ depending on $\tN$. First by \eqref{eq:WSNA},
\begin{align*}
\brs{ \mathcal{W} \prs{\gs^{(\ell)},\tN +  A} }_{\mathring{g}} \leq C \brs{A}_{\mathring{g}} \prs{1 + \brs{\gs^{(\ell)}}_{\mathring{g}} }.
\end{align*}
From this we can deduce, using H\"{o}lder's inequality, that
\begin{align*}
\brs{\brs{ \mathcal{W} \prs{\gs^{(\ell)},\tN +  A}} }_{W^{1,p}} \leq C \prs{ \brs{\brs{\gs^{(\ell)}}}^2_{W^{2,p}} + \brs{\brs{A}}_{W^{1,p}} }.
\end{align*}
Using \eqref{eq:lapSeqn} above, we see that
\begin{align*}
\brs{ D_{\tN}^* \brk{\mathcal{W}\prs{\gs^{(\ell)}, \N + A}} }_{\mathring{g}}
&\leq C \prs{ \brs{\tN \gs^{(\ell)}}_{\mathring{g}}^2 + \prs{1 + \brs{\gs^{(\ell)}}_{\mathring{g}} }  \prs{\brs{\tN\gs^{(\ell)}}_{\mathring{g}} \prs{ \brs{\gY + A}_{\mathring{g}} + 1 } + \brs{\tN A}_{\mathring{g}}}}.
\end{align*}
Therefore we can expand out with Holder's inequality, noting that $\N \in \mathcal{U}_{\ge}^*$ and conclude that
\begin{align*}
\brs{\brs{D^*_{\tN}\mathcal{W}\prs{\gs^{(\ell)}, \N + A} }}_{W^{1,p}} \leq  C \brs{\brs{\gs^{(\ell)}}}_{W^{2,p}}^2 + K \ge \brs{\brs{ \gs^{(\ell)} }}_{W^{2,p}} + \brs{\brs{ A }}_{W^{2,p}}.
\end{align*}
%
%
Thus, as long as $\delta_X$ is sufficiently small, we can obtain inductively that
\begin{equation*}
\brs{\brs{ \gs^{(\ell)} }}_{W^{2,p}} \leq C \delta_X. 
\end{equation*}
We adapt this iteration scheme to conclude that $\gs^{(\ell)}$ converges in $W^{2,p}$ to a solution $\gs$ satisfying
\begin{align*}
\brs{\brs{ \gs}}_{W^{2,p}} \leq C \delta_X.
\end{align*}
As a result of standard Sobolev embeddings, $\gs$ has some H\"{o}lder regularity, which, when elliptic regularity is applied, can be bootstrapped to conclude that $\gs$ is in fact smooth. If we exponentiate $\gs$ and apply H\"{o}lder's inequality we obtain a smooth $\tN$-Coulomb gauge $\vs \brk{\N + A}$ satisfying
\begin{equation*}
\brs{\brs{\vs - \Id}}_{W^{2,p}} , \brs{\brs{\vs^{-1} - \Id}}_{W^{2,p}} \leq C \delta_X.
\end{equation*}
As a consequence of the gauge transformation action and \eqref{eq:TT92} of the assumptions on $A$,
\begin{align*}
\brs{\brs{ \vs \brk{ \N + A} - \N }}_{W^{1,p}} \leq C_X \delta_X.
\end{align*}
therefore since $p \in \prs{2,4}$ we have
\begin{align*}
\brs{\brs{ \vs \brk{ \N + A} - \N }}_{W^{1,2}} \leq C_X \delta_X.
\end{align*}
If $\delta_X$ is sufficiently small, as a consequence of our bootstrapping estimate of Lemma \ref{lem:bootstrap} we have
\begin{align*}
\brs{\brs{ \vs \brk{ \N + A } - \tN}}_{W^{1,2} } \leq K \ge.
\end{align*}
which is precisely the desired result.
\end{proof}
\end{prop}

\subsection{Morrey-type Inequalities}
Let $R > 0$ and $\eta \in C^{\infty}$ be a nonnegative function, where
\begin{equation}\eta(x) =
\begin{cases}
 1 & \text{ when } x \in  B_{R/2}, \\
 0 & \text{ when } x \notin B_{R}.
\end{cases}
\end{equation}
\begin{rmk}[More notational conventions] We will be using an unusual convention when working with cutoff functions in this argument.
Our notation simply notifies the reader that \emph{some power} of the cutoff is present. Ultimately this makes the proof easier to read; the choice of power of the cutoff is not necessary to the argument, but it is clear it is finite. Take
\begin{align*}
d V_{\mathring{g},\eta} := \eta^K d V_{\mathring{g}}, \text{ where } K \in \mathbb{N} \text{ is sufficiently large}.
\end{align*}
Refer to Remark \ref{rmk:notation} regarding our notation for scaling coefficients.
\end{rmk}

\begin{lemma}\label{lem:MorreyUp}
Given the assumptions of Proposition \ref{prop:LMM4.1}, \eqref{eq:sobconndif}, and \eqref{eq:N2Up} there exists $\gb >0$ such that
\begin{equation*}
\brs{\brs{ \tN \gY }}_{\mathcal{M}_{\gb}^{1,2}} \leq C \prs{\prs{\ga - 1} + \delta}.
\end{equation*}
\begin{proof} We will compute the Morrey inequalities for $\tN \gY$ and $\tN^{(2)} \gY$ separately. We also point out that due to the estimate \eqref{eq:LMM3.2} of Proposition \ref{prop:LMM3.1} combined with Lemma \ref{lem:chiest} give that
\begin{equation*}
\prs{\ga - 1} \brs{\brs{ \N \log \chi_{\la}}}_{L^2}^{\mu} + \prs{\ga - 1} \brs{\brs{ \N^{(2)} \log \chi_{\la}}}_{L^2}^{\mu} \leq C \delta \quad \mu \in \sqg{1,2}.
\end{equation*}
We show below that $\tN \gY \in \mathcal{M}_{2}^2$ and $\tN^{(2)} \gY \in \mathcal{M}_{\gb}^2$ to conclude $\tN \gY \in \mathcal{M}^{1,2}_{\gb}$ with necessary bounds.\\ \\
\noindent \fbox{$\tN \gY \in \mathcal{M}_{2}^2$} For this, we simply have the following by applying H\"{o}lder's inequality followed by applying our global estimates on $\tN^{(k)} \gY$ for $k \in \sqg{0,1,2}$,
\begin{align*}
\int_{B_R} \brs{\tN \gY}_{\mathring{g}}^2 \, \dVe
&\leq \prs{ \int_{B_R} \dVe }^{1/2} \prs{\int_{B_R} \brs{\tN \gY}_{\mathring{g}}^4 \, \dVe }^{1/2} \\
&\leq \prs{ \int_{B_R} \dVe }^{1/2} \prs{\int_{\mathbb{S}^4}  \brs{\tN \gY}_{\mathring{g}}^4 \, dV_{\mathring{g}} }^{1/2}\\ &\leq C_S \prs{ \delta + \prs{\ga - 1}} R^2.
\end{align*}
Therefore $\tN \gY \in \mathcal{M}_2^{2}$.

\noindent \fbox{$\tN^{(2)} \gY \in \mathcal{M}_{\gb}^2$} We perform a hole-filling argument. To begin,
\begin{equation}\label{eq:maineqMlocb}
\int_{\mathbb{S}^4} \brs{\tN^{(2)} \gY}_{\mathring{g}}^2 \, \dVe 
\leq \brk{C \int_{\mathbb{S}^4} \brs{\tN^{(2)} \gY}_{\mathring{g}} \brs{\tN\gY}_{\mathring{g}} \brs{\N \eta}_{\mathring{g}} \, \dVe }_{T_1} + \brk{- \int_{\mathbb{S}^4} \ip{\tN \gY,  \widetilde{\lap} \tN \gY}_{\mathring{g}} \, \dVe}_{T_2}.
\end{equation}
For the first term we have that, using a weighted Young's inequality for $\nu>0$ to be chosen and applying the local Poincar\'{e} inequality (Proposition \ref{prop:poinlocal})
\begin{align*}
T_1 &\leq  \nu \int_{\mathbb{S}^4} \brs{\tN^{(2)} \gY}_{\mathring{g}}^2 \, \dVe  + \tfrac{C}{\nu R^2} \int_{B_R \backslash B_{R/2}} \brs{\tN \gY}_{\mathring{g}}^2 \, dV_{\mathring{g}}\\
& \leq \nu \int_{\mathbb{S}^4} \brs{\tN^{(2)} \gY}_{\mathring{g}}^2 \, \dVe  + C_P \int_{B_R \backslash B_{R/2}} \brs{\tN^{(2)} \gY}_{\mathring{g}}^2 \, dV_{\mathring{g}}.
\end{align*}
We next manipulate $T_2$, commuting derivatives and applying \eqref{eq:D*Fdiff},
\begin{align*}
T_{2} &= -\int_{\mathbb{S}^4} \ip{\tN \gY,\widetilde{\lap} \tN \gY}_{\mathring{g}} \, \dVe \\
&= \brk{-  \int_{\mathbb{S}^4} \ip{ \tN_i \gY_j, \tN_k \brk{\tN_k, \tN_i} \gY_j }_{\mathring{g}} \dVe}_{T_{21}} \\
&\quad + \brk{- \int_{\mathbb{S}^4} \ip{ \tN_i \gY_j, \brk{\tN_k, \tN_i} \tN_k \gY_j }_{\mathring{g}} \dVe}_{T_{22}} +  \brk{-\int_{\mathbb{S}^4} \ip{\tN \gY,\tN \widetilde{\vartriangle} \gY}_{\mathring{g}} \, \dVe}_{T_{23}}  .
\end{align*}
Note that the estimates of $T_{21}$ and $T_{22}$ follow in suit with the manipulations of \eqref{eq:hessionest}, and thus we conclude that
\begin{align*}
T_{21} + T_{22} \leq 11 \int_{B_R} \brs{\widetilde{\N} \gY}_{\mathring{g}}^2 \, dV_{\mathring{g}, \eta}
+ 8\int_{B_R} \brs{\gY}_{\mathring{g}}^2 \, dV_{\mathring{g}, \eta}
&\leq   C_P \prs{ R^4 \int_{B_R} \brs{\tN^{(2)} \gY}_{\mathring{g}}^2 \, dV_{\mathring{g}}+ R^2 \prs{\prs{\ga - 1}  + \delta} }\\
& \leq  C \prs{\prs{\ga - 1}  + \delta} R^2.
\end{align*}
%
Now we approach $T_{23}$. Applying \eqref{eq:D*Fdiff} coming from the $\ga$-critical equation,
\begin{align*}
T_{23} &=\brk{ - \int_{\mathbb{S}^4} \ip{\tN \gY, 3 \tN \gY}_{\mathring{g}} \, \dVe - \int_{\mathbb{S}^4} \ip{\tN \gY, \brk{\widetilde{F}, \tN \gY} }_{\mathring{g}} \, \dVe }_{T_{231}} \\
& \hsp + \brk{\int_{\mathbb{S}^4} \prs{\tN \gY}^{\ast 3} \, \dVe}_{T_{232}} + \brk{\int_{\mathbb{S}^4} \tN^{(2)} \gY \ast  \tN \gY \ast \gY \, \dVe }_{T_{233}}\\
&\hsp + \brk{ - \int_{\mathbb{S}^4} \ip{\tN \gY, \tN \Theta_1 }_{\mathring{g}} \, \dVe }_{T_{234}} + \brk{- \int_{\mathbb{S}^4} \ip{\tN \gY, \tN \Theta_2 }_{\mathring{g}} \, \dVe }_{T_{235}}.
\end{align*}
For the first term, note that using \eqref{eq:Fcomm} once more,
\begin{align*}
T_{231} & \leq  \int_{\mathbb{S}^4} \brs{\tN \gY}_{\mathring{g}}^2 \, \dVe \leq C \prs{\prs{\ga - 1} + \delta} R^2.
\end{align*}
Next we have that, applying H\"{o}lder's inequality, then global Sobolev embedding $W^{1,2} \hookrightarrow L^4$, and localized Poincar\'{e} inequality and finally incorporating the global $L^2$-estimate for $\tN^{(2)} \gY$.
\begin{align*}
T_{232} &\leq \int_{\mathbb{S}^4} \brs{\tN \gY}_{\mathring{g}}^3 \, \dVe\\
 &\leq \prs{\int_{\mathbb{S}^4} \brs{\tN \gY}_{\mathring{g}}^4 \, \dVe }^{1/2} \prs{ \int_{\mathbb{S}^4} \brs{\tN\gY}_{\mathring{g}}^2 \, \dVe }^{1/2} \\
 &\leq C_S \delta \prs{\int_{\mathbb{S}^4} \brs{\tN \gY}_{\mathring{g}}^2 \, \dVe + \int_{\mathbb{S}^4} \brs{\tN \brk{\eta \tN \gY}}_{\mathring{g}}^2 \, \dVe} \\
&\leq C \delta \prs{\int_{B_R} \brs{\tN \gY}_{\mathring{g}}^2 \, dV_{\mathring{g}} + \int_{\mathbb{S}^4} \brs{\tN^{(2)}  \gY}_{\mathring{g}}^2 \, \dVe + \int_{\mathbb{S}^4} \brs{\tN \eta }_{\mathring{g}}^2 \brs{\tN \gY}^2 \, \dVe} \\
&\leq C \delta  \int_{\mathbb{S}^4} \brs{\tN^{(2)}  \gY}_{\mathring{g}}^2 \, \dVe + C \delta \prs{\int_{B_R} \brs{\tN \gY}_{\mathring{g}}^2 \, dV_{\mathring{g}}  +  \tfrac{1}{R^2} \int_{B_R \backslash B_{R/2}}  \brs{\tN \gY}_{\mathring{g}}^2 \, dV_{\mathring{g}}} \\
&\leq C \delta  \int_{\mathbb{S}^4} \brs{\tN^{(2)}  \gY}_{\mathring{g}}^2 \, \dVe + C_P \delta \prs{ R^2 \int_{B_R} \brs{\tN^{(2)} \gY}_{\mathring{g}}^2 \, dV_{\mathring{g}}  +  \int_{B_R \backslash B_{R/2}}  \brs{\tN^{(2)} \gY}_{\mathring{g}}^2 \, dV_{\mathring{g}}} \\
&\leq C \delta  \int_{\mathbb{S}^4} \brs{\tN^{(2)}  \gY}_{\mathring{g}}^2 \, \dVe + C \delta \int_{B_R \backslash B_{R/2}} \brs{\tN^{(2)} \gY}_{\mathring{g}}^2 \, dV_{\mathring{g}} + C \prs{\prs{\ga - 1} + \delta} R^2. 
\end{align*}
Next, applying a weighted Young's inequality, H\"{o}lder's inequality, applying the global $L^4$-bound on $\gY$, and Sobolev embedding $L^4 \hookrightarrow W^{1,2}$ and then applying the localized Poincar\'{e} inequality,
\begin{align*}
T_{233} &\leq C \int_{\mathbb{S}^4} \brs{\tN^{(2)} \gY}_{\mathring{g}} \brs{\tN \gY}_{\mathring{g}} \brs{\gY}_{\mathring{g}} \, \dVe  \\
& \leq  \nu \int_{\mathbb{S}^4} \brs{\tN^{(2)} \gY}_{\mathring{g}}^2 \, \dVe  + \tfrac{C}{\nu}\int_{\mathbb{S}^4} \brs{\tN \gY}_{\mathring{g}}^2 \brs{\gY}_{\mathring{g}}^2 \, \dVe\\
& \leq  \nu \int_{\mathbb{S}^4} \brs{\tN^{(2)} \gY}_{\mathring{g}}^2 \, \dVe  + C\prs{\int_{\mathbb{S}^4} \brs{\tN \gY}_{\mathring{g}}^4 \, \dVe }^{1/2}\prs{\int_{\mathbb{S}^4} \brs{\gY}_{\mathring{g}}^4 \, \dVe }^{1/2}\\
& \leq  \nu \int_{\mathbb{S}^4} \brs{\tN^{(2)} \gY}_{\mathring{g}}^2 \, \dVe  + C_S \delta \prs{\int_{\mathbb{S}^4} \brs{\tN \gY}_{\mathring{g}}^2 \, \dVe + \int_{\mathbb{S}^4} \brs{\tN \brk{ \eta \tN\gY}}_{\mathring{g}}^2 \, \dVe }\\
& \leq  \nu \int_{\mathbb{S}^4} \brs{\tN^{(2)} \gY}_{\mathring{g}}^2 \, \dVe  + C \delta  \prs{\int_{\mathbb{S}^4} \brs{\tN \gY}_{\mathring{g}}^2 \, \dVe + \int_{\mathbb{S}^4} \brs{\tN^{(2)} \gY}_{\mathring{g}}^2 \, \dVe + \tfrac{1}{R^2} \int_{B_R \backslash B/2} \brs{\tN \gY}_{\mathring{g}}^2  \, \dVe }\\
& \leq  \nu \int_{\mathbb{S}^4} \brs{\tN^{(2)} \gY}_{\mathring{g}}^2 \, \dVe  + C_S \delta  \prs{ R^2 \int_{B_R} \brs{\tN^{(2)} \gY}_{\mathring{g}}^2 \, dV_{\mathring{g}} + \int_{\mathbb{S}^4} \brs{\tN^{(2)} \gY}_{\mathring{g}}^2 \, \dVe + \int_{B_R \backslash B/2} \brs{\tN^{(2)} \gY}_{\mathring{g}}^2  \, dV_{\mathring{g}} }\\
& \leq  \nu \int_{\mathbb{S}^4} \brs{\tN^{(2)} \gY}_{\mathring{g}}^2 \, \dVe  + C_S \prs{ \delta + \prs{\ga - 1} } R^2 + C \delta  \int_{\mathbb{S}^4} \brs{\tN^{(2)} \gY}_{\mathring{g}}^2 \, \dVe  + C \int_{B_R \backslash B/2} \brs{\tN^{(2)} \gY}_{\mathring{g}}^2  \, dV_{\mathring{g}}.
\end{align*}
We now approach the first term with $\ga$ dependence,
\begin{align*}
T_{234} &= \brk{\int_{\mathbb{S}^4} \ip{ \widetilde{\lap} \gY, \Theta_1}_{\mathring{g}} \, \dVe }_{T_{2341}}+ \brk{ \int_{\mathbb{S}^4} \tN \gY \ast \Theta_1 \ast \N \eta \, \dVe}_{T_{2342}}.
\end{align*}
We compute these separately. We have that
\begin{align}
\begin{split}\label{eq:T2341}
T_{2341} &\leq \int_{\mathbb{S}^4} \brs{\tN^{(2)} \gY} \brs{\Theta_1} \, \dVe \\
&\leq  C \prs{\ga - 1} \int_{\mathbb{S}^4}  \brs{\tN^{(2)} \gY}_{\mathring{g}}^2 \, \dVe + C \prs{\ga - 1} \int_{\mathbb{S}^4}  \brs{\tN^{(2)} \gY}_{\mathring{g}} \brs{\tN \gY}_{\mathring{g}} \brs{\gY}_{\mathring{g}} \, \dVe \\
& \hsp + C  \prs{\ga - 1} \int_{\mathbb{S}^4}  \brs{\tN^{(2)} \gY}_{\mathring{g}}  \brs{\gY}_{\mathring{g}} \, \dVe + C  \prs{\ga - 1} \int_{\mathbb{S}^4}  \brs{\tN^{(2)} \gY}_{\mathring{g}}  \brs{\gY}_{\mathring{g}}^3 \, \dVe.
\end{split}
\end{align}
The first term can be absorbed, and the second is precisely $T_{233}$. For the third term of \eqref{eq:T2341} we apply a weighted Young's inequality, localized Poincar\'{e} inequality twice, and the global estimate to $\tN^{(2)} \gY$.
\begin{align*}
\int_{\mathbb{S}^4}  \brs{\tN^{(2)} \gY}_{\mathring{g}}  \brs{\gY} \, \dVe &\leq \nu \int_{\mathbb{S}^4} \brs{\tN^{(2)} \gY}_{\mathring{g}}^2 \, \dVe + \tfrac{C}{\nu} \int_{\mathbb{S}^4} \brs{\gY}_{\mathring{g}}^2 \, \dVe \\
&\leq \nu \int_{\mathbb{S}^4} \brs{\tN^{(2)} \gY}_{\mathring{g}}^2 \, \dVe + C\int_{B_R} \brs{\gY}_{\mathring{g}}^2 \, dV_{\mathring{g}}\\
&\leq \nu \int_{\mathbb{S}^4} \brs{\tN^{(2)} \gY}_{\mathring{g}}^2 \, \dVe + C_P \prs{ R^4 \int_{B_R} \brs{\tN^{(2)} \gY}_{\mathring{g}}^2 \, dV_{\mathring{g}}+ R^2 \delta } \\
&\leq \nu \int_{\mathbb{S}^4} \brs{\tN^{(2)} \gY}_{\mathring{g}}^2 \, \dVe +  C R^2 \prs{(\ga - 1) + \delta}.
\end{align*}
For the last integral in \eqref{eq:T2341} we have
\begin{align*}
&\int_{\mathbb{S}^4} \brs{\tN^{(2)} \gY}_{\mathring{g}} \brs{\gY}^3_{\mathring{g}} \, \dVe \\
&\leq \nu \int_{\mathbb{S}^4} \brs{\tN^{(2)} \gY}_{\mathring{g}}^2 \, \dVe + \tfrac{C}{\nu} \prs{\int_{\mathbb{S}^4} \brs{\gY}_{\mathring{g}}^6 \, \dVe}^{1/3}\prs{\int_{\mathbb{S}^4} \brs{\gY}_{\mathring{g}}^6 \, \dVe}^{2/3} \\
&\leq \nu \int_{\mathbb{S}^4} \brs{\tN^{(2)} \gY}_{\mathring{g}}^2 \, \dVe + C \prs{\prs{\ga - 1} + \delta }\prs{\int_{\mathbb{S}^4} \brs{\gY}_{\mathring{g}}^6 \, \dVe}^{1/3} \\
&\leq \nu \int_{\mathbb{S}^4} \brs{\tN^{(2)} \gY}_{\mathring{g}}^2 \, \dVe + C_S \prs{\prs{\ga - 1} + \delta } \prs{ \int_{\mathbb{S}^4} \brs{\gY}_{\mathring{g}}^2 \, \dVe  + \int_{\mathbb{S}^4} \brs{\tN \brk{ \eta \gY}}_{\mathring{g}}^2 \, \dVe + \int_{\mathbb{S}^4} \brs{\tN^{(2)} \brk{ \eta \gY}}_{\mathring{g}}^2 \, \dVe} \\
&\leq \nu \int_{\mathbb{S}^4} \brs{\tN^{(2)} \gY}_{\mathring{g}}^2 \, \dVe + C \prs{\prs{\ga - 1} + \delta } \prs{ \int_{\mathbb{S}^4} \brs{\gY}_{\mathring{g}}^2 \, \dVe  + \int_{\mathbb{S}^4} \brs{\tN \gY}_{\mathring{g}}^2 \, \dVe + \int_{\mathbb{S}^4} \brs{\gY}_{\mathring{g}}^2 \brs{ \N \eta }_{\mathring{g}}^2 \, \dVe} \\
&\hsp + C \prs{\prs{\ga - 1} + \gd} \prs{ \int_{\mathbb{S}^4} \brs{\tN^{(2)} \gY}_{\mathring{g}}^2 \, \dVe +  \int_{\mathbb{S}^4} \brs{\tN \gY}_{\mathring{g}}^2 \brs{ \N \eta}_{\mathring{g}}^2 \, \dVe +  \int_{\mathbb{S}^4} \brs{\gY}_{\mathring{g}}^2 \brs{ \N^{(2)} \eta}_{\mathring{g}}^2 \, \dVe } \\
&\leq \nu \int_{\mathbb{S}^4} \brs{\tN^{(2)} \gY}_{\mathring{g}}^2 \, \dVe + C \prs{\prs{\ga - 1} + \delta } \prs{ \int_{B_R} \brs{\gY}_{\mathring{g}}^2 \, dV_{\mathring{g}} + \int_{B_R} \brs{\tN \gY}_{\mathring{g}}^2 \, dV_{\mathring{g}} + \tfrac{1}{R^2}\int_{B_R \backslash B_{R/2}} \brs{\gY}_{\mathring{g}}^2 \, dV_{\mathring{g}}} \\
&\hsp + C \prs{\prs{\ga - 1} + \gd} \prs{ \int_{\mathbb{S}^4} \brs{\tN^{(2)} \gY}_{\mathring{g}}^2 \, \dVe +  \tfrac{1}{R^2}\int_{B_R \backslash B_{R/2}} \brs{\tN \gY}_{\mathring{g}}^2  \, dV_{\mathring{g}} +  \tfrac{1}{R^4}\int_{B_R \backslash B_{R/2}} \brs{\gY}_{\mathring{g}}^2 \, dV_{\mathring{g}} } \\
&\leq \nu \int_{\mathbb{S}^4} \brs{\tN^{(2)} \gY}_{\mathring{g}}^2 \, \dVe + C_P \prs{\prs{\ga - 1} + \gd} \prs{ \int_{\mathbb{S}^4} \brs{\tN^{(2)} \gY}_{\mathring{g}}^2 \, \dVe + \int_{B_R \backslash B_{R/2}} \brs{\tN^{(2)} \gY}_{\mathring{g}}^2  \, dV_{\mathring{g}}}.
\end{align*}
We now address the next term
\begin{align}
\begin{split}\label{eq:T2342}
T_{2342} &= \int_{\mathbb{S}^4} \tN \gY \ast \Theta_1 \ast \N \eta \, \dVe \\
& \leq C \prs{\ga - 1} \int_{\mathbb{S}^4} \brs{\tN^{(2)} \gY}_{\mathring{g}} \brs{\tN \gY}_{\mathring{g}} \brs{\N \eta}_{\mathring{g}} \, \dVe
 + C \prs{\ga - 1} \int_{\mathbb{S}^4} \brs{\tN \gY}_{\mathring{g}}^2 \brs{\gY}_{\mathring{g}} \brs{\N \eta}_{\mathring{g}} \, \dVe \\
&\hsp   + C \prs{\ga - 1} \int_{\mathbb{S}^4} \brs{\tN \gY}_{\mathring{g}} \brs{\gY}_{\mathring{g}} \brs{\N \eta}_{\mathring{g}} \, \dVe  +C \prs{\ga - 1} \int_{\mathbb{S}^4} \brs{\tN \gY}_{\mathring{g}}  \brs{\gY}^3_{\mathring{g}} \brs{\N \eta}_{\mathring{g}} \, \dVe.
\end{split}
\end{align}
The first term is exactly $T_1$, so we can apply the same estimate. Now we approach the second term of \eqref{eq:T2342}, applying Young's inequality,  H\"{o}lder's inequality, a global Sobolev embedding of $W^{1,2} \hookrightarrow L^4$, and localized Poincar\'{e} inequalities, and global $L^2$ control of $\tN^{(2)} \gY$,
\begin{align*}
\int_{\mathbb{S}^4} \brs{\tN \gY}_{\mathring{g}}^2 \brs{\gY}_{\mathring{g}} \brs{\N \eta}_{\mathring{g}} \, \dVe & \leq C \int_{\mathbb{S}^4} \brs{\tN \gY}_{\mathring{g}}^4 \, \dVe + \tfrac{C}{R^2} \int_{B_R \backslash B_{R/2}} \brs{\gY}_{\mathring{g}}^2 \, dV_{\mathring{g}} \\
&\leq  C \prs{\delta + \prs{\ga - 1}} \prs{\int_{\mathbb{S}^4} \brs{\tN \gY}_{\mathring{g}}^4 \, \dVe}^{1/2}+ R^2 C_P \int_{B_R \backslash B_{R/2}} \brs{\tN^{(2)}\gY}_{\mathring{g}}^2 \, dV_{\mathring{g}} \\
& \leq C_S \prs{\prs{\ga - 1} + \delta} \prs{ \int_{\mathbb{S}^4} \brs{\tN \gY}_{\mathring{g}}^2 \, \dVe + \int_{\mathbb{S}^4} \brs{\tN \brk{\eta \tN \gY}}_{\mathring{g}}^2 \, \dVe  } + C R^2 \prs{\prs{\ga - 1} + \delta}\\
&\leq C \int_{B_R \backslash B_{R/2}} \brs{\tN^{(2)} \gY}_{\mathring{g}}^2 \, dV_{\mathring{g}} + C_P \prs{\prs{\ga - 1} + \delta} \int_{\mathbb{S}^4} \brs{\tN^{(2)} \gY}_{\mathring{g}}^2 \, \dVe + C R^2 \prs{\prs{\ga - 1} + \delta}.
\end{align*}
For the third term, applying weighted Young's Inequality in preparation for an application of Poincar\'{e} inequality, then a Holder's inequality followed by applying global $L^4$ control of $\gY$,
\begin{align*}
\int_{\mathbb{S}^4} \brs{\tN \gY}_{\mathring{g}} \brs{\gY}_{\mathring{g}} \brs{\N \eta}_{\mathring{g}} \, \dVe &\leq \tfrac{\nu}{C_P R^2}\int_{B_R \backslash B_{R/2}} \brs{\tN \gY}_{\mathring{g}}^2  \, dV_{\mathring{g}} +  \tfrac{C C_P}{\nu} \int_{B_R \backslash B_{R/2}} \brs{\gY}_{\mathring{g}}^2  \, dV_{\mathring{g}} \\
 &\leq \nu \int_{B_R \backslash B_{R/2}} \brs{\tN^{(2)} \gY}_{\mathring{g}}^2  \, dV_{\mathring{g}} + C\prs{ \int_{\mathbb{S}^4} \, dV_{\mathring{g}} }^{1/2} \prs{ \int_{B_R \backslash B_{R/2}} \brs{\gY}_{\mathring{g}}^4 \, \dVe }^{1/2}\\
  &\leq \nu \int_{B_R \backslash B_{R/2}} \brs{\tN^{(2)} \gY}_{\mathring{g}}^2  \, dV_{\mathring{g}} + C \delta R^2.
\end{align*}
For the fourth integral we apply weighted Young's inequality followed by localized Poincar\'{e} inequality and  H\"{o}lder's inequality, using the global Sobolev embedding $ W^{2,2}\hookrightarrow L^8$ and once more applying localized Poincar\'{e},
\begin{align*}
\int_{\mathbb{S}^4} \brs{\tN \gY}_{\mathring{g}} \brs{\gY}_{\mathring{g}}^3 \brs{\N \eta}_{\mathring{g}} \, \dVe & \leq \tfrac{1}{R^2} \int_{B_R \backslash B_{R/2}} \brs{\gY}_{\mathring{g}}^2 \, dV_{\mathring{g}} +  \int_{\mathbb{S}^4} \brs{\tN \gY}_{\mathring{g}}^2 \brs{\gY}_{\mathring{g}}^4 \, \dVe \\
&\leq C_P R^2 \int_{B_R \backslash B_{R/2}} \brs{\tN^{(2)} \gY}_{\mathring{g}}^2 \, dV_{\mathring{g}} + C\prs{\int_{\mathbb{S}^4} \brs{\tN \gY}_{\mathring{g}}^4 \, \dVe }^{1/2} \prs{\int_{\mathbb{S}^4} \brs{ \gY }_{\mathring{g}}^8 \, \dVe }^{1/2} \\
&\leq  C \int_{B_R \backslash B_{R/2}} \brs{\tN^{(2)} \gY}_{\mathring{g}}^2 \, dV_{\mathring{g}}  + C_S \delta \prs{\int_{\mathbb{S}^4} \brs{\tN \gY}_{\mathring{g}}^2 \, \dVe + \int_{\mathbb{S}^4} \brs{\tN \brk{ \eta \tN \gY}}_{\mathring{g}}^2 \, \dVe}\\
&\leq  C_P \delta \prs{\int_{\mathbb{S}^4} \brs{\tN^{(2)} \gY}_{\mathring{g}}^2 \, \dVe + \int_{B_R \backslash B_{R/2}} \brs{\tN^{(2)} \gY}_{\mathring{g}}^2 \, dV_{\mathring{g}}} + C R^2 \prs{ \prs{\ga - 1} + \delta}. 
\end{align*}
Next we have that
\begin{align*}
T_{235} &= \brk{\int_{\mathbb{S}^4} \ip{ \lap \gY, \Theta_2}_{\mathring{g}} \, \dVe}_{T_{2351}} + \brk{\int_{\mathbb{S}^4} \N \gY \ast \Theta_2 \ast \N \eta \, \dVe}_{T_{2352}}.
\end{align*}
Then we expand out
\begin{align}
\begin{split}\label{eq:T2351}
T_{2351} &\leq C(\ga-1) \int_{\mathbb{S}^4} \brs{\tN^{(2)} \gY}_{\mathring{g}} \brs{\Theta_2}_{\mathring{g}} \, \dVe \\
&\leq C(\ga-1)\int_{\mathbb{S}^4} \brs{\tN^{(2)} \gY}_{\mathring{g}}\brs{\tN \gY}_{\mathring{g}}  \brs{\N \log \chi_{\la}}_{\mathring{g}}  \, \dVe + C(\ga-1)\int_{\mathbb{S}^4} \brs{\tN^{(2)} \gY}_{\mathring{g}} \brs{\N \log \chi_{\la}}_{\mathring{g}} \, \dVe \\
&\hsp + C(\ga-1)\int_{\mathbb{S}^4} \brs{\tN^{(2)} \gY}_{\mathring{g}} \brs{ \gY}^2  \brs{\N \log \chi_{\la}}_{\mathring{g}} \, \dVe.
\end{split}
\end{align}
For the first term of \eqref{eq:T2351}, applying a weighted Young's inequality and H\"{o}lder's inequality, followed by global Sobolev embedding $W^{1,2} \hookrightarrow L^4$,
\begin{align*}
C\prs{\ga - 1}&\int_{\mathbb{S}^4} \brs{\tN^{(2)} \gY}_{\mathring{g}}  \brs{\tN \gY} \brs{\N \log \chi_{\la}}_{\mathring{g}} \, \dVe \\
&\hsp \leq  \prs{\ga - 1} \nu \int_{\mathbb{S}^4} \brs{\tN^{(2)} \gY}_{\mathring{g}}^2 \, \dVe + \tfrac{C \prs{\ga - 1}}{\nu}\int_{\mathbb{S}^4} \brs{\N \log \chi_{\la}}_{\mathring{g}}^2 \brs{\tN \gY}_{\mathring{g}}^2 \, \dVe\\
&\hsp  \leq \prs{\ga - 1} \nu \int_{\mathbb{S}^4} \brs{\tN^{(2)} \gY}_{\mathring{g}}^2 \, \dVe + C \prs{\ga - 1} \prs{\int_{\mathbb{S}^4} \brs{\N \log \chi_{\la}}_{\mathring{g}}^4 \, \dVe }^{1/2} \prs{\int_{\mathbb{S}^4} \brs{\tN \gY}_{\mathring{g}}^4 \, \dVe }^{1/2}\\
&\hsp  \leq  \prs{\ga - 1} \nu \int_{\mathbb{S}^4} \brs{\tN^{(2)} \gY}_{\mathring{g}}^2 \, \dVe + C_S 
\delta \prs{\int_{\mathbb{S}^4} \brs{\tN \gY}_{\mathring{g}}^2 \, \dVe + \int_{\mathbb{S}^4} \brs{\tN \brk{ \eta \tN \gY}}_{\mathring{g}}^2 \, \dVe }\\
&\hsp  \leq  \nu \int_{\mathbb{S}^4} \brs{\tN^{(2)} \gY}_{\mathring{g}}^2 \, \dVe +C_P \delta \prs{ \int_{\mathbb{S}^4} \brs{\tN^{(2)} \gY}_{\mathring{g}}^2 \, \dVe + \int_{B_R \backslash B_{R/2}} \brs{\tN^{(2)} \gY}_{\mathring{g}}^2 \, dV_{\mathring{g}} + R^2 \prs{\ga-1 + \delta }}.
\end{align*}
Next we have that, for the second term of \eqref{eq:T2351}, using weighted Young's inequality followed by Holder's inequality and applying the global bounds to $\log \chi_{\la}$ type terms,
\begin{align*}
C\prs{\ga - 1}\int_{\mathbb{S}^4} &\brs{\tN^{(2)} \gY}_{\mathring{g}} \brs{\N \log \chi_{\la}}_{\mathring{g}} \, \dVe \\
&\leq \prs{\ga - 1} \nu \int_{\mathbb{S}^4} \brs{\tN^{(2)} \gY}_{\mathring{g}}^2 \, \dVe + \tfrac{C}{\nu}\prs{\ga - 1} \int_{\mathbb{S}^4} \brs{\N \log \chi_{\la}}_{\mathring{g}}^2 \, \dVe \\
&\leq \prs{\ga - 1} \nu \int_{\mathbb{S}^4} \brs{\tN^{(2)} \gY}_{\mathring{g}}^2 \, \dVe + C \prs{\ga - 1} \prs{ \int_{\mathbb{S}^4}\brs{\N \log \chi_{\la}}_{\mathring{g}}^4 \, \dVe }^{1/2}\prs{ \int_{\mathbb{S}^4}  \, \dVe}^{1/2}\\
&\leq \nu \int_{\mathbb{S}^4} \brs{\tN^{(2)} \gY}_{\mathring{g}}^2 \, \dVe + C R^2 \delta.
\end{align*}
Lastly for the third term of \eqref{eq:T2351}, we apply a weighted Young's inequality, H\"{o}lder's inequality, a global Sobolev embedding and then localized Poincar\'{e} to the remaining pieces,
\begin{align*}
C\prs{\ga - 1}\int_{\mathbb{S}^4} &\brs{\tN^{(2)} \gY}_{\mathring{g}} \brs{\gY}^2 \brs{\N \log \chi_{\la}}_{\mathring{g}} \, \dVe \\ &\leq \prs{\ga - 1} \nu \int_{\mathbb{S}^4} \brs{\tN^{(2)} \gY}_{\mathring{g}}^2 \, \dVe + \prs{\ga - 1} \tfrac{C}{\nu}\int_{\mathbb{S}^4} \brs{\gY}_{\mathring{g}}^4 \brs{ \N \log \chi_{\la} }_{\mathring{g}}^2 \, \dVe \\
&\leq \prs{\ga - 1} \nu \int_{\mathbb{S}^4} \brs{\tN^{(2)} \gY}_{\mathring{g}}^2 \, \dVe  + \prs{\ga - 1} C \prs{ \int_{\mathbb{S}^4} \brs{\gY}_{\mathring{g}}^8  \, \dVe}^{1/2}\prs{ \int_{\mathbb{S}^4} \brs{\N \log \chi_{\la}}_{\mathring{g}}^4  \, \dVe}^{1/2} \\
&\leq \prs{\ga - 1} \nu \int_{\mathbb{S}^4} \brs{\tN^{(2)} \gY}_{\mathring{g}}^2 \, \dVe  + C_S  \delta \prs{ \int_{\mathbb{S}^4} \brs{\gY}_{\mathring{g}}^2  \, \dVe + \int_{\mathbb{S}^4} \brs{\tN \brk{\eta \gY}}_{\mathring{g}}^2  \, \dVe + \int_{\mathbb{S}^4} \brs{\tN^{(2)} \brk{\eta \gY}}_{\mathring{g}}^2  \, \dVe}^2\\
&\leq \nu \int_{\mathbb{S}^4} \brs{\tN^{(2)} \gY}_{\mathring{g}}^2 \, \dVe  + C_P \delta^2 \prs{\int_{B_R \backslash B_{R/2}} \brs{\tN^{(2)} \gY}_{\mathring{g}}^2  \, dV_{\mathring{g}} + \int_{\mathbb{S}^4} \brs{\tN^{(2)} \gY}_{\mathring{g}}^2  \, \dVe + R^2 \prs{\ga - 1 + \delta}}.
\end{align*}
We next expand out
\begin{align}
\begin{split}\label{eq:T2352}
T_{2352} &\leq C  \int_{\mathbb{S}^4}  \brs{\tN \gY}_{\mathring{g}} \brs{\Theta_2}_{\mathring{g}} \brs{\N \eta}_{\mathring{g}} \, \dVe \\
&\leq C \prs{\ga-1}\int_{\mathbb{S}^4}  \brs{\tN \gY}_{\mathring{g}} \brs{\N \log \chi_{\la}}_{\mathring{g}} \brs{\N \eta}_{\mathring{g}} \, \dVe + C \prs{\ga-1}\int_{\mathbb{S}^4}  \brs{\tN \gY}_{\mathring{g}}^2 \brs{\N \log \chi_{\la}}_{\mathring{g}} \brs{\N \eta}_{\mathring{g}} \, \dVe\\
&\hsp + C \prs{\ga-1}\int_{\mathbb{S}^4}  \brs{\tN \gY}_{\mathring{g}} \brs{\gY}_{\mathring{g}}^2 \brs{\N \log \chi_{\la}}_{\mathring{g}} \brs{\N \eta}_{\mathring{g}} \, \dVe.
\end{split}
\end{align}
For the first term of \eqref{eq:T2352}, applying a weighted Young's inequality in preparation for a Poincar\'{e} inequality, then H\"{o}lder's inequality and applying the global bounds of $\log \chi_{\la}$ type terms,
\begin{align*}
C \prs{\ga-1}&\int_{\mathbb{S}^4}  \brs{\tN \gY}_{\mathring{g}} \brs{\N \log \chi_{\la}}_{\mathring{g}} \brs{\N \eta}_{\mathring{g}} \, \dVe \\
& \leq  \tfrac{\nu}{C_P R^2} \int_{B_R \backslash B_{R/2}} \brs{\tN \gY}_{\mathring{g}}^2  \, dV_{\mathring{g}} + \prs{\ga-1} \tfrac{C C_P}{\nu} \int_{\mathbb{S}^4} \brs{\N \log \chi_{\la}}_{\mathring{g}}^2 \, dV_{\mathring{g}} \\
 & \leq  \nu \int_{B_R \backslash B_{R/2}} \brs{\tN^{(2)} \gY}_{\mathring{g}}^2  \, dV_{\mathring{g}} + C\prs{\ga-1} \prs{ \int_{\mathbb{S}^4} \, \dVe}^{1/2} \prs{\int_{\mathbb{S}^4} \brs{\N \log \chi_{\la}}_{\mathring{g}}^4 \, \dVe}^{1/2} \\
  & \leq  \nu \int_{B_R \backslash B_{R/2}} \brs{\tN^{(2)} \gY}_{\mathring{g}}^2  \, dV_{\mathring{g}} + C_{S} \delta R^2.
\end{align*}
For the second term of \eqref{eq:T2352} we apply H\"{o}lder's inequality twice followed by global $L^4$ bounds of $\N \log \chi_{\la}$, a global Sobolev embedding $W^{1,2} \hookrightarrow L^4$ and the localized Poincar\'{e} inequality
\begin{align*}
C \prs{\ga-1}&\int_{\mathbb{S}^4} \brs{\tN \gY}_{\mathring{g}}^2 \brs{\N \log \chi_{\la}}_{\mathring{g}} \brs{\N \eta}_{\mathring{g}} \, \dVe \\
&\leq \tfrac{C}{R} \prs{\ga-1} \prs{ \int_{\mathbb{S}^4} \brs{\tN \gY}_{\mathring{g}}^4 \, \dVe}^{1/2} \prs{ \int_{\mathbb{S}^4} \brs{\N \log \chi_{\la}}_{\mathring{g}}^2 \, \dVe}^{1/2}\\
&\leq \tfrac{C}{R} \prs{\ga-1} \prs{ \int_{\mathbb{S}^4} \brs{\tN \gY}_{\mathring{g}}^4 \,\dVe}^{1/2} \prs{\int_{B_R} \, dV_{\mathring{g}} }^{1/4}\prs{ \int_{\mathbb{S}^4} \brs{\N \log \chi_{\la}}_{\mathring{g}}^4 \, \dVe}^{1/4}\\
&\leq C_S \delta \prs{ \int_{\mathbb{S}^4} \brs{\tN \gY}_{\mathring{g}}^2 \, \dVe +\int_{\mathbb{S}^4} \brs{\tN\brk{ \eta \gY}}_{\mathring{g}}^2 \, \dVe  }\\
&\leq C_P \prs{ \delta + \prs{\ga - 1}} R^2 + C_P \delta \int_{B_R \backslash B_{R/2}} \brs{\tN^{(2)} \gY}_{\mathring{g}}^2 \, dV_{\mathring{g}}.
\end{align*}
For the third term of \eqref{eq:T2352}, applying weighted Young's inequality (`preparing' for the application of the Poincar\'{e} inequality with our choice of weight), then Poincar\'{e} inequality and H\"{o}lder's inequality, then applying Sobolev embedding $W^{2,2} \hookrightarrow L^8$ and Poincar\'{e} inequalities once more
\begin{align*}
 C \prs{\ga - 1}\int_{\mathbb{S}^4} &\brs{\tN \gY}_{\mathring{g}} \brs{\gY}_{\mathring{g}}^2 \brs{\N \log \chi_{\la}}_{\mathring{g}} \brs{\N \eta}_{\mathring{g}}\, \dVe\\
&\leq \tfrac{\nu}{C_P R^2} \int_{B_R \backslash B_{R/2}} \brs{\tN \gY}_{\mathring{g}}^2 \, \dVe + \tfrac{C \prs{\ga - 1} C_P}{\nu}\int_{\mathbb{S}^4} \brs{\gY}_{\mathring{g}}^4 \brs{ \N \log \chi_{\la} }_{\mathring{g}}^2 \, \dVe \\
&\leq \nu \int_{\mathbb{S}^4} \brs{\tN^{(2)} \gY}_{\mathring{g}}^2 \, \dVe  + C \prs{\ga - 1}\prs{ \int_{\mathbb{S}^4} \brs{\gY}_{\mathring{g}}^8  \, \dVe}^{1/2}\prs{ \int_{\mathbb{S}^4} \brs{\N \log \chi_{\la}}_{\mathring{g}}^4  \, \dVe}^{1/2} \\
&\leq \nu \int_{\mathbb{S}^4} \brs{\tN^{(2)} \gY}_{\mathring{g}}^2 \, \dVe  + C_S \delta \prs{ \int_{\mathbb{S}^4} \brs{\gY}_{\mathring{g}}^2  \, \dVe + \int_{\mathbb{S}^4} \brs{\tN \brk{\eta \gY}}_{\mathring{g}}^2  \, \dVe + \int_{\mathbb{S}^4} \brs{\tN^{(2)} \brk{\eta \gY}}_{\mathring{g}}^2  \, \dVe}^2\\
&\leq \nu \int_{\mathbb{S}^4} \brs{\tN^{(2)} \gY}_{\mathring{g}}^2 \, \dVe  + C_P \delta \prs{\int_{B_R \backslash B_{R/2}} \brs{\tN^{(2)} \gY}_{\mathring{g}}^2  \, dV_{\mathring{g}} + \int_{\mathbb{S}^4} \brs{\tN^{(2)} \gY}_{\mathring{g}}^2  \, \dVe + R^2 \prs{\prs{\ga - 1} + \gd}}.
\end{align*}
Take $\nu \leq \tfrac{1}{10}$ so that the terms scaled by $\nu$ may be absorbed into the left hand side of \eqref{eq:maineqMlocb}. Furthermore, choose $\delta$, $\ga$ sufficiently small so that remaining terms may be absorbed over,
\begin{equation}\label{eq:holefilling}
\int_{\frac{R}{2}} \brs{\tN^{(2)} \gY}_{\mathring{g}}^2 \, dV_{\mathring{g}} \leq C  \prs{\int_{B_R \backslash B_{\frac{R}{2}}} \brs{\tN^{(2)} \gY}_{\mathring{g}}^2 \, dV_{\mathring{g}} + R^2 \prs{\prs{\ga - 1} + \gd}}.
\end{equation}
Define
\begin{align*}
f (R) := \int_{B_R \backslash B_{\frac{R}{2}}} \brs{\tN^{(2)} \gY}_{\mathring{g}}^2 \, dV_{\mathring{g}} + R^2 \prs{\prs{\ga - 1} + \gd}.
\end{align*}
Remanipulating \eqref{eq:holefilling} above, we have that $f \prs{\tfrac{R}{2}}\leq \tfrac{C}{C + 1} f(R)$, and therefore for all $k\geq 1$,
\begin{align*}
f\prs{\tfrac{R}{2^k}}\leq \prs{\tfrac{C}{C+1}}^k \prs{ \int_{B_R} \brs{\tN \gY}_{\mathring{g}}^2 \, dV_{\mathring{g}} + R^2 \prs{\prs{\ga - 1} + \gd }}\leq \prs{\tfrac{C}{C+1}}^k C \prs{\prs{\ga-1} + \delta}.
\end{align*}
This implies that there exists some $\gb > 0$ such that
\begin{align*}
\int_{B_R} \brs{\tN \gY}_{\mathring{g}}^2 \, dV_{\mathring{g}} \leq C \prs{\ga - 1 + \delta} R^{2 \gb},
\end{align*}
which yields the desired Morrey bound.
\end{proof}
\end{lemma}
%
%
%
\bibliography{sources.bib}{}

\begin{thebibliography}{ADHM78}

\bibitem[ADHM78]{ADHM}
Michael Atiyah, Vladimir Drinfeld, Nigel Hitchin, and Yuri~I. Manin.
\newblock Construction of instantons.
\newblock {\em Phys. Lett.}, 65:185--187, 1978.

\bibitem[AS53]{AS}
Warren Ambrose and Isadore Singer.
\newblock A theorem on holonomy.
\newblock {\em Trans. Amer. Math. Soc.}, 75(3):428--443, 1953.

\bibitem[BL81]{BL}
Jean-Pierre Bourguignon and H.~Blaine Lawson.
\newblock Stability and isolation phenomena for {Yang-Mills} fields.
\newblock {\em Comm. Math. Phys.}, 79(2):189--230, 1981.

\bibitem[Bou06a]{Bourguignon}
Jean-Pierre Bourguignon.
\newblock Harmonic curvature for gravitational and {Yang-Mills} fields.
\newblock In {\em Harmonic Maps}, volume 949 of {\em Lecture Notes in
  Mathematics}, pages 35--47. Springer, August 2006.

\bibitem[Bou06b]{Bourguignon1}
Jean-Pierre Bourguignon.
\newblock Yang-mills theory: The differential geometric side.
\newblock {\em Lecture Notes in Mathematics}, 1263:13--54, September 2006.

\bibitem[BPST75]{BPST}
Alexander Belavin, Alexander Polyakov, Albert Schwarz, and Yu.~S. Tyupkin.
\newblock Pseudoparticle solutions of the yang-mills equations.
\newblock {\em Phys. Lett.}, 59B:85--87, 1975.

\bibitem[DK90]{DK}
Simon Donaldson and Peter Kronheimer.
\newblock {\em The geometry of four-manifolds}.
\newblock Oxford Mathematical Monographs. Oxford Science Publications, 1990.

\bibitem[Gia83]{Giaquinta}
Mariano Giaquinta.
\newblock {\em Multiple Integrals in the Calculus of Variations and Nonlinear
  Elliptic Systems}.
\newblock Annals of Mathematics Studies. Princeton University Press, 1983.

\bibitem[GM12]{GM}
Mariano Giaquinta and Luca Martinazzi.
\newblock {\em An Introduction to the Regularity Theory For Elliptic Systems,
  Harmonic Maps and Minimal Graphs}, volume~11 of {\em Lecture Notes. Scuola
  Normale Superiore di Pisa (New Series)}.
\newblock Appunti. Scuola Normale Superiore di Pisa (Nuova Serie), 2 edition,
  2012.

\bibitem[HS13]{HS}
Min-Chun Hong and Lorenz Schabrun.
\newblock The energy identity for a sequence of {Yang-Mills}
  $\alpha$-connections.
\newblock {\em arXiv}, August 2013.

\bibitem[HTY15]{HTY}
Min-Chun Hong, Gang Tian, and Hao Yin.
\newblock The {Yang-Mills} $\alpha$-flow in vector bundles over four manifold
  and its applications.
\newblock {\em Commentarii Mathematici Helvetici}, (90):75--120, 2015.

\bibitem[Iso08]{Isobe}
Takeshi Isobe.
\newblock A regularity result for a class of degenerate {Yang-Mills}
  connections in critical dimensions.
\newblock {\em Forum Math.}, 20:1109--1139, 2008.

\bibitem[LMM15]{LMM}
Tobias Lamm, Andrea Malchiodi, and Mario Micallef.
\newblock Limits of $\alpha$-harmonic maps.
\newblock {\em arXiv}, 08 2015.

\bibitem[Nab10]{Naber}
Gregory Naber.
\newblock {\em Topology, Geometry and Gauge Fields: Foundations}.
\newblock Number~25 in Texts in Applied Mathematics. Springer, 2 edition, 2010.

\bibitem[Slo92]{Slovak}
Jan Slovak.
\newblock Natural operators on conformal manifolds.
\newblock {\em Proceedings of the Conference on Differential Geometry and its
  Applications}, pages 335--349, 1992.

\bibitem[TT04]{TT}
Terence Tao and Gang Tian.
\newblock A singularity removal theorem for {Yang-Mills} fields in higher
  dimensions.
\newblock {\em Journal of the American Mathematical Society}, 17(3):557--593,
  2004.

\bibitem[Uhl82a]{Uhlenbeck}
Karen Uhlenbeck.
\newblock Connections with {$L^p$}-bounds on curvature.
\newblock {\em Comm. Math. Phys.}, 83:31--42, 1982.

\bibitem[Uhl82b]{Uhlenbeck1}
Karen Uhlenbeck.
\newblock Removable singularities in {Yang-Mills} fields.
\newblock {\em Comm. Math. Phys.}, 83:11--29, 1982.

\end{thebibliography}
\bibliographystyle{alpha}
\end{document}